\documentclass[reqno,11pt]{amsart}
\usepackage{amsmath, amsfonts}

\usepackage{color}
\usepackage{amsmath}
\usepackage{amsfonts}
\usepackage{amssymb}
\usepackage{amsthm}
\usepackage{newlfont}
\usepackage{graphicx}
\usepackage{hyperref}

\newtheorem{thm}{Theorem}[section]

\newtheorem{lem}[thm]{Lemma}
\newtheorem{prop}[thm]{Proposition}
\newtheorem{conj}[thm]{Conjecture}
\theoremstyle{definition}
\newtheorem{defn}[thm]{Definition}
\theoremstyle{remark}
\newtheorem{rem}[thm]{Remark}

\numberwithin{equation}{section}

\newcommand{\eps}{\varepsilon}

\newcommand{\lsm}{\lesssim}

\newcommand{\C}{{\mathbb{C}}}
\newcommand{\R}{{\mathbb{R}}}

\newcommand{\rrf}{\R\times\R^4}
\newcommand{\srf}{[0,\infty)\times\R^4}
\newcommand{\srt}{[0,\infty)\times\R^2}
\newcommand{\srs}{[0,\infty)\times\R^3}
\newcommand{\srd}{[0,\infty)\times\R^d}
\newcommand{\ltrf}{L^2_x(\R^4)}
\newcommand{\ltrd}{L^2_x (\R^d)}
\newcommand{\propgto}{e^{-i\tau\Delta}}
\newcommand{\propt}{e^{it\Delta}}

\newcommand{\pn}{P_N}
\newcommand{\pntd}{\tilde P_N}
\newcommand{\pnf}{P_N^{12}}

\newcommand{\pnp}{P_N^{12+}}
\newcommand{\pnn}{P_N^{12-}}
\newcommand{\pns}{P_{\le N}^{34}}
\newcommand{\xf}{x^{12}}
\newcommand{\yf}{y^{12}}

\newcommand{\xs}{x^{34}}
\newcommand{\xft}{x_1,x_2}
\newcommand{\xst}{x_3,x_4}
\newcommand{\phr}{\phi_{>R}}
\newcommand{\phrt}{\phi_{>\frac R2}}
\newcommand{\midd}{\phi_{\frac 1R<|\xf|\le R^c}}
\newcommand{\frc}{\phi_{>\frac 12 R^c}}
\newcommand{\frcs}{\phi_{\le \frac 12R^c}}
\newcommand{\bnto}{\phi_{>\frac{N\tau}{200}}}
\newcommand{\snto}{\phi_{\le\frac{N\tau}{200}}}
\newcommand{\otg}{P_N^+}
\newcommand{\icm}{P_N^-}
\newcommand{\kmid}{\phi_{\frac 1{N^3}<|\xf|\le \frac 1N}}
\newcommand{\pd}{P_N^{d_1}}
\newcommand{\xd}{x^{d_1}}
\newcommand{\yd}{y^{d_1}}

\newcounter{smalllist}
\newenvironment{SL}{\begin{list}{{$($\roman{smalllist}\/$)$\hss}}{%
\setlength{\topsep}{0mm}\setlength{\parsep}{0mm}\setlength{\itemsep}{0mm}%
\setlength{\labelwidth}{2.0em}\setlength{\itemindent}{2.5em}\setlength{\leftmargin}{0em}\usecounter{smalllist}%
}}{\end{list}}

\newcommand{\bit}{\noindent$\bullet$ }

\title[Solitary wave conjecture]{On the rigidity of solitary waves for the
 focusing mass-critical NLS in
dimensions $d\ge 2$}
\author{Dong Li}
\address{Institute for Advanced Study, Princeton, NJ}
\author{Xiaoyi Zhang}
\address{Academy of Mathematics and System Sciences, Beijing, and Institute for Advanced Study, Princeton, NJ}


\begin{document}

\maketitle

\begin{abstract}
For the focusing mass-critical NLS $iu_t+\Delta u=-|u|^{\frac 4d}u$,
it is conjectured that the only global non-scattering solution with
ground state mass must be a solitary wave up to symmetries of the
equation. In this paper, we settle the conjecture for $H^1_x$
initial data in dimensions $d=2,3$ with spherical symmetry and $d\ge
4$ with certain splitting-spherically symmetric initial data.
\end{abstract}

\tableofcontents
\section{Introduction}

\subsection{Background and main results}We consider the focusing mass-critical nonlinear Schr\"odinger equation
\begin{equation}\label{nls}
iu_t+\Delta u=-|u|^{\frac 4d}u
\end{equation}
in dimensions $d\geq 2$; here $u(t,x)$ is a complex-valued function
on $\mathbb R\times \R^d$. The equation is invariant under a number
of symmetries,
\begin{align}
u(t,x)&\mapsto u(t+t_0,x+x_0), \ t_0\in \R, x_0\in\R^d,\label{trans}\\
u(t,x)&\mapsto e^{i\theta_0}u(t,x), \ \theta_0\in \R,\label{phase}\\
u(t,x)&\mapsto  \lambda_0^{\frac
d2}u(\lambda_0^{2}t,\lambda_0 x)\label{scaling}, \ \lambda_0>0,\\
u(t,x)&\mapsto e^{\xi_0\cdot (x-\xi_0t)}u(t,x-2\xi_0t), \
\xi_0\in\R^d.\label{galilean}
\end{align}
From Ehrenfest's law, they lead to the following conserved
quantities:
\begin{align}
\mbox{Mass: } &M(u(t))=\int_{\R^d} |u(t,x)|^2 dx=M(u_0),\notag\\
\mbox{Energy: } &E(u(t))=\frac 12 \int_{\R^d} |\nabla u(t,x)|^2
dx-\frac d{2(d+2)} |u(t,x)|^{\frac {2(d+2)} d} dx \label{eq_Eudef}  \\
&\qquad\quad =E(u_0), \notag\\
\mbox{Momentum: } & P(u(t))=Im \int_{\R^d} \nabla u\bar u(t,x)dx=P(u_0)
\notag
\end{align}

The equation is called mass-critical since mass is invariant under
the scaling symmetry \eqref{scaling}. It is also critical in the
sense that the power of the nonlinearity is the smallest to admit
finite time blowup solutions.

Equation \eqref{nls} also preserves some other symmetries in space
due to the fact the Laplacian is an isotropic operator which is
invariant under orthogonal change of coordinates. For example, the
spherical symmetry is preserved under the NLS flow. The Cauchy
problem of \eqref{nls} for large spherically symmetric $L_x^2(\R^d)$
initial data has been intensively studied in recent works
\cite{ktv:2d}, \cite{kvz:blowup}. The advantage of using the
spherical symmetry ultimately stems from the fact that the solution
has to localize at the spatial origin $x=0$ and frequency origin
$\xi=0$, and also has strong decay as $|x|\to \infty$.

A natural generalization of the radial symmetry is
"splitting-spherical symmetry" which is also preserved under the
flow. To set the stage for later discussions, we now introduce:

\begin{defn}[Splitting-spherical symmetry]\label{spit}
Let $d\ge 2$. A function $f:\R^d\mapsto \mathbb C$ is said to be
splitting-spherically symmetric if there exists $k\ge 1$ and
$d_1,\cdots d_k$ with $d_i\ge 2$, $\sum_{i=1}^k  d_i=d$ such that
$\R^d=\R^{d_1}\times\cdots\times\R^{d_k}$ and $f$ is spherically
symmetric when restricted to each of the $\R^{d_i}$ subspaces. To
ensure the uniqueness, we require the fold number $k$ is minimal.
\end{defn}

Clearly when $k=1$, splitting-spherical symmetry coincides with the
usual notion of spherical symmetry. In this work, we will study the
Cauchy problem of \eqref{nls} in dimensions $d\ge 2$ with general
splitting-spherically symmetric initial data. We first make the
notion of a solution to this Cauchy problem more precise:

\begin{defn}[Solution]\label{D:solution}
A function $u: I \times \R^d \to \C$ on a non-empty time interval $I
\subset \R$ (possibly infinite or semi-infinite) is a \emph{strong
$L^2_x(\R^d)$ solution} (or \emph{solution} for short) to
\eqref{nls} if it lies in the class $C^0_t L^2_x(K \times \R^d) \cap
L^{2(d+2)/d}_{t,x}(K \times \R^d)$ for all compact $K \subset I$
and obeys the Duhamel formula
\begin{align}\label{old duhamel}
u(t_1) = e^{i(t_1-t_0)\Delta} u(t_0) + i \int_{t_0}^{t_1}
e^{i(t_1-t)\Delta} \bigl(|u|^{\frac 4d}u\bigr)(t)\, dt
\end{align}
for all $t_0, t_1 \in I$.  We refer to the interval $I$ as the
\emph{lifespan} of $u$. We say that $u$ is a \emph{maximal-lifespan
solution} if the solution cannot be extended to any strictly larger
interval. We say that $u$ is a \emph{global solution} if $I = \R$.
\end{defn}

The condition that $u$ belongs to $L_{t,x}^{2(d+2)/d}$ locally in
time is natural for several reasons. From the Strichartz estimate
(see Lemma \ref{L:strichartz}), we see that solutions to the linear
equation lie in this space. Moreover, the existence of solutions
that belong to this space is guaranteed by the local theory (see
Theorem~\ref{T:local} below). This condition is also necessary in
order to ensure uniqueness of solutions. Solutions to \eqref{nls} in
this class have been intensively studied; see, for example,
\cite{BegoutVargas, cwI, hmidi-keraani, keraani, ktv:2d, kvz:blowup,
merle1, merle2, compact, W1, weinstein:charact} and the many
references therein.

The local theory for the Cauchy problem of \eqref{nls} in the
critical $L_x^2(\R^d)$ space is established by Cazenave and Weissler
in \cite{cwI}. We record their results in the following

\begin{thm}[Local wellposedess, \cite{cwI, caz}] \label{T:local}
Given $u_0\in L_x^2(\R^d)$ and $t_0\in \R$, there exists a unique
maximal-lifespan solution $u$ to \eqref{nls} with $u(t_0)=u_0$. Let
$I$ denote the maximal lifespan.  Then,

\bit (Local existence) $I$ is an open neighborhood of $t_0$.

\bit (Mass conservation) The solution obeys $M(u(t))=M(u_0)$.

\bit (Blowup criterion) If\/ $\sup I$ or $\inf I$ are finite, then
\begin{align*}
\|u\|_{L_{t,x}^{\frac{2(d+2)}d}([t, \sup
I)\times\R^d)}=\infty;\mbox{ or }
\|u\|_{L_{t,x}^{\frac{2(d+2)}d}((\inf I, t]\times\R^d)}=\infty; \
t\in I.
\end{align*}

\bit (Continuous dependence) The map that takes initial data to the
corresponding strong solution is uniformly continuous on compact
time intervals for bounded sets of initial data.

\bit (Scattering) If\/ $\sup I=\infty$ and $u$ has finite spacetime
norm forward in time:
$\|u\|_{L_{t,x}^{\frac{2(d+2)}d}([t,\infty))}<\infty$,
 then $u$ scatters in that direction,
that is, there exists a unique $u_+\in L_x^2(\R^d)$ such that
\begin{align}\label{like u+}
\lim_{t\to\infty}\|u(t)-e^{it\Delta}u_+\|_2= 0.
\end{align}
Conversely, given $u_+ \in L^2_x(\R^d)$ there is a unique solution
to \eqref{nls} in a neighborhood of infinity so that \eqref{like u+}
holds.  Analogous statements hold in the negative time direction.

\bit (Small data global existence and scatter) If\/ $M(u_0)$ is
sufficiently small depending on the dimension $d$, then $u$ is a
global solution with finite $L_{t,x}^{2(d+2)/d}$-norm.
\end{thm}

By Theorem \ref{T:local}, all solutions with sufficiently small mass
are global and scatter both forward and backward in time. In that
regime, the dispersion effect of the free evolution dominates the
focusing nonlinearity. However for solutions with large mass, there
is competition between the two and solutions may display different
behaviors: they can exist globally and scatter, or blow up at finite
time, or persist like a solitary wave, or even be the superposition
of them \cite{tao:attact,tao:asym}. For this mass-critical problem
\eqref{nls}, there exists soliton solutions of the type
$e^{it}R(x)$, where $R(x)$ solves the elliptic equation
\begin{align} \label{eq_R}
\Delta R-R+|R|^{\frac 4d}R=0.
\end{align}
This equation has infinitely many solutions, but only one positive,
spherically symmetric solution which is Schwartz and has minimal
mass within the set of solutions. This unique solution $Q(x)$ is
known as the "ground state" of \eqref{nls}. It is widely believed
that the mass of $Q$ serves as the borderline between the scattering
solutions and possible non-scattering solutions. \footnote{By
"non-scattering" we mean the $L_{t,x}^{\frac{2(d+2)}d}$-norm of the
solution is infinite, so the solution can blows up at finite time or
exist globally but does not scatter.} If the mass of the initial
data is less than that of the ground state, then it is conjectured
that the corresponding solution exists globally and scatters. A more
precise statement is the following

\begin{conj}[Scattering conjecture]\label{scat}
Let $u_0\in L_x^2(\R^d)$ and $M(u_0)<M(Q)$. Then there exists unique
global solution $u(t,x)$ such that
\begin{align*}
\|u\|_{L_{t,x}^{\frac{2(d+2)}d}(\R\times\R^d)}\le C(M(u))<\infty.
\end{align*}
\end{conj}

By this conjecture, non-scattering solutions must have at least
ground state mass. Two examples of the non-scattering solutions are
generated by the ground state. One is the solitary wave solution
$e^{it}Q$ which exists globally but does not scatter on both sides.
Applying pseudo-conformal transformation:
\begin{align}
u(t,x)\to \frac 1{|t|^{\frac d2}}e^{\frac{i|x|^2}{4t}}\bar u(
\frac1t,\frac x
t),\label{pct}
\end{align}
 which is invariant
for the mass-critical NLS, one obtains a finite time blowup solution
\begin{align}
\frac 1{|t|^{\frac d2}}e^{\frac{i(|x|^2-4)}{4t}}Q(\frac x
t).\label{pcq}
\end{align}
These two examples are believed to be the only two obstructions to
scattering when the solution has ground state mass. In particular,
solitary wave is conjectured to be the only global non-scattering
solution. We formulate this as the following
\begin{conj}[Solitary wave conjecture]\label{swc} Let $d\ge 1$. For general initial data $u_0\in
L_x^2(R^d)$ with ground state mass, the corresponding non-scattering
global solution must be the solitary wave up to symmetries
\eqref{trans}-\eqref{galilean}.
\end{conj}

Our purpose of the paper is to settle the solitrary wave conjecture under
additional assumptions on the initial data. More precisely, we will
establish the conjecture for $H_x^1$ initial data in $2,3$ dimensions
with spherical symmetry and in dimensions $d\ge 4$ with some splitting-spherical symmetry. For
simplicity, we name these "certain symmetries" as the following

\begin{defn}[Admissible symmetry]\label{admis}
In dimensions $d=2,3$, the admissible symmetry refers to the
spherical symmetry; in dimensions $d\ge 4$, it refers to the
splitting-spherical symmetry with $k=2$, and $d_1=[\frac d2]$ or
$d_1>\frac d3$ for sufficiently large $d$. Here $[x]$
denotes the integer part of a real number $x$.
\end{defn}

Then our main result reads as follows.

\begin{thm}[Non-scattering solutions must coincide with the solitary wave] \label{main}
Let $d\ge 2$. Let $u_0\in H_x^1(\R^d)$ and have the admissible
symmetry. Suppose also the corresponding solution exist globally,
then only the following two scenarios can occur

1. The solution scatters in both time direction, i.e.,
$\|u\|_{L_{t,x}^{\frac{2(d+2)}d}(\R\times \R^d)}<\infty.$

2. The solution is spherically symmetric and there exist $\theta_0,
\lambda_0$ such that
\begin{align*}
u(t,x)=e^{i\theta_0}\lambda_0^{\frac d2}Q(\frac x{\lambda_0}).
\end{align*}
\end{thm}

\begin{rem}
For the splitting-spherical symmetric initial data, Theorem \ref{main} holds under
the conditional assumption that the corresponding scattering conjecture holds for such $L^2$
intial data. In the spherical symmetric case for dimensions $d\ge 2$, the scattering conjecture
has been proved in recent works \cite{ktv:2d} \cite{kvz:blowup}. In a future publication we will address the scattering
problem for the splitting-spherical symmetric case for dimensions $d\ge 4$.
\end{rem}

In the following, we will first discuss the connections between
this result and previous ones. Then at the end of this
section, we introduce the main steps of the proof.

For this mass critical problem, there has been lots of work
addressing the wellposedness theory of the solutions.
We will mainly discuss the results related
to the aforementioned two conjectures, with a little extension on the blowup
theory. In the following discussions, we distinguish three different
cases when the solution has subcritical, critical and supercritical
mass respectively.

\vspace{0.2cm}

\textbf{Case 1}. The solution has subcritical mass $M(u)<M(Q)$.

\vspace{0.2cm}

As shown in the scattering conjecture \ref{scat}, in this regime,
the dispersion of the linear flow dominates and solutions are all
believed to exist globally and scatter. The first result toward this
conjecture is due to Weinstein. In \cite{W1}, he established the
following variational characterization of the ground state which
says that the ground state $Q$ extremizes the Gagliardo-Nirenberg
inequality
\begin{prop}[Sharp Gagliardo--Nirenberg inequality, \cite{W1}]\label{P:variational}
For $f\in H_x^1(\mathbb R^d)$,
\begin{equation}\label{sharp-gn}
\|f\|_{\tfrac{2(d+2)}d}^{\tfrac{2(d+2)}d}\le
\frac{d+2}d\left(\frac{\|f\|_2}{\|Q\|_2} \right)^{\frac 4d}\|\nabla
f\|_2^2,
\end{equation}
with equality if and only if
\begin{equation}\label{GNeq}
f(x)= c e^{i\theta_0}\lambda_0^{\frac d2}Q(\lambda_0 (x-x_0))
\end{equation}
for some $\theta\in[0,2\pi)$, $x_0\in\R^d$, and
$c,\lambda_0\in(0,\infty)$.  In particular, if $M(f)=M(Q)$, then
$E(f)\geq 0$ with equality if and only \eqref{GNeq} holds with
$c=1$.
\end{prop}
As a consequence of the inequality \eqref{sharp-gn}, Weinstein
\cite{W1} showed that if $u_0\in H_x^1(\R^d)$ with $M(u_0)<M(Q)$,
then the corresponding solution satisfies $\|\nabla
u(t)\|_{\ltrd}<Const\cdot E(u(t))$ for all $t$ in the maximal
lifespan, by which standard local theory in $H^1_x$-space yields the
global wellposedness.  Therefore, no finite time singularities can
form in this subcritical regime at least for $H^1_x$ initial data.
However, this result does not address the scattering issue of the
solution.

A great breakthrough toward the scattering conjecture was recently made by
Killip-Tao-Visan \cite{ktv:2d} where they settled the conjecture in
two dimensions with spherical symmetry. This result was later
extended to high dimensions $d\ge 3$ by Killip-Visan-Zhang
\cite{kvz:blowup}. The spherical symmetry in these results is used
in an essential way. For example, with the spherical symmetry, the
center of mass will freeze at the origin in both physical and
frequency spaces. Moreover, spherical symmetry forces decay in space
when $|x|\to \infty$. The techniques developed in these works take
advantage of decay property and will break down when such a decay is
not available.

Therefore, it is quite challenging to prove the scattering
conjecture \ref{scat} without spherical symmetry. As a first step
forward, we would like to understand the case when the solution has
some symmetry which is enough to freeze the center of mass, and at the
same time, provide some averaging effect. The splitting symmetry
then fits this motivation. In a future publication, we will show that the scattering
conjecture holds true in high dimensions under these weaker symmetries.

We make two remarks before completing this  discussion. First of
all, as we shall explain, with the additional $H_x^1$ assumption,
the main part of the proof, which is also a
necessity in proving Theorem \ref{main}, is to show the localization
of kinetic energy for the minimal non-scattering solutions.
Secondly, it is possible to build the scattering conjecture with this
splitting-spherical symmetry for $L_x^2$ initial data. We
will address this issue in the future works. We turn now to

\vspace{0.2cm}

\textbf{Case 2.} The solution has super-critical mass $M(u)>M(Q)$.

\vspace{0.2cm}

In this case, the focusing nonlinearity dominates the dispersion
effect of the linear evolution and finite time blowups may occur.
The existence of finite time blowup solutions was first obtained by
Glassey \cite{glassey} for $\Sigma=\{f\in H_x^1(\R^d),xf\in
L_x^2(\R^d)\}$ initial data with negative energy. The argument is
based on a simple application of the virial identity
\begin{align}
\frac {d^2}{dt^2}\int |x|^2 |u(t,x)|^2 dx=8E(u_0),
\end{align}
This identity essentially expresses the conservation of
pseudo-conformal energy which comes from the pseudo-conformal
symmetry \eqref{pct}.

Concerning the quantitative theory of blowups, there has been a lot
of work addressing the dynamics of the blowup solutions. For
example, it is possible to show that any finite time blowup solution
will concentrate at least ground state mass near the blowup time.
Results in this direction were established by
Merle-Tsutsumi\cite{merle-tsu} and Nawa\cite{nawa},
Weinstein\cite{weinstein:contem} for the finite energy blowup
solutions. For merely finite mass blowup solutions, such results
were proved by Bourgain\cite{bourgain:concen},
Keraani\cite{keraani}, Killip-Tao-Visan\cite{ktv:2d} and
Killip-Visan-Zhang\cite{kvz:blowup}.

To understand the structure of blowups, an important first step is
to study the optimal blowup rate of an $H_x^1$ blowup solution. If
an $H_x^1$ solution $u(t,x)$ blows up at a finite time $T$, then by
scaling arguments, the kinetic energy blows up at least at a power
like rate:
\begin{align*}
\|\nabla u(t)\|_{\ltrd}\gtrsim \frac 1{|t-T|^{\frac 12}}.
\end{align*}
In dimensions $d=1,2$, the existence of blowup solution with rate
$\frac 1{|t-T|}$ was proved by Bourgain-Wang in
\cite{bourgain-wang}. Perelman in \cite{perelman} constructed a
solution blowing up at rate $ \sqrt{\frac{\ln(\ln|t-T|)}{|t-T|}}$ in
one dimension. In \cite{merle-raph}, Merle-Raphael proved that blowup
rate has an upper bound $\sqrt{\frac{|\ln(T-t)|^{\frac d2}}{T-t}}$
for solution with negative energy, and having mass slightly bigger
than that of the ground state.

\vspace{0.2cm}

\textbf{Case 3}. The solution has the ground state mass $M(u)=M(Q)$.

\vspace{0.2cm}

As shown in the scattering conjecture, $M(Q)$ is conjectured to be
the minimal mass for all the non-scattering solutions. The
ground state provides two non-scattering solutions at this minimal
mass: the solitary wave and the pseudo-conformal transformation of
the solitary wave. They are believed to be the only two thresholds
for scattering at minimal mass. In particular, as indicated in the
solitary conjecture, the only non-scattering solution is conjectured
to be the solitary wave up to symmetries of the equation.

 The first work which addressed
the description of the minimal non-scattering solutions is due to F.
Merle in \cite{merle1}. He proved that any finite time blowup
solution with $H^1_x$ initial data must be pseudo-conformal ground
state solution \eqref{pcq} up to symmetries. The proof in
\cite{merle1}, which was later simplified by
Hmidi-Keraani\cite{hmidi-keraani} relies heavily on the finiteness
of the blowup time. To see the connections between his work and the
solitary wave conjecture, we simply use the pseudo-conformal
transformation to transform a global non-scattering solution to a
finite time blowup solution. Therefore, when the initial data
belongs to $\Sigma$ space(thus the transformed finite time blowup
solution belongs to $H_x^1(\R^d)$), solitary wave conjecture follows
directly from Merle's results.

Without the strong decay assumption, the solitary wave conjecture
remains totally open. In \cite{klvz}, Killip-Li-Visan-Zhang first
proved this conjecture in dimension $d\ge 4$ for $H_x^1(\R^d)$
initial data with spherical symmetry. In our main theorem
\ref{main}, we establish the result for $H_x^1(\R^d)$ initial data
in 2,3-dimensions with spherical symmetry and in dimensions $d\ge 4$
with certain splitting-spherical symmetry. Combining our results
with Merle's results, we conclude that the solitary wave, the
pseudo-conformal ground state and scattering are the only three
possible states for solutions with minimal mass.

As we shall see, the techniques we are going to use in this paper
rely on the fact that $u_0\in H_x^1$ and the splitting-spherical
symmetry. For example, we need the regularity to define the energy,
which then allows us to use virial-type argument. We also need this
regularity to conduct the spectral analysis around the ground state
$Q$. Similar to the previous work \cite{klvz,ktv:2d,kvz:blowup}, the
decay property stemming from the splitting-spherical symmetry plays
a very important role. It is not clear to us how to extend the
result to the rough initial data and to the case without the
splitting-spherical symmetry.

Another challenging problem is to consider the conjecture in one
dimension, for example, with the evenness assumption. Since the one
dimensional function does not have any spatial decay, this problem
is indeed equally hard with the above mentioned ones.

We now introduce the

\subsection{Outline of the proof}

We first consider the case $d=2,3$. Let $u_0\in H_x^1$,
$M(u_0)=M(Q)$ be spherically symmetric. If the corresponding
solution $u(t,x)$ exists globally and scatter, then Theorem
\ref{main} holds vacuously. Therefore we assume $u(t,x)$ does not
scatter at least in one time direction, for example,
$\|u\|_{L_{t,x}^{\frac{2(d+2)}d}(\srd)}=\infty$. Our goal is then to
show that $u_0=Q$ up to phase rotation and scaling. The coincidence
of the solution with the solitary wave follows from the uniqueness
of the solution.

From the variational characterization of the ground state (see
Proposition \ref{P:variational}),
 ground state $Q$ minimizes the energy. Therefore, if the solution has zero energy,
 we conclude the initial data coincides with the ground state $Q$. This leaves us
to consider the positive energy case and we will get a contradiction
in this case.

The contradiction will follow from a suitable truncated version of
the virial identity as we now explain. On the one hand, a simple
computation shows that the truncated virial has a uniform upper
bound. On the other hand, its second derivative in time will have a
positive lower bound as long as we can show \emph{the kinetic energy
concentrates at the origin uniformly in time}. These two facts
ultimately yields the desired contradiction.

Hence, our task is reduced to showing that the kinetic energy of the
solution is uniformly localized in time. Suppose by contradiction
that the kinetic energy is not localized, then there are always
significant portion of kinetic energy ripples live on an ever large
radii. This scenario is what we need to preclude, on the other
hand, looks very "consistent" with the focusing nature of
our problem. More precisely, due to the focusing nature, the
solution may have asymptotic infinite kinetic energy, hence it is
quite possible to shed them on the ever large radius.

To preclude this dangerous scenario, our first step is to understand
how the total kinetic energy ripples are distributed away from the origin. If
the kinetic energy of the solution is uniformly bounded, then so
will be these total ripples. Then comes the interesting case where
the kinetic energy goes to infinity along a subsequence. In this
case, a qualitative argument shows that after rescaling, the
solution converges to the ground state in $H_x^1(\R^d)$. Therefore
along this subsequence the solution can be decomposed into a
rescaled copy of the ground state (ever concentrated) plus an error.
A surprising result from our non-sharp decomposition Proposition
\ref{prop_nonsharp} essentially gives a good control of the error. In
all, we have the following

\begin{prop}[Weak localization of kinetic energy]\label{uni_bdd}
Let $u_0\in H_x^1$ have admissable symmetry and  $M(u_0)=M(Q)$. Let $u(t,x)$ be the
corresponding maximal-lifespan solution on $I$. Then for all $t\in
I$, we have
\begin{align}
\|\phi_{\gtrsim 1} \nabla u(t)\|_{\ltrd}\lsm 1.
\end{align}
\end{prop}

A more precise form of Proposition \ref{uni_bdd} is given by Lemma \ref{lem_uni_bdd} which
is proved in Section 3. We still have to upgrade this weak localization to stronger one.
That is to say, instead of being bounded, the ripples far away from
the origin are actually very small. This fact is heavily used in
the aforementioned virial argument.

The solution does enjoy some certain strong concentration property,
however, not expressed in terms of the kinetic energy, but in terms
of the mass. In the spherically symmetric case, this is direct result of
\cite{BegoutVargas,keraani-2,compact} together with the
identification of $M(Q)$ as the minimal mass, which was done in
\cite{ktv:2d}, \cite{kvz:blowup}. In the case with admissable symmetry,
this follows from the resolution of the corresponding scattering conjecture.

We state the following result in the general dimensions $d\ge2$
with admissable symmetry. In the general case without symmetry, we
need to assume the scattering conjecture in order to identify $M(Q)$
as the minimal mass.

\begin{thm}[Almost periodicity modulo symmetries]
Let $d\ge 2$. Let $u:[t_0,\infty) \times \mathbb R^d\to\C$ be a
solution to \eqref{nls} which satisfies $M(u)=M(Q)$ and
$$\|u\|_{L_{t,x}^{\frac{2(d+2)}d}([t_0,\infty)\times\R^d)}=\infty,$$
where $u$ have the admissable symmetry.
In the general case without symmetry, we also assume the scattering conjecture. Then $u$
is \emph{almost periodic modulo symmetries} in the following sense:
there exist functions $N:[t_0,\infty)\to \mathbb R^+$ and $C:\mathbb
R^+\to\mathbb R^+$, $x,\xi:\R^d\to\R^d$ such that
\begin{equation}\label{eq_mass_local}
\int_{|x-x(t)|\ge C(\eta)/N(t)}|u(t,x)|^2 dx\le \eta \quad
\text{and} \quad  \int_{|\xi-\xi(t)|\ge C(\eta)N(t)}|\hat
u(t,\xi)|^2 d\xi\le \eta
\end{equation}
for all $t\in [t_0, \infty)$ and $\eta>0$. Equivalently, the orbit
$\{N(t)^{-\frac d2}e^{ix\cdot\xi(t)}u(t,\frac{x-x(t)}{N(t)})\}$
falls into a compact set in $L_x^2(\R^d)$. If $u$ is spherically
symmetric or splitting-spherically symmetric, then $x(t)=\xi(t)=0$.
\end{thm}

\begin{rem}\label{R:bdd N} The parameter $N(t)$ measures the frequency scale of the solution at time $t$ while $1/N(t)$ measures its spatial scale.
Further properties of the function $N(t)$ are discussed in
\cite{ktv:2d, compact}.
\end{rem}

One important consequence of the fact that $u$ is almost periodic
modulo scaling (near positive infinity) is the following Duhamel
formula, where the free evolution term disappears:

\begin{lem}[{\cite[Section 6]{compact}}]\label{duhamel L}
Let $u$ be an almost periodic solution to \eqref{nls} on
$[t_0,\infty)$.  Then, for all $t\in [t_0,\infty)$,
\begin{equation}\label{duhamel}
u(t) = - \lim_{T\nearrow\,\infty}i\int_t^T e^{i(t-t')\Delta}
\bigl(|u|^{\frac 4d}u\bigr)(t')\,dt'
\end{equation}
as a weak limit in $L_x^2$.
\end{lem}

We will use this Duhamel formula with the in-out decomposition
technique as used in \cite{klvz} to upgrade the weak localization of kinetic
energy to the desired stronger one. Note that the in-out decomposition
technique exploits heavily the spherical symmetry of the solution. Since
Proposition \ref{uni_bdd} already gives us the uniform $H_x^1$ boundedness of the
solution away from the spatial origin, the main obstacle is the control of the nonlinearity
near the origin. In this regime due to the high degree of nonlinearity and lack of
$H_x^1$ control, the contribution of this part seems difficult to handle. We shall
overcome this difficulty by introducing a linear flow trick.
More precisely, we decompose the solution into incoming and outgoing waves; this
serves to minimize the contribution from the nonlinearity near the
origin where we do not have uniform control on the kinetic energy.
As was already mentioned, the part close to the origin is the most problematic, since the kinetic energy
may grow out of control as $t\to \infty$. For this regime, we
decrease the power of nonlinearity by substituting the nonlinearity
by the linear flow. After integration in time and using some kernel estimates, we succeed to control it
by the mere mass of $u$. At large radii we can take advantage of spherical symmetry to
obtain smallness. Since we have the uniform control on the kinetic
energy at this regime, the high power of nonlinearity does not cause
trouble. The key point of the in-out decomposition is to use the
Duhamel formula into the future to control the outgoing portion of
$u$ and the Duhamel formula into the past to control the incoming
portion. The particular decomposition we use is taken from
\cite{ktv:2d,kvz:blowup}; the tool we use to exploit the spherical
symmetry is a weighted Strichartz inequality Lemma~\ref{L:wes} and
the weighted Sobolev embedding from \cite{ktv:2d}, \cite{kvz:blowup}
and \cite{tvz:hd}.

As a consequence of the above analysis, we not only prove the
frequency decay estimate Proposition \ref{fre_decay}, but also obtain a
spatial decay estimate Proposition \ref{local} with the spatial
scale independent of $t$. Combining these two propositions yields a
\emph{uniform} kinetic energy localization result. This is very
surprising comparing with the mass localization property where mass
is localized with spatial scale varying with time $t$. Specifically,
we have the following

\begin{thm}[Kinetic energy localization in $2,3$-dimensions]\label{T:kinetic_loc}
Let $d=2,\ 3$. Let $u_0\in H_x^1(\R^d)$ be spherically symmetric and
satisfy $M(u)=M(Q)$. In particular, as a consequence of Corollary
\ref{uni_bdd}, the solution satisfies
\begin{align*}
\|\phi_{\gtrsim 1} \nabla u(t)\|_{\ltrd}\lsm 1.
\end{align*}
Let the corresponding solution $u(t,x)$ exists globally forward in
time and satisfy the Duhamel formula \eqref{duhamel L}. The for
any $\eta>0$, there exists $C(\eta)$ such that
\begin{align*}
\|\phi_{>C(\eta)}\nabla u(t)\|_{\ltrd}\le \eta, \ \forall t\ge 0.
\end{align*}
\end{thm}

As explained before, this proposition together with the cheap
localized virial argument establishes Theorem \ref{main}.

To conclude, the spherical symmetry is used in an essential way to
establish the kinetic energy localization in 2,3 dimensions.
First of all, it forces the solution to localize at the origin, weak
localization result Proposition \ref{uni_bdd} cannot hold without
this assumption since the center of the solution can move with time
$t$. Secondly, the spherical symmetry forces decay as $|x|\to
\infty$ which contributes directly to the additional smoothness
(Proposition \ref{fre_decay}) and the additional decay estimates(Proposition
\ref{local}), hence the kinetic energy localization Theorem
\ref{T:kinetic_loc} follows.

It is then interesting to consider the case when the solution is not
spherically symmetric, but has some symmetry enough to localize the
solution at the origin, at the same time, provides enough averaging
effect. This motives us to consider the splitting-spherical
symmetry(see Definition \ref{spit}). For example, in four
dimensional case, the only nontrivial splitting-spherical symmetry
simply requires the function to be spherically symmetric in each of
the two 2-dimensional subspaces.

In the splitting-spherical symmetric case, by the same argument as
in lower dimensions with radial assumption, all the matter is
reduced to showing the kinetic energy localization. Since this
symmetry also forces the solution to stay at the origin, weak
localization of kinetic energy (Proposition \ref{uni_bdd}) still
holds in this case. Therefore all we have to do is to upgrade this
weak localization to a stronger one. More precisely, we will prove
the following

\begin{thm}[Kinetic energy localization in $d\ge 4$ with admissible symmetry]\label{T:kine_hd}
Let $d\ge 4$ and $u_0\in H_x^1(\R^d)$ with $M(u_0)\le M(Q)$. Let
$u_0$ have the admissible symmetry. Let $u(t,x)$ be corresponding
global solution forward in time satisfying the Duhamel formula
\eqref{duhamel}. Assume also
\begin{equation}
\|\phi_{\gtrsim 1} \nabla u(t)\|_{\ltrd}\lsm 1.\label{unibdd}
\end{equation}
Then for any $\eta>0$, there exists $C(\eta)>0$ such that
\begin{equation*}
\|\phi_{>C(\eta)}\nabla u(t)\|_{\ltrd}\le \eta.
\end{equation*}
\end{thm}

We remark that in the case when $M(u)<M(Q)$, \eqref{unibdd} is a
direct consequence of the sharp Gagliardo- Nirenberg inequality; in the case
when $M(u)=M(Q)$, it is ensured by Proposition \ref{uni_bdd} which
is a consequence of the non-sharp decomposition Proposition
\ref{prop_nonsharp}.

The proof of Theorem \ref{T:kine_hd} is technically more involved.
It is ultimately reduced to understanding the decay of a single
frequency of the solution: $P_N u(t)$ in both frequency and physical
space. However, not like the radial case where we have only one
preferred direction (namely, the radial direction), in this case, we
will have two. Due to this anisotropicity, the waves that travel at
certain speed $N$ may have the same speed in one direction, but stay
static in the other. It is not hard to imagine that the desired
smoothness and the decay will only comes from the traveling part of
waves. To confirm this, we will apply the sub-dimensional in-out
decomposition technique as we shall explain in Section \ref{twotwo}
and Section \ref{highd}. However, since the spatial cutoff does not
necessarily induce cutoffs of the waves in the preferred direction,
this decomposition is not directly useful to minimize the
contribution of these waves near the origin. The outcome turns out
to be a detailed discussion of various mixture terms with different
sub-dimensional spatial cutoff and the sub-dimensional
incoming/outgoing projection operators. Same with the 2,3
dimensional case, the contribution of the nonlinearity away from the
origin is the dominate part, and the proof relies heavily on the
spherical symmetry in that sub-dimension. It is here the restriction
on the fold $k$ and the minimal dimension $d_1$ of the splitting
spherical symmetry appears. It would be interesting to extend this
result to the more general splitting-spherical symmetry case.

\subsection*{Acknowledgements}
Both authors were supported in part by the National Science Foundation under agreement No. DMS-0635607.
  X.~Zhang was also supported by NSF grant No.~10601060
and project 973 in China.

%
%
%
%

\section{Preliminaries}

\subsection{Some notation}
We write $X \lesssim Y$ or $Y \gtrsim X$ to indicate $X \leq CY$ for some constant $C>0$.  We use $O(Y)$ to denote any quantity $X$
such that $|X| \lesssim Y$.  We use the notation $X \sim Y$ whenever $X \lesssim Y \lesssim X$.  The fact that these constants
depend upon the dimension $d$ will be suppressed.  If $C$ depends upon some additional parameters, we will indicate this with
subscripts; for example, $X \lesssim_u Y$ denotes the assertion that $X \leq C_u Y$ for some $C_u$ depending on $u$.
We sometimes write $C=C(Y_1, \cdots, Y_n)$ to stress that the constant $C$ depends on quantities $Y_1$, $\cdots$, $Y_n$.
We denote by $X\pm$ any quantity of the form $X\pm \epsilon$ for any $\epsilon>0$.

We use the `Japanese bracket' convention $\langle x \rangle := (1 +|x|^2)^{1/2}$.

We write $L^q_t L^r_{x}$ to denote the Banach space with norm
$$ \| u \|_{L^q_t L^r_x(\R \times \R^d)} := \Bigl(\int_\R \Bigl(\int_{\R^d} |u(t,x)|^r\ dx\Bigr)^{q/r}\ dt\Bigr)^{1/q},$$
with the usual modifications when $q$ or $r$ are equal to infinity, or when the domain $\R \times \R^d$ is replaced by a smaller
region of spacetime such as $I \times \R^d$.  When $q=r$ we abbreviate $L^q_t L^q_x$ as $L^q_{t,x}$.

Throughout this paper, we will use $\phi\in C^\infty(\R^d)$ be a
radial bump function supported in the ball $\{ x \in \R^d: |x| \leq
\frac{25} {24} \}$ and equal to one on the ball $\{ x \in \R^d: |x|
\leq 1 \}$.  For any constant $C>0$, we denote $\phi_{\le C}(x):=
\phi \bigl( \tfrac{x}{C}\bigr)$ and $\phi_{> C}:=1-\phi_{\le C}$.

\subsection{Basic harmonic analysis}\label{ss:basic}

For each number $N > 0$, we define the Fourier multipliers
\begin{align*}
\widehat{P_{\leq N} f}(\xi) &:= \phi_{\leq N}(\xi) \hat f(\xi)\\
\widehat{P_{> N} f}(\xi) &:= \phi_{> N}(\xi) \hat f(\xi)\\
\widehat{P_N f}(\xi) &:= (\phi_{\leq N} - \phi_{\leq N/2})(\xi) \hat
f(\xi)
\end{align*}
and similarly $P_{<N}$ and $P_{\geq N}$.  We also define
$$ P_{M < \cdot \leq N} := P_{\leq N} - P_{\leq M} = \sum_{M < N' \leq N} P_{N'}$$
whenever $M < N$.  We will usually use these multipliers when $M$ and $N$ are \emph{dyadic numbers} (that is, of the form $2^n$
for some integer $n$); in particular, all summations over $N$ or $M$ are understood to be over dyadic numbers.  Nevertheless, it
will occasionally be convenient to allow $M$ and $N$ to not be a power of $2$.  As $P_N$ is not truly a projection, $P_N^2\neq P_N$,
we will occasionally need to use fattened Littlewood-Paley operators:
\begin{equation}\label{PMtilde}
\tilde P_N := P_{N/2} + P_N +P_{2N}.
\end{equation}
These obey $P_N \tilde P_N = \tilde P_N P_N= P_N$.

Like all Fourier multipliers, the Littlewood-Paley operators commute with the propagator $e^{it\Delta}$, as well as with
differential operators such as $i\partial_t + \Delta$. We will use basic properties of these operators many many times,
including

\begin{lem}[Bernstein estimates]\label{Bernstein}
 For $1 \leq p \leq q \leq \infty$,
\begin{align*}
\bigl\| |\nabla|^{\pm s} P_N f\bigr\|_{L^p_x(\R^d)} &\sim N^{\pm s} \| P_N f \|_{L^p_x(\R^d)},\\
\|P_{\leq N} f\|_{L^q_x(\R^d)} &\lesssim N^{\frac{d}{p}-\frac{d}{q}} \|P_{\leq N} f\|_{L^p_x(\R^d)},\\
\|P_N f\|_{L^q_x(\R^d)} &\lesssim N^{\frac{d}{p}-\frac{d}{q}} \| P_N f\|_{L^p_x(\R^d)}.
\end{align*}
\end{lem}

While it is true that spatial cutoffs do not commute with Littlewood-Paley operators, we still have the following:

\begin{lem}[Mismatch estimates in real space]\label{L:mismatch_real}
Let $R,N>0$.  Then
\begin{align*}
\bigl\| \phi_{> R} \nabla P_{\le N} \phi_{\le\frac R2} f \bigr\|_{L_x^p(\R^d)}  &
\lsm_m N^{1-m} R^{-m} \|f\|_{L_x^p(\R^d)} \\
\bigl\| \phi_{> R}  P_{\leq N} \phi_{\le\frac R2} f \bigr\|_{L_x^p(\R^d)}
&\lsm_m N^{-m} R^{-m} \|f\|_{L_x^p(\R^d)}
\end{align*}
for any $1\le p\le \infty$ and $m\geq 0$.
\end{lem}

\begin{proof}
We will only prove the first inequality; the second follows similarly.

It is not hard to obtain kernel estimates for the operator $\phi_{>
R}\nabla P_{\le N}\phi_{\le\frac R2}$. Indeed, an exercise in
non-stationary phase shows
\begin{align*}
\bigl|\phi_{> R}\nabla P_{\le N}\phi_{\le\frac R2}(x,y)\bigr|
\lesssim N^{d+1-2k} |x-y|^{-2k}\phi_{|x-y|>\frac R2}
\end{align*}
for any $k\geq 0$.  An application of Young's inequality yields the claim.
\end{proof}

Similar estimates hold when the roles of the frequency and physical spaces are interchanged.  The proof is easiest when
working on $L_x^2$, which is the case we will need; nevertheless, the following statement holds on $L_x^p$ for any $1\leq p\leq \infty$.

\begin{lem}[Mismatch estimates in frequency space]\label{L:mismatch_fre}
For $R>0$ and $N,M>0$ such that $\max\{N,M\}\geq 4\min\{N,M\}$,
\begin{align*}
\bigl\|  P_N \phi_{\le{R}} P_M f \bigr\|_{\ltrd} &\lsm_m \max\{N,M\}^{-m} R^{-m} \|f\|_{\ltrd} \\
\bigl\|  P_N \phi_{\le {R}} \nabla P_M f \bigr\|_{\ltrd} &\lsm_m M
\max\{N,M\}^{-m} R^{-m} \|f\|_{\ltrd}.
\end{align*}
for any $m\geq 0$.  The same estimates hold if we replace $\phi_{\le
R}$ by $\phi_{>R}$.
\end{lem}

\begin{proof}
The first claim follows from Plancherel's Theorem and Lemma~\ref{L:mismatch_real} and its adjoint.  To obtain the second claim from this, we write
$$
P_N \phi_{\le {R}} \nabla P_M = P_N \phi_{\le {R}} P_M \nabla \tilde
P_M
$$
and note that $\|\nabla \tilde P_M\|_{L_x^2\to L_x^2}\lesssim M$.
\end{proof}

We will need the following radial Sobolev embedding to exploit the
decay property of a radial function. For the proof and the more
complete version, one refers to see \cite{tvz:hd}.

\begin{lem}[Radial Sobolev embedding, \cite{tvz:hd}]\label{L:radial_embed}
Let dimension $d\ge 2$. Let $s>0$, $\alpha>0$, $1<p,q<\infty$ obeys
the scaling restriction: $\alpha+s=d(\frac 1q-\frac 1p)$. Then the
following holds:
\begin{align*}
\||x|^{\alpha} f\|_{L^p(\R^d)}\lsm \||\nabla|^s f\|_{L^q(\R^d)},
\end{align*}
where the implicit constant depends on $s,\alpha,p,q$. When
$p=\infty$, we have
\begin{align*}
\||x|^{\frac{d-1}2}P_N f\|_{L^\infty(\R^d)}\lsm N^{\frac 12}\|P_N
f\|_{\ltrd}.
\end{align*}
\end{lem}

We will need the following fractional chain rule lemma.
\begin{lem}[Fractional chain rule for a $C^1$ function, \cite{chris:weinstein}\cite{staf}\cite{taylor}]\label{lem_chain}
Let $F\in C^1(\mathbb C)$, $\sigma \in (0,1)$, and $1<r,r_1,r_2<\infty$ such that $\frac 1r=\frac 1{r_1}+\frac 1{r_2}$.
Then we have
\begin{align*}
\||\nabla|^{\sigma}F(u)\|_r\lesssim \|F'(u)\|_{r_1}\||\nabla|^{\sigma}u\|_{r_2}.
\end{align*}
\end{lem}
\begin{proof}
See \cite{chris:weinstein}, \cite{staf} and \cite{taylor}.
\end{proof}

\subsection{Strichartz estimates}

The free Schr\"odinger flow has the explicit expression:
\begin{align*}
e^{it\Delta } f(x)=\frac 1{(4\pi t)^{d/2}}\int_{\R^d}
e^{i|x-y|^2/4t}f(y)dy,
\end{align*}
from which we can derive the kernel estimate of the frequency
localized propagator. We record the following

\begin{lem}[Kernel estimate\cite{ktv:2d,kvz:blowup}]\label{L:kernel}
For any $m\ge 0$, we have
\begin{align*}
|(\pn e^{it\Delta}(x,y)|\lsm_m
\begin{cases}
|t|^{-d/2},&: |x-y|\sim Nt;\\
\frac{N^d}{|N^2 t|^m\langle N|x-y|\rangle^m}&: \mbox{otherwise}
\end{cases}
\end{align*}
for $|t|\ge N^{-2}$ and
\begin{align*}
|(\pn e^{it\Delta})(x,y)|\lsm_m N^d\langle N|x-y|\rangle^{-m}
\end{align*}
for $|t|\le N^{-2}$.
\end{lem}

We will frequently use the standard Strichartz estimate:

\begin{lem}[Strichartz]\label{L:strichartz} Let $d\ge 2$. Let $I$ be an interval, $t_0 \in I$, and let $u_0 \in L^2_x(\R^d)$
and $F \in L^{2(d+2)/(d+4)}_{t,x}(I \times \R^d)$.  Then, the function $u$ defined by
$$ u(t) := e^{i(t-t_0)\Delta} u_0 - i \int_{t_0}^t e^{i(t-t')\Delta} F(t')\ dt'$$
obeys the estimate
$$
\|u \|_{L^\infty_t L^2_x} + \| u \|_{L^{\frac{2(d+2)}d}_{t,x}}
    \lesssim \| u_0 \|_{L^2_x} + \|F\|_{L^{\frac{2(d+2)}{d+4}}_{t,x}},
$$
where all spacetime norms are over $I\times\R^d$.
\end{lem}

\begin{proof}
See, for example, \cite{gv:strichartz, strichartz}.  For the endpoint see \cite{tao:keel}.
\end{proof}

We will also need a weighted Strichartz estimate, which exploits heavily the spherical symmetry in order to obtain spatial decay.

\begin{lem}[Weighted Strichartz, \cite{ktv:2d, kvz:blowup}]\label{L:wes} Let $I$ be an interval, $t_0 \in I$, and let
$F:I\times\R^d\to \C$ be spherically symmetric.  Then,
$$ \biggl\|\int_{t_0}^t e^{i(t-t')\Delta} F(t')\, dt' \biggr\|_{L_x^2}
\lesssim \bigl\||x|^{-\frac{2(d-1)}q}F \bigr\|_{L_t^{\frac{q}{q-1}}L_x^{\frac{2q}{q+4}}(I \times \R^d)}
$$
for all $4\leq q\leq \infty$.
\end{lem}

\subsection{The in-out decomposition}
We will need an incoming/outgoing decomposition; we will use the one developed in \cite{ktv:2d, kvz:blowup}.
As there, we define operators $P^{\pm}$ by
\begin{align*}
[P^{\pm} f](r) :=\tfrac12 f(r)\pm \tfrac{i}{\pi} \int_0^\infty \frac{r^{2-d}\,f(\rho)\,\rho^{d-1}\,d\rho}{r^2-\rho^2},
\end{align*}
where the radial function $f: \R^d\to \C$ is written as a function of radius only.
We will refer to $P^+$ is the projection onto outgoing spherical waves; however, it is not a true projection as it is neither idempotent
nor self-adjoint.  Similarly, $P^-$ plays the role of a projection onto incoming spherical waves; its kernel is the complex
conjugate of the kernel of $P^+$ as required by time-reversal symmetry.

For $N>0$ let $P_N^{\pm}$ denote the product $P^{\pm}P_N$ where $P_N$ is the Littlewood-Paley projection.
We record the following properties of $P^{\pm}$ from \cite{ktv:2d, kvz:blowup}:

\begin{prop}[Properties of $P^\pm$, \cite{ktv:2d, kvz:blowup}]\label{P:P properties}\leavevmode
\begin{SL}
\item $P^+ + P^- $ represents the projection from $L^2$ onto
$L^2_{\text{rad}}$.

\item Fix $N>0$.  Then
$$
\bigl\| \chi_{\gtrsim\frac 1N} P^{\pm}_{\geq N} f\bigr\|_{L^2(\R^d)}
\lesssim \bigl\| f \bigr\|_{L^2(\R^d)}
$$
with an $N$-independent constant.

\item If the dimension $d=2$, then the $P^{\pm}$ are bounded on $L^2(\R^2)$.
\item For $|x|\gtrsim N^{-1}$ and $t\gtrsim N^{-2}$, the
integral kernel obeys
\begin{equation*}
\bigl| [P^\pm_N e^{\mp it\Delta}](x,y) \bigr| \lesssim \begin{cases}
    (|x||y|)^{-\frac {d-1}2}|t|^{-\frac 12}  &: \  |y|-|x|\sim  Nt \\[1ex]
     \frac{N^d}{(N|x|)^{\frac{d-1}2}\langle N|y|\rangle^{\frac{d-1}2}}
     \bigl\langle N^2t + N|x| - N|y| \bigr\rangle^{-m}
            &: \  \text{otherwise}\end{cases}
\end{equation*}
for all $m\geq 0$.

\item For $|x|\gtrsim N^{-1}$ and $|t|\lesssim N^{-2}$, the
integral kernel obeys
\begin{equation*}
\bigl| [P^\pm_N e^{\mp it\Delta}](x,y) \bigr|
    \lesssim   \frac{N^d}{(N|x|)^{\frac{d-1}2}\langle N|y|\rangle^{\frac{d-1}2}}
     \bigl\langle N|x| - N|y| \bigr\rangle^{-m}
\end{equation*}
for any $m\geq 0$.
\end{SL}
\end{prop}

We will also need the following Proposition concerning the
properties of $P^{\pm}$ in the small $x$ regime (i.e. $|x| \lsm
N^{-1}$) where Bessel functions have logarithmic singularities. More
precisely, we have
\begin{prop}[Properties of $P^\pm$, small $x$ regime] \label{Psmall_properties}
\

Let the dimension $d=2$.
 \begin{SL}
 \item For $t \gtrsim N^{-2}$, $N^{-3} \lsm |x| \lsm N^{-1}$, $|y|\ll Nt$  or $|y| \gg Nt$, the integral
kernel satisfies
\begin{align*}
 \bigl|  [P^\pm_N e^{\mp it\Delta}](x,y) \bigr| \lsm \frac {N^2 \log N}{\langle N|y| \rangle^{1/2} }
\langle N^2 t + N|y| \rangle^{-m}, \quad \forall\, m\ge 0.
\end{align*}

\item For $t \gtrsim N^{-2}$, $N^{-3} \lsm |x| \lsm N^{-1}$, $|y|\sim Nt$, the integral
kernel satisfies
\begin{align*}
 \bigl|  [P^\pm_N e^{\mp it\Delta}](x,y) \bigr| \lsm \frac {N^2 \log N}{\langle N|y| \rangle^{1/2} }.
\end{align*}
 \end{SL}
\end{prop}

\begin{proof}
 We shall only provide the proof for $P_N^{+} e^{-it \Delta}$ since the other kernel is its complex conjugate.
The first claim is an exercise in stationary phase. By definition we have the following formula for the kernel
\begin{align}\label{kernel}
[P^+_N e^{-it\Delta}](x,y) = \tfrac12 (2\pi)^{-2}\int_{\R^2} H^{(1)}_0( |\xi| |x|) e^{ it|\xi|^2} J_0\bigl(|\xi||y|\bigr)
    \psi\bigl(\tfrac{\xi}N\bigr)\,d\xi,
\end{align}
and for the Hankel function $H_0^{(1)}$ we have
\begin{align} \label{eq_tmp911_1}
H_0^{(1)}(r)= J_0(r)+ i Y_0(r).
\end{align}
Observe that in the regime $|\xi | \sim N$, $|x| \lsm N^{-1}$, we have $r=|\xi|\cdot|x| \lsm 1$. Since
\begin{align*}
 J_0(r) = \sum_{m=0}^{\infty} \frac {(-1)^m} {(m!)^2} \left( \frac r 2 \right)^{2m},
\end{align*}
it is easy to see that
\begin{align*}
 \left| \frac {\partial^m J_0(r)}{\partial r^m} \right| \lsm 1, \quad \forall\, m\ge 0, \, r\lsm 1.
\end{align*}
For $Y_0$ we have
\begin{align*}
 Y_0(r) = \frac 2 {\pi} \bigl( \log (\tfrac 12 r) + \gamma \bigr) J_0(r) + \frac 2 {\pi} \sum_{k=1}^\infty
(-1)^{k+1} H_k \cdot \frac {(\frac 14 r^2)^k}{ (k!)^2},
\end{align*}
where $\gamma$ is the Euler-Masheroni constant and $H_k$ is a harmonic number ($H_k = \sum_{n=1}^k \frac 1n$).
Clearly we can then write
\begin{align} \label{eq_tmp911_2}
 Y_0(r)= \frac 2 {\pi} ( \log r) \cdot J_0(r) +b(r),
\end{align}
where $b(r)$ obeys
\begin{align*}
 \left| \frac {\partial^m b(r) }{\partial r^m} \right| \lsm 1, \quad \forall\, m\ge 0, \, r\lsm 1.
\end{align*}
We also need the following information about Bessel functions
\begin{align} \label{eq_tmp911_3}
 J_0(r) = \frac{a(r) e^{ir}}{\langle r\rangle^{1/2}} + \frac{\bar a(r) e^{-ir}}{\langle r\rangle^{1/2}},
\end{align}
where $a(r)$ obeys the estimates
\begin{align*}
\Bigr| \frac{\partial^m a(r)}{\partial r^m} \Bigr| \lesssim \langle r \rangle^{-m}
    \quad \text{for all $m\geq0$.}
\end{align*}
Substitute \eqref{eq_tmp911_1}, \eqref{eq_tmp911_2}, \eqref{eq_tmp911_3} into \eqref{kernel}, we
obtain that a stationary point can only occur at $|y| \sim Nt$. Observe that
$\log (|\xi| \cdot |x|) = \log |\xi| + \log |x|$ and $\log |x|$ can be taken outside
of the integral for $\xi$, the logarithmetic singularity of $Y_0$ is ok for us. Now since we
assume $|y| \ll Nt$ or $|y| \gg Nt$, the desired claim follows by integrating by parts.
This establishes the first claim. Finally the second claim follows easily from a $L^1$ estimate.
\end{proof}

\begin{rem}
In Proposition \ref{Psmall_properties}, we set the lower regime of $|x|$ to be $N^{-3}$ only for
simplicity of presentation. The same result holds if one changes $N^{-3}$ to be $N^{-\alpha}$ where
$\alpha\ge 1$ (with the implied constant depending on $\alpha$).
\end{rem}

\section{A non-sharp decomposition}
In this section we establish the following non-sharp decomposition
for $H_x^1$ functions with ground state mass.
\begin{prop}[Non-sharp decomposition of $H_x^1$ functions with ground state mass]
\label{prop_nonsharp}
 Fix the dimension $d\ge 1$. Let $Q$ be the ground state defined in \eqref{eq_R}. There
exist constants $C_1>0$, $C_2>0$ which depend only on the dimension
$d$ such that the following holds: for any $u\in H_x^1(\R^d)$ (not
necessarily radial) with $M(u)=M(Q)$, there exists $x_0 = x_0(u) \in
\R^d$, $\theta_0 = \theta_0(u) \in \R$, $\epsilon = \epsilon(u) \in
H_x^1(\R^d)$ for which we have
\begin{align*}
 u = \lambda^{\frac d2} e^{i\theta_0} Q( \lambda (\cdot - x_0)) + \epsilon,
\end{align*}
where
\begin{align} \label{3_39c}
          \frac 1{C_2}  \cdot \frac {\| \nabla u\|_{L_x^2}} {\| \nabla Q\|_{L_x^2}}
\le \lambda \le C_2 \cdot \frac {\| \nabla u\|_{L_x^2}} {\| \nabla
Q\|_{L_x^2}}, \qquad \text{if $ {\|\nabla u\|_{L_x^2}^2}  \ge C_1
E(u) $},
\end{align}
and
\begin{align*}
 \lambda =1, \qquad \text{if $ {\|\nabla u\|_{L_x^2}^2}  < C_1 E(u) $}.
\end{align*}
The term $\epsilon$ satisfies the bound:
\begin{align} \label{3_39d}
 \| \epsilon \|_{H_x^1} \lesssim_d \sqrt{E(u)} +1.
\end{align}
Here the energy $E(u)$ is the same as defined in \eqref{eq_Eudef}.
If in addition $u$ is an even function (i.e. $u(x)=u(-x)$ for any
$x\in \R^d$), then we can take $x_0 =0$.
\end{prop}
The next lemma is essentially a corollary of Proposition
\ref{prop_nonsharp}. For the convenience of presentation, we
postpone the proofs of both Proposition \ref{prop_nonsharp} and
Lemma \ref{lem_uni_bdd} till the end of this section.
\begin{lem}[Uniform $H_x^1$ boundedness away from the origin] \label{lem_uni_bdd}
 Fix the dimension $d\ge 1$. Let $Q$ be the ground state in \eqref{eq_R}. Let
$c\ge c_0>0$, $E_0 \ge 0$ be given numbers. Then for any even
function $u \in H_x^1(\R^d)$ with $M(u)=M(Q)$, $E(u)=E_0$, we have
\begin{align} \label{3_40a}
 \| \phi_{>c} \nabla u \|_{L_x^2(\R^d)} + \| \phi_{>c} u \|_{H_x^1(\R^d)} \lesssim_{c_0, E_0,d} 1.
\end{align}
\end{lem}
To establish Proposition \ref{prop_nonsharp} and Lemma
\ref{lem_uni_bdd}, we will prepare some elementary lemmas and some
general discussions. To this end, consider the $d$ dimensional
focusing NLS of the form
\begin{align*}
 i \partial_t u + \Delta u+ |u|^p u=0,
\end{align*}
where we assume $p \le \frac 4 d$. Let $Q$ be the associated ground
state which is a positive radial Schwartz function. The linearized
operators are defined by
\begin{align*}
 L_+ &= -\Delta +1 - (p+1) Q^p, \\
L_- &= -\Delta +1 -Q^p.
\end{align*}
Set $w=u+iv$ and
\begin{align*}
 L =\begin{pmatrix}
     &0 \quad &L_-\\
     &-L_+ \quad &0
    \end{pmatrix}.
\end{align*}

We need the following fact from Weinstein \cite{W2}.
\begin{prop}[Conditional positivity of $L_+$, \cite{W2}]  \label{prop_orPos}
Let $p\le \frac 4d$. Then
 \begin{align*}
 \inf_{(f,Q)=0} (L_+f,f) =0.
\end{align*}
\end{prop}
\begin{proof}
 See Proposition 2.7 of \cite{W2}.
\end{proof}

\begin{lem}[Spectral properties of $L$, \cite{W3},\cite{W2},\cite{kwong}] \label{lem_Lspec}
\

Assume $p\le \frac 4 d$. Then
\begin{SL}
 \item $L_-$ is a nonnegative self-adjoint operator in $L^2(\R^d)$ with the null space $N(L_-)=\text{span}\{Q\}$.
\item $L_+$ is a self-adjoint operator in $L^2(\R^d)$ with null space $N(L_+)= \text{span}\{Q_{x_i}:\, 1\le i\le d\}$.
\item $L_+$ has exactly one negative eigenvalue and its multiplicity is $1$.
\end{SL}
\end{lem}
\begin{proof}
Claim (i), (ii) and first part of Claim (iii) follows directly from
\cite{W2}, \cite{kwong} and \cite{W3}. The second part in Claim
(iii) concerning the multiplicity of the negative eigenvalue is a
simple consequence of Proposition \ref{prop_orPos}.
\end{proof}

By Lemma \ref{lem_Lspec} we can denote the negative eigenvalue of
$L_+$ as $-\lambda_0$ ($\lambda_0>0$) and the corresponding
eigenfunction with unit $L^2$ mass as $W$.

\begin{lem}[Coercivity of $L$, \cite{W2}, \cite{MM01}] \label{lem_coer}
\

Let $p=\frac 4d$. There exist constants $\sigma_1$, $\sigma_2$
depending only on the dimension $d$ such that the following holds:
\begin{SL}
 \item For any $\epsilon \in H^1$, if $(\epsilon, W)=0$ and $(\epsilon,Q_{x_i})=0$ for any $1\le i\le d$, then
\begin{align*}
 (L_+\epsilon, \epsilon) \ge \sigma_1 (\epsilon,\epsilon).
\end{align*}

\item For any $\epsilon \in H^1$, if $(\epsilon, Q)=0$, then
\begin{align*}
 (L_-\epsilon, \epsilon) \ge \sigma_2 (\epsilon, \epsilon).
\end{align*}
\end{SL}
\end{lem}
\begin{proof}
 See \cite{W2} and the improvement in \cite{MM01} for the critical case $p=\frac 4d$.
\end{proof}

We need the following compactness lemma from Hmidi and Keraani
\cite{hmidi-keraani}.
\begin{lem}[Compactness lemma, \cite{hmidi-keraani}] \label{lem_compact_nonradial}
 Let $\{v_n\}_{n=1}^\infty$ be a bounded family of $H^1(\R^d)$ such that
\begin{align*}
 \limsup_{n\to \infty} \|\nabla v_n \|_{L_x^2} \le M \quad \text{and}\quad
\limsup_{n\to \infty} \|v_n\|_{L_x^{\frac{2(d+2)}d}} \ge m.
\end{align*}
 Then there exists $\{x_n\}_{n=1}^\infty \subset \R^d$ such that, up to a subsequence
\begin{align*}
 v_n ( \cdot + x_n) \rightharpoonup V \quad \text{weakly},
\end{align*}
with
\begin{align*}
 \| V\|_{L_x^2} \ge \left( \frac d {d+2} \right)^{\frac d 4} \cdot \frac {m^{\frac d2+1}}{M^{\frac d2}}
\|Q\|_{L_x^2}.
\end{align*}
\end{lem}
\begin{proof}
 See Theorem 1 of \cite{hmidi-keraani}.
\end{proof}

\begin{lem}[Rigidity of the ground state, non-quantitative version] \label{lem_nonquant}
 Let the dimension $d\ge 1$ and $Q$ be the ground state corresponding to $p=4/d$. Assume
$u\in H^1(\R^d)$ satisfies
\begin{align} \label{eq_tmp_nqt_0}
 \| u\|_{L_x^2} = \|Q\|_{L_x^2} \quad\text{and} \quad \|\nabla u\|_{L_x^2} = \|\nabla Q\|_{L_x^2}.
\end{align}
Then the energy $E(u)\ge 0$ and there exist $\gamma_0 = \gamma_0(u)
\in \R$, $x_0 = x_0(u) \in \R^d$ such that
\begin{align} \label{eq_tmp_nqt_1}
 \| Q- e^{i\gamma_0} u(\cdot +x_0) \|_{H^1(\R^d)} \le \delta (E(u)),
\end{align}
where $\delta(E(u)) \to 0$ if $E(u) \to 0$.
\end{lem}
\begin{proof}
 That $E(u)$ is nonnegative follows easily follow the sharp Gagliardo-Nirenberg interpolation
inequality. Assume \eqref{eq_tmp_nqt_1} is false. Then there exists
$\epsilon_0 >0$ and a sequence $\{u_n\}_{n=1}^\infty \subset
H^1(\R^d)$ satisfying \eqref{eq_tmp_nqt_0} such that
\begin{align} \label{eq_tmp_nqt_2}
 \inf_{\substack{\gamma \in \R \\ x \in \R^d}} \| Q- e^{i\gamma} u_n (\cdot +x) \|_{H^1(\R^d)} \ge  \epsilon_0,
\end{align}
and $E(u_n) \to 0$ as $n\to \infty$. By \eqref{eq_tmp_nqt_0} and the
fact $E(u_n) \to 0$, we obtain
\begin{align*}
 \limsup_{n\to \infty} \| u_n \|_{L_x^{\frac{2(d+2)} d}} &
\ge \| \nabla Q\|_{L_x^2}^{\frac d{d+2}} \cdot \left( \frac{d+2} d \right)^{\frac d {2(d+2)}} \\
& = \| Q\|_{L_x^{\frac {2(d+2)} d}}.
\end{align*}
By Lemma \ref{lem_compact_nonradial}, there exists
$\{x_n\}_{n=1}^\infty \subset \R^d$ such  that up to a subsequence
\begin{align*}
 u_n (\cdot + x_n) \rightharpoonup V \quad \text{weakly},
\end{align*}
and
\begin{align*}
 \| V\|_{L_x^2} \ge \| Q \|_{L_x^2}.
\end{align*}
It follows that $u_n( \cdot +x_n) \to V$ strongly in $L_x^2$ and
also in $L_x^{\frac {2(d+2)} d}$ by interpolation. We then obtain
$\|V\|_{L_x^2} = \|Q\|_{L_x^2}$, $\|V\|_{L_x^{\frac{2(d+2)} d}} =
\|Q\|_{L_x^{\frac{2(d+2)}d}}$ and $\|\nabla V \|_{L_x^2} \le
\|\nabla Q\|_{L_x^2}$. By the sharp Gagliardo-Nirenberg
interpolation inequality, we conclude that $V= e^{-i \gamma_0}
Q(\cdot - x_0)$ for some $\gamma_0 \in \R$, $x_0 \in \R^d$. It then
follows that
\begin{align*}
 e^{i\gamma_0} u_n ( \cdot + x_n+x_0) \to Q\quad \text{strongly in $H^1$},
\end{align*}
which is an obvious contradiction to \eqref{eq_tmp_nqt_2}. The lemma
is proved.
\end{proof}

The next lemma is central to our non-sharp decomposition.
\begin{lem}[Rigidity of the ground state, quantitative version] \label{lem_modulate}
Let $d\ge 1$ be the dimension, $p=4/d$ and $Q$ be the corresponding
ground state. There exist constants $\eta>0$, $C> 1$, $K>0$
depending only on the dimension $d$ such that the following is true:

Let $u\in H^1(\R^d)$ be such that
\begin{align*}
\|u\|_{L_x^2} =\|Q\|_{L_x^2}, \quad \|\nabla u\|_{L_x^2} = \|\nabla
Q\|_{L_x^2},
\end{align*}
and
\begin{align*}
0 \le E(u) \le \eta.
\end{align*}
(The condition $E(u)\ge 0$ is actually unnecessary by the sharp
Gagliardo-Nirenberg inequality). Then there exist $\gamma_0 =
\gamma_0(u) \in \R$, $x_0 = x_0(u) \in \R^d$, $\lambda_0
=\lambda_0(u)>0$ with
\begin{align} \label{eq_modulate_10_1}
\frac 1 C \le \lambda_0 \le C
\end{align}
such that
\begin{align} \label{eq_modulate_10_2}
\epsilon (x) = e^{i\gamma_0} \lambda_0^{\frac d2} u(\lambda_0 x+x_0)
- Q(x)
\end{align}
satisfies the following:
\begin{enumerate}
\item $Re(\epsilon)$ is orthogonal to the negative and neutral directions of $L_+$:
\begin{align} \label{eq_modulate_10_3}
(Re(\epsilon), W) =0 \quad \text{and}\quad (Re(\epsilon), Q_{x_j} )
=0,\quad  \text{for any $1\le j\le d$}.
\end{align}
Here $W$ is the eigenfunction corresponding to the negative
eigenvalue of the linear operator $L_+$ (see Lemma \ref{lem_Lspec}
and Lemma \ref{lem_coer}).

\item $Im(\epsilon)$ is orthogonal to the neutral directions of $L_-$:
\begin{align} \label{eq_modulate_10_4}
   (Im(\epsilon), Q)=0.
      \end{align}

\item The $H^1$ norm of $\epsilon$ is small, more precisely:
\begin{align} \label{eq_modulate_10_55}
\| \epsilon\|_{H^1} \le K \sqrt {E(u)}.
\end{align}

\item If $u$ is an even function of $x$ (i.e. $u(x)=u(-x)$ for all $x\in \R^d$), then we can take
$x_0=x_0(u)=0$.

\end{enumerate}
\end{lem}

\begin{proof}

\texttt{Step 1}: We show that Claim (1) and (2) holds. As in
\cite{merle-raph}, the idea is to use the implicit function theorem.
For any $\alpha>0$, define the neighborhood
\begin{align*}
 U_\alpha = \{ u \in H^1(\R^d):\quad \|u-Q\|_{H^1} < \alpha \}.
\end{align*}
For any $\gamma \in \R$, $x\in \R^d$, $u \in H^1(\R^d)$, define
\begin{align} \label{ep111}
 \epsilon_{\lambda, \gamma, x}(y) = e^{i\gamma} \lambda^{\frac d 2} u(\lambda y +x) - Q(y).
\end{align}
We first claim that there exist $\alpha_0 >0$ and a unique $C^1$
map: $U_{\alpha_0} \to \R^+_{\lambda} \times \R_\gamma \times
\R_x^d$ such that for any $u \in U_{\alpha_0}$, there exist unique
$\lambda=\lambda(u)>0$, $\gamma=\gamma(u) \in \R$, $x=x(u) \in \R^d$
satisfying the following properties:
\begin{enumerate}
 \item The real part of $\epsilon_{\lambda,\gamma,x}$ is orthogonal to the negative and neutral directions of $L_+$:
\begin{align} \label{eq_tmp32_1}
        (Re(\epsilon_{\lambda,\gamma,x}), W)=0, \quad \text{and} \quad (Re(\epsilon_{\lambda,\gamma,x}), Q_{x_i}) =0,
\quad \forall\, 1\le i\le d.
       \end{align}

\item The imaginary part of $\epsilon_{\lambda, \gamma,x}$ is orthogonal to the neutral direction of $L_-$:
\begin{align} \label{eq_tmp32_2}
 (Im(\epsilon_{\lambda, \gamma,x}), Q)=0.
\end{align}
\end{enumerate}
Furthermore there is a constant $K_1>0$ such that if $0<\alpha<\bar
\alpha$, $u\in U_\alpha$, then
\begin{align} \label{eq_tmp32_3}
 \| \epsilon_{\lambda,\gamma,x} \|_{H^1} + |\lambda-1| + |\gamma| + |x| \le K_1 \alpha.
\end{align}
To establish the above claims, we define the following functionals
\begin{align*}
&\rho_1(u,\lambda,\gamma,x) = \int_{\R^d} Re(\epsilon_{\lambda,\gamma,x}(y)) W(y) dy,\\
&\rho_j(u,\lambda,\gamma,x) = \int_{\R^d} Re(\epsilon_{\lambda,\gamma,x}(y)) Q_{x_{i-1}}(y) dy,\quad 2\le j\le d+1,\\
&\rho_{d+2}(u,\lambda,\gamma,x) = \int_{\R^d}
Im(\epsilon_{\lambda,\gamma,x}(y)) Q(y) dy.
\end{align*}
We then compute at $(\lambda, \gamma,x ) =(1,0,0)$,
\begin{align}
& \frac {\partial(\epsilon_{\lambda,\gamma,x})}{\partial \lambda} = \frac d 2 u +y\cdot \nabla u, \label{eq_a32_1} \\
& \frac{\partial(\epsilon_{\lambda,\gamma,x})}{\partial \gamma} = iu, \notag \\
& \frac {\partial( \epsilon_{\lambda,\gamma,x})} {\partial_{x_i}}
=u_{x_i}, \quad \forall\, 1\le i\le d. \label{eq_a32_1a}
\end{align}
To compute $\frac{\partial \rho_1}{\partial \lambda}$ at $(\lambda,
\gamma, x,u) = (1,0,0,Q)$, we use the algebraic identity:
\begin{align} \label{eq_a32_2}
L_+ \left( \frac d 2 Q + y \cdot \nabla Q \right) = -2Q,
\end{align}
which can be easily obtained by using scaling invariance. By
\eqref{eq_a32_1}, \eqref{eq_a32_2} and the fact that
\begin{align*}
L_+ W = -\lambda_0 W,
\end{align*}
we have
\begin{align*}
\left. \frac{\partial \rho_1} { \partial \lambda}
\right|_{(1,0,0,Q)} & = \int_{\R^d}
\left( \frac d2 Q+ y \cdot \nabla Q \right) \cdot W(y) dy \\
& = - \frac 1 {\lambda_0} \int_{\R^d} \left( \frac d 2 Q +y \cdot \nabla Q \right) \cdot L_+ W(y) dy \\
& = \frac 2 {\lambda_0} \int_{\R^d} Q(y) W(y) dy \ne 0,
\end{align*}
where the last integral does not vanish by Proposition
\ref{prop_orPos}. We then compute for $2\le i \le d+2$,
\begin{align*}
\left. \frac {\partial \rho_i} { \partial \lambda}
\right|_{(1,0,0,Q)} =0.
\end{align*}
It is easy to find that for $1\le i \le d+1$
\begin{align*}
\left. \frac {\partial \rho_i} {\partial \gamma} \right|_{(1,0,0,Q)}
=0,
\end{align*}
and also
\begin{align*}
\left. \frac{\partial \rho_{d+2}} {\partial \gamma}
\right|_{(1,0,0,Q)} = \| Q\|_{L_x^2}^2.
\end{align*}
Lastly by \eqref{eq_a32_1a}, we obtain
\begin{align*}
\left. \frac {\partial \rho_i } {\partial x_j} \right|_{(1,0,0,Q)} =
\delta_{j,i-1} \| Q_{x_j}\|_{L_x^2}^2, \quad \forall\, 2\le i\le
d+1, \; \text{and} \; 1\le j\le d,
\end{align*}
where $\delta_{j,i-1}$ is the usual Kronecker delta function. It
follows easily that the Jacobian
\begin{align*}
&\left| \frac{\partial(\rho_1, \rho_{d+2},\rho_2, \cdots,
\rho_{d+1})}{ \partial (\lambda,\gamma,
x_1,\cdots, x_d)} \right|_{(1,0,0,Q)} \\
= & \frac 2 {\lambda_0} \left( \int_{\R^d} Q(y) W(y) dy \right)
\cdot \|Q \|_{L_x^2}^2
\cdot \prod_{j=1}^d \| Q_{x_j} \|_{L_x^2}^2 \\
\ne & 0.
\end{align*}
Therefore by the implicit function theorem, there exist $\alpha_0
>0$, a neighborhood of $(\lambda,\gamma,x)=(1,0,0)$ in
$\R_\lambda^{+} \times \R_\gamma \times \R^d_x$, and a unique $C^1$
map: $U_{\alpha_0} \to \R_\lambda^+ \times \R_\gamma \times \R^d_x$
such that \eqref{eq_tmp32_1}, \eqref{eq_tmp32_2} holds. Furthermore
by Lemma \ref{lem_nonquant} and choosing $\eta$ sufficiently small,
we have \eqref{eq_tmp32_3} holds. Finally Claim (1) and (2) follows
easily from applying the previous result and Lemma
\ref{lem_nonquant}.

\texttt{Step 2}: We show that Claim (3) holds. By step 1, we can
choose $\lambda_0>0$, $x_0 \in \R^d$, $\gamma_0 \in \R$ such that
\eqref{eq_modulate_10_1}-\eqref{eq_modulate_10_4} holds. By Lemma
\ref{lem_coer} and denoting $\sigma=\min\{\sigma_1,\, \sigma_2\}$,
we obtain
\begin{align}
\sigma \|\epsilon \|_{L_x^2}^2 & \le (L_+ Re(\epsilon), Re(\epsilon)) + (L_- Im(\epsilon), Im(\epsilon)) \notag \\
& = \| \epsilon\|_{H^1}^2 - \frac {d+4} d (Q^{\frac 4d}
Re(\epsilon), Re(\epsilon)) -(Q^{\frac 4d} Im(\epsilon),
Im(\epsilon)). \label{eq_Nov14_1}
\end{align}
By conservation of mass, we have
\begin{align} \label{eq_Nov14_2}
\| \epsilon \|_{L_x^2}^2 + 2 (Re(\epsilon), Q) =0.
\end{align}
From conservation of energy, we get
\begin{align*}
\lambda_0^2 E(u) & = \frac 12 \| \nabla Q+ \nabla \epsilon \|^2_{L_x^2} - \frac d {2(d+2)} \| Q+ \epsilon\|^{\frac{2(d+2)}d}_{L_x^{\frac {2(d+2)} d}} \\
& = \frac 12 \| \nabla Q\|_{L_x^2}^2 - (\Delta Q, Re(\epsilon)) + \frac 12 \|\nabla \epsilon \|_{L_x^2}^2 \\
& \qquad -\frac d {2(d+2)} \| Q+ \epsilon \|_{L_x^{\frac{2(d+2)}
d}}^{\frac {2(d+2)} d}.
\end{align*}
Since the ground state $Q$ satisfies
\begin{align*}
-\Delta Q = Q^{1+\frac 4d} -Q,
\end{align*}
and by the sharp Gagliardo-Nirenberg inequality
\begin{align*}
\| \nabla Q\|^2_{L_x^2} = \frac d {d+2} \| Q \|^{\frac{2(d+2)}
d}_{L_x^{\frac{2(d+2)} d}},
\end{align*}
we obtain
\begin{align}
2 \lambda_0^2 E(u) & = \frac d{d+2} \left( \|
Q\|^{\frac{2(d+2)}d}_{L_x^{\frac{2(d+2)}d}}
-\|Q+\epsilon \|^{\frac{2(d+2)}d}_{L_x^{\frac{2(d+2)}d}} \right)+ \| \nabla \epsilon \|^2_{L_x^2} \notag \\
& \qquad - 2 (Q, Re(\epsilon)) + 2 (Q^{1+\frac 4d}, Re(\epsilon)).
\label{eq_Nov14_3}
\end{align}
Adding together \eqref{eq_Nov14_2} and \eqref{eq_Nov14_3}, we get
\begin{align}
2 \lambda_0^2 E(u) & = \frac d{d+2} \left( \|
Q\|^{\frac{2(d+2)}d}_{L_x^{\frac{2(d+2)}d}}
-\|Q+\epsilon \|^{\frac{2(d+2)}d}_{L_x^{\frac{2(d+2)}d}} \right)+ \| \epsilon \|^2_{H^1} \notag \\
& \qquad  + 2 (Q^{1+\frac 4d}, Re(\epsilon)). \label{eq_Nov14_4}
\end{align}
Substitute \eqref{eq_Nov14_4} into RHS of \eqref{eq_Nov14_1}, we
have
\begin{align} \label{eq_Nov14_5}
\sigma \| \epsilon \|_{L_x^2}^2 \le (L\epsilon, \epsilon) \le 2
\lambda_0^2 E(u) + \frac d{d+2} \int_{\R^d} |F(\epsilon)(y)| dy,
\end{align}
where
\begin{align*}
 F(\epsilon) &= |Q +\epsilon|^{\frac{2(d+2)}d} - Q^{\frac{2(d+2)} d} - \frac {d+2} d Q^{\frac 4d} (Im(\epsilon))^2 \\
& \qquad - \frac {2(d+2)} d Q^{1+\frac 4d} (Re(\epsilon)) - \frac
{d+2} d \cdot \frac {d+4} d \cdot Q^{\frac 4d} (Re(\epsilon))^2.
\end{align*}
At this moment, we need the following lemma which gives a bound of
$F(\epsilon)$.
\begin{lem} \label{bdd_Feps}
 There exists some constant $B_1>0$ depending only on the dimension $d$ such that
\begin{align} \label{eq_bddFeps_1}
 \int_{\R^d} |F(\epsilon)(y) |d y \le B_1 \cdot \| \epsilon \|^{\frac{2(d+2)} d}_{L_x^{\frac{2(d+2)} d}},
\quad \text{if $d\ge 4$},
\end{align}
and
\begin{align} \label{eq_bddFeps_1a}
 \int_{\R^d} |F(\epsilon)(y)|d y \le B_1 \cdot \left( \| \epsilon \|^{\frac{2(d+2)} d}_{L_x^{\frac{2(d+2)} d}}
+ \| \epsilon\|^3_{L_x^3} \right), \quad \text{if $1\le d\le 3$}.
\end{align}
\end{lem}
Postponing the proof of Lemma \ref{bdd_Feps} for the moment, we now
show how to finish the proof of step 2. Let first $d\ge 4$. Then by
Lemma \ref{bdd_Feps} and \eqref{eq_Nov14_5}, we obtain
\begin{align}
 \sigma \| \epsilon \|_{L_x^2}^2 & \le ( L\epsilon, \epsilon) \notag \\
& \le 2 \lambda_0^2 E(u) + B_1 \cdot \| \epsilon
\|_{L_x^{\frac{2(d+2)} d}}^{\frac{2(d+2)} d}. \label{eq_Nov14_6}
\end{align}
By definition of $L_+$, $L_-$, we have
\begin{align*}
 \| \epsilon \|^2_{H^1} & \le (L\epsilon,\epsilon)+ \left( \frac 4d+1 \right) \int_{\R^d} Q^{\frac 4d} Re(\epsilon)^2 dy
+ \int_{\R^d} Q^{\frac 4d} Im(\epsilon)^2 dy \\
& \le C(d) \cdot \| \epsilon \|_{L_x^2}^2 + 2\lambda_0^2 E(u) +B_1 \cdot \| \epsilon \|_{L_x^{\frac{2(d+2)} d}}^{\frac{2(d+2)} d}\\
& \le \lambda_0^2 \cdot \left( \frac{2 C(d)} {\sigma} +1 \right)
E(u) +B_1 \cdot \left( \frac{C(d)}{\sigma} +1 \right) \cdot \|
\epsilon \|_{L_x^{\frac{2(d+2)} d}}^{\frac{2(d+2)} d},
\end{align*}
where $C(d)$ is a constant depending on the dimension $d$, and the
last inequality follows from \eqref{eq_Nov14_6} together with
Sobolev embedding. Clearly now \eqref{eq_modulate_10_55} follows by
taking $\eta$ sufficiently small and using Lemma \ref{lem_nonquant}.
This finishes the proof of Claim 3 in the case $d\ge 4$. The case
$1\le d\le 3$ is similar by using Lemma \ref{bdd_Feps}. We omit the
details.

\texttt{Step 3}: We show that claim (4) holds. If $u$ is even, then
instead of \eqref{ep111} we define
\begin{align*}
 \epsilon_{\lambda,\gamma}(y)=e^{i\gamma} \lambda^{\frac d2} u(\lambda y) -Q(y).
\end{align*}
Since $u$ is even, by direct computation we have
$Re(\epsilon_{\lambda,\gamma})$ is orthogonal to the neutral
directions of $L_+$:
\begin{align*}
 (Re(\epsilon_{\lambda,\gamma}),\, Q_{x_i}) =0, \quad\forall\, 1\le i\le d.
\end{align*}
Therefore we only need to take care of two directions: the negative
direction of $L_+$ and the neutral direction of $L_-$. Similar to
Step 1, we define two functionals
\begin{align*}
 &\rho_1(u,\lambda,\gamma) = \int_{\R^d} Re(\epsilon_{\lambda,\gamma}(y)) W(y) dy, \\
& \rho_2(u,\lambda,\gamma) = \int_{\R^d}
Im(\epsilon_{\lambda,\gamma}(y)) Q(y) dy.
\end{align*}
By same calculations, we find that the Jacobian
\begin{align*}
 & \left |\frac{\partial(\rho_1,\rho_2)}{\partial(\lambda,\gamma)} \right|_{(\lambda,\gamma,u)=(1,0,Q)} \\
= & \frac 2 {\lambda_0} \int_{\R^d} W(y) Q(y) dy \cdot \|
Q\|^2_{L_x^2} \ne 0.
\end{align*}
The rest of the argument now follows essentially the same as in Step
1. We omit the details. This finishes the proof of Claim (4).
\end{proof}
Next we complete the
\begin{proof}[Proof of Lemma \ref{bdd_Feps}]
 We separate the integral on the LHS of \eqref{eq_bddFeps_1} into two regions. In the first region
$\{x:\; |\epsilon (x) | \gtrsim Q(x) \}$, we have the pointwise
estimate
\begin{align*}
 |F(\epsilon)(x) | \lsm | \epsilon (x) |^{\frac{2(d+2)} d},
\end{align*}
and therefore
\begin{align} \label{eq_bddFeps_2}
 \int_{|\epsilon(x) | \gtrsim Q(x)} |F(\epsilon)(x) | dx \lsm \| \epsilon \|^{\frac{2(d+2)}d}_{L_x^{\frac{2(d+2)}d}},
\end{align}
which is clearly good for us.

In the second region $\{x:\; |\epsilon(x) | \ll Q(x) \}$, we can
write
\begin{align*}
 F(\epsilon)(x) = Q(x)^{\frac{2(d+2)} d} \cdot G(Q(x)^{-1} \epsilon (x)),
\end{align*}
where
\begin{align*}
 G(z) &= |1+z|^{\frac{2(d+2)}d} - 1 -\frac{d+2}d (Im(z))^2 - \frac{2(d+2)} d Re(z) \\
& \qquad - \frac{d+2}d \cdot \frac {d+4} d (Re(z))^2.
\end{align*}
In the region $|z| \ll 1$, by expanding the function
$|1+z|^{\frac{2(d+2)} d}$ into power series, it is not difficult to
check that
\begin{align*}
 |G(z) | \lsm |z|^3, \quad \text{if $|z| \ll 1$}.
\end{align*}
Therefore we obtain
\begin{align} \label{eq_bddFeps_4}
 \int_{|\epsilon(x)| \ll Q(x)} |F(\epsilon)(x) | dx
\lsm \int_{|\epsilon (x) | \ll Q(x)} Q(x)^{\frac{2(d+2)} d} \cdot
\left| Q(x)^{-1} \epsilon(x) \right|^3 dx.
\end{align}
Now we discuss two cases. If $d\ge 4$, then since $\frac{2(d+2)}d
\le 3$,
\begin{align*}
 | \text{RHS of \eqref{eq_bddFeps_4} }| & \lsm \int_{|\epsilon(x)| \ll Q(x)}
|Q(x)^{-1} \epsilon (x) |^{3-\frac{2(d+2)} d} \cdot |\epsilon (x) |^{\frac{2(d+2)} d} dx \\
& \lsm \| \epsilon \|^{\frac{2(d+2)} d}_{\frac{2(d+2)} d}.
\end{align*}
This together with \eqref{eq_bddFeps_2} gives \eqref{eq_bddFeps_1}.
Finally if $1\le d\le 3$, then we have
\begin{align*}
 | \text{RHS of \eqref{eq_bddFeps_4} }| & \lsm \int_{|\epsilon(x)| \ll Q(x)} Q(x)^{\frac 4d-1} |\epsilon(x)|^3 dx \\
& \lsm \| \epsilon \|_{L_x^3}^3,
\end{align*}
which together with \eqref{eq_bddFeps_2} gives
\eqref{eq_bddFeps_1a}. The lemma is proved.
\end{proof}

We now complete the
\begin{proof}[Proof of Proposition \ref{prop_nonsharp}]
Let $u \in H_x^1(\R^d)$ and $M(u) =M(Q)$. Define the rescaled
function $\tilde u = \mu^{\frac d2} u(\mu x)$ where $\mu =\|\nabla
Q\|_2 / \|\nabla u\|_2$. Clearly then
\begin{align*}
 \| \tilde u\|_2 =\| Q\|_2, \quad \text{and} \quad \| \nabla \tilde u\|_2 =\| \nabla Q\|_2,
\end{align*}
also
\begin{align*}
 E(\tilde u) = \frac{\| \nabla Q\|_2^2} {\| \nabla u\|_2^2} E(u).
\end{align*}
Now let the constants $\eta=\eta(d)>0$, $C=C(d)>1$, $K=K(d)>0$ be
the same as in Lemma \ref{lem_modulate}. Define $C_1=\|\nabla
Q\|_{2}^2 /\eta$. Then obviously if $C_1 E(u) \le \|\nabla u\|_2^2$,
then $E(\tilde u) \le \eta$. By Lemma \ref{lem_modulate}, we can
then write
\begin{align} \label{3_39a}
 \tilde \epsilon (x) = e^{i \tilde \gamma_0} {\tilde \lambda_0}^{\frac d2} \tilde u (\tilde \lambda_0 x +\tilde x_0)
-Q(x),
\end{align}
where $\frac 1 C \le \tilde \lambda_0 \le C$, and
\begin{align} \label{3_39b}
 \| \tilde \epsilon \|_{H_x^1} \le K \sqrt{E(\tilde u)}.
\end{align}
Expressing \eqref{3_39a}, \eqref{3_39b} in terms of the original
function $u$, we then obtain
\begin{align*}
 u(x) & = e^{- i \tilde \gamma_0} (\tilde \lambda_0 \mu)^{-\frac d2} Q(\frac{x-\mu \tilde x_0} {\tilde \lambda_0 \mu})
+e^{- i \tilde \gamma_0} (\tilde \lambda_0 \mu)^{-\frac d2} \tilde \epsilon (\frac{x-\mu \tilde x_0} {\tilde \lambda_0 \mu}) \\
& = e^{i\theta_0} \lambda^{\frac d2} Q(\lambda(x-x_0)) +\epsilon,
\end{align*}
where $\theta_0 = -\tilde \gamma_0$, $x_0 = \mu \tilde x_0$,
$\lambda= (\tilde \lambda_0 \mu)^{-1} = {\tilde \lambda_0}^{-1} \|
\nabla u\|_2/\|\nabla Q\|_2$, and
\begin{align*}
 \| \epsilon \|_{H_x^1} & \le (1+\lambda) K \sqrt{E(\tilde u)} \\
& \le \frac {1+\lambda} {\lambda} \cdot K \cdot \sqrt{E(u)} \\
& \le K \cdot \frac {C \cdot \|\nabla Q\|_2} {\sqrt{C_1}} + K \sqrt{E(u)} \\
& \lsm_d 1+\sqrt{E(u)}.
\end{align*}
Choose $C_2=C$, and it is clear that \eqref{3_39c} holds. Next if
$\|\nabla u\|_2^2 < C_1 E(u)$, then we just set $\theta_0 =0$,
$x_0=0$, $\lambda=1$ and $\epsilon = u-Q$. It is obvious that
\eqref{3_39d} holds in this case. Finally if $u$ is an even function
of $x$, then one can repeat the previous steps with the observation
that Lemma \ref{lem_modulate} holds in this case for $x_0=0$. We
omit the details.
\end{proof}

Finally we finish the
\begin{proof}[Proof of Lemma \ref{lem_uni_bdd}]
Let $u\in H_x^1(\R^d)$ be an even function of $x$ and $M(u)=M(Q)$,
$E(u)=E_0$. Let $C_1=C_1(d)>0$, $C_2=C_2(d)>0$ be the same constants
as in Proposition \ref{prop_nonsharp}. By Proposition
\ref{prop_nonsharp},  we can write
\begin{align} \label{3_39dd}
 u(x) = \lambda^{\frac d2} e^{i\theta_0} Q( \lambda x) + \epsilon (x),
\end{align}
where
\begin{align*}
          \frac 1{C_2}  \cdot \frac {\| \nabla u\|_{L_x^2}} {\| \nabla Q\|_{L_x^2}}
\le \lambda \le C_2 \cdot \frac {\| \nabla u\|_{L_x^2}} {\| \nabla
Q\|_{L_x^2}}, \qquad \text{if $ {\|\nabla u\|_{L_x^2}^2}  \ge C_1
E_0 $},
\end{align*}
and
\begin{align*}
 \lambda =1, \qquad \text{if $ {\|\nabla u\|_{L_x^2}^2}  < C_1 E_0 $},
\end{align*}
also
\begin{align} \label{3_39d1}
 \| \epsilon \|_{H_x^1} \lesssim_{d,E_0} 1.
\end{align}
Now if ${\|\nabla u\|_{L_x^2}^2}  < C_1 E_0 $, then \eqref{3_40a}
holds trivially in this case. If ${\|\nabla u\|_{L_x^2}^2}  \ge C_1
E_0 $, then due to the decomposition \eqref{3_39dd} and the bound
\eqref{3_39d1}, we only need to estimate the LHS of \eqref{3_40a}
assuming $\epsilon=0$ in \eqref{3_39dd}, that is
\begin{align*}
\text{LHS of \eqref{3_40a}} & \lsm_{c_0,d} \left(
\int_{|x|>\frac{c_0}{100}} \lambda^{d+2} |(\nabla Q)(\lambda x)|^2
dx \right)^{\frac 12}
+ \left( \int_{|x| > \frac{c_0}{100}} \lambda^d |Q(\lambda x)|^2 dx \right)^{\frac 12} \\
& \lsm_{c_0,d} 1 + (\lambda^2+1) \int_{|y|>\frac{\lambda c_0}{100}} (|\nabla Q(y)|^2+ |Q(y)|^2) dy \\
& \lsm_{c_0,d} 1 + \frac {\lambda^2+1}{\lambda^2} \int_{\R^d} |y|^2 (|\nabla Q(y)|^2+ |Q(y)|^2) dy \\
& \lsm_{c_0,d,E_0} 1,
\end{align*}
where we have used the fact that $Q$ is a Schwartz function. This
finishes the proof of the lemma.
\end{proof}

%
%
%
%
\section{Reduction of the proof}\label{main_body}

In this section, we first prove Theorem
\ref{main} by assuming the kinetic energy localization Theorem
\ref{T:kinetic_loc}, \ref{T:kine_hd}. Then we will explain how the
kinetic energy localization result Theorem \ref{T:kinetic_loc},
Theorem \ref{T:kine_hd} can be derived from two Propositions:
Proposition \ref{local} and Proposition \ref{fre_decay}. Therefore
our task is reduced to proving the two Propositions in all dimensions
$d\ge 2$ with admissable symmetry, which will be done in next
sections.

\vspace{0.2cm}

{\textbf{Proof of Theorem \ref{main}}
\begin{proof}
Let $u$ be the solution of \eqref{nls} satisfying the conditions in
Theorem \ref{main}, that is

$\cdot$ $M(u)=M(Q)$;\\

$\cdot$ $u$ is a global solution and does not scatter forward in
time in the sense that
\begin{align*}
\|u\|_{L_{t,x}^{\frac{2(d+2)}d}(\srd)}=\infty.
\end{align*}

$\cdot$ $u_0$ verifies the admissible symmetry.

Then our aim is to show that actually, the corresponding solution
has to be spherically symmetric and there exists $\lambda_0>0$,
$\theta_0\in\R$ such that
\begin{align*}
u(t,x)=\lambda_0^{-\frac d2} e^{i\theta_0}e^{it/\lambda_0^2}Q(\frac
x{\lambda_0}).
\end{align*}

From the sharp Gagliardo-Nirenberg inequality, we observe that the
energy of $u$ is always non-negative. If $E(u_0)=0$, the variational
characterization of the ground state Proposition \ref{P:variational}
together with the condition $M(u)=M(Q)$ yields that: $u_0$ is
spherically symmetric and there exists $\lambda_0>0,\theta_0\in\R$
such that
\begin{align*}
u_0(x)=\lambda_0^{-\frac d2}e^{i\theta_0}Q(\frac x{\lambda_0}).
\end{align*}
The uniqueness of the solution then implies that the corresponding
solution is spherically symmetric, moreover
\begin{align*}
u(t,x)=\lambda_0^{-\frac d2}e^{i\theta_0}e^{it/\lambda^2_0}Q(\frac
x{\lambda_0}).
\end{align*}

So we only need to consider the positive energy case and try to get
a contradiction in this case. 
First we can show the strong localization of the kinetic energy.
From the results in \cite{ktv:2d,kvz:blowup} for the radial case and
assuming the scattering conjecture holds for the splitting-spherical
symmetric case, we conclude that the ground state $M(Q)$ can be
identified as the minimal mass in all dimensions $d\ge 2$ with
admissible symmetry. Therefore $u$ satisfies the Duhamel formula
\eqref{duhamel}. Moreover, from Lemma \ref{lem_uni_bdd}, $u$
satisfies the weak localization of the kinetic energy since
$M(u)=M(Q)$ and $u\in H_x^1$. The strong localization of the kinetic
energy then follows immediately from Theorem \ref{T:kinetic_loc},
Theorem \ref{T:kine_hd}. That is, $\forall\ \eta>0$, there exists
$C(\eta)>0$ such that
\begin{align}
\|\phi_{>C(\eta)}\nabla u(t)\|_{\ltrd}<\eta.\label{sca_kine}
\end{align}

Now we get the contradiction from the truncated virial argument.

Let $\phi_{\le R}$ be the smooth cutoff function, we define the
truncated virial as
\begin{align*}
V_R(t):=\int {\phi_{\le R}}(x)|x|^2 |u(t,x)|^2 dx.
\end{align*}
Obviously,
\begin{align}
V_R(t)\lsm R^2,\; \ \forall\, t\in \R.\label{vr_bdd}
\end{align}
On the other hand, we compute the second derivative of $V_R(t)$,
this gives
\begin{align}
&\partial_{tt} V_R(t)= 8E(u)\notag\\
&+O\biggl(\int_{|x|>R}|\nabla u(t,x)|^2+|u(t,x)|^{\frac{2(d+2)}d}
dx+\frac 1{R^2}\int_{|x|>R}|u(t,x)|^{2}dx\biggr).\label{vir}
\end{align}
Since $E(u)>0$, from \eqref{sca_kine} and by taking $R$ large
enough, we have
\begin{align*}
\int_{|x|>R}|\nabla u(t,x)|^2dx+\frac 1{R^2}\int_{|x|>R}|u(t,x)|^2
dx\le \frac 1{100}E(u),\ \forall\; t\ge 0.
\end{align*}
Moreover, from the Gagliardo-Nirenberg inequality, we have
\begin{align*}
\int_{|x|>R} |u(t,x)|^{\frac{2(d+2)}d}dx&\lsm
\|\phi_{>R/2}u(t)\|_{L_x^{\frac{2(d+2)}d}(\R^d)}^{\frac{2(d+2)}d}\\
&\lsm \|\phi_{>R/2} u(t)\|_{\ltrd}^{\frac
4d}\|\nabla(\phi_{>R/2}u(t))\|_{\ltrd}^2\\
&\lsm \frac 1{100} E(u).
\end{align*}
Therefore \eqref{vir} finally gives
\begin{align*}
\partial_{tt} V_R(t)\ge 4 E(u)>0, \ \forall t\ge 0
\end{align*}
which is a contradiction to \eqref{vr_bdd}.

The proof of Theorem \ref{main} is then finished.
\end{proof}

Now let us explain how the kinetic energy localization property
Theorem \ref{T:kinetic_loc} and Theorem \ref{T:kine_hd} can be
derived from the following two

\begin{prop}(Frequency decay estimate)\label{fre_decay}
Let $d\ge 2$. Let $u(t,x)\in H_x^1(\R^d)$ be a global solution to
\eqref{nls} forward in time and have the admissible symmetry.
Assume also $u$ satisfy the Duhamel formula \eqref{duhamel} and is
weakly localized in $H_x^1(\R^d)$ in the following sense
\begin{align}
\|\phi_{\gtrsim 1} \nabla u(t)\|_{\ltrd}\lsm 1, \ \forall\, t\,\ge
0.\label{weak_com}
\end{align}
Then there exists $\eps=\eps(d)$ such that, for any dyadic number
$N\ge 1$, we have
\begin{align*}
\|\phi_{>1}\pn u(t)\|_{\ltrd}\lsm \|\pntd u_0\|_{\ltrd}+N^{-1-\eps}.
\end{align*}
\end{prop}

\begin{prop}(Spatial decay estimate)\label{local} Let $u$ satisfy
the same conditions as in Proposition \ref{fre_decay}. Let $N_1,
N_2$ be two dyadic numbers. Then there exist $\delta=\delta(d)$ and
$R_0=R_0(N_1,N_2,u)$ such that for any $N_1<N\le N_2$ and $R>R_0$,
\begin{align*}
\|\phi_{>R}\pn u(t)\|_{\ltrd}\lsm \|\phi_{>\sqrt
R}u_0\|_{\ltrd}+R^{-\delta(d)}, \ \forall\, t\ge 0.
\end{align*}
\end{prop}

The proofs of Proposition \ref{fre_decay}, \ref{local} will be presented in
later sections. Assuming the two propositions hold for the moment, we now prove the strong kinetic
localization. That is, for any $\eta>0$, there exists $C(\eta)>0$
such that
\begin{align}
\|\phi_{>C(\eta)}\nabla u(t)\|_{\ltrd}\le \eta, \ \forall t\ge
0.\label{our_aim}
\end{align}

Let $N_1(\eta), N_2(\eta)$ be dyadic numbers and $C(\eta)$ a large
constant to be specified later. We estimate the LHS of
\eqref{our_aim} by splitting it into low, medium and high
frequencies:
\begin{align*}
\|\phi_{>C(\eta)}\nabla u(t)\|_{\ltrd}&\lsm \|\phi_{>C(\eta)}\nabla
P_{\le N_1(\eta)}u(t)\|_{\ltrd}\\
&\qquad+\|\phi_{>C(\eta)}\nabla P_{N_1(\eta)<\cdot\le
N_2(\eta)}u(t)\|_{\ltrd} \\
&\qquad+\|\phi_{>C(\eta)}\nabla P_{>N_2(\eta)}
u(t)\|_{\ltrd}
\end{align*}
For the low frequencies, we simply discard the cutoff and use
Bernstein,
\begin{align}\label{low_est}
\|\phi_{>C(\eta)}\nabla P_{\le N_1(\eta)}u(t)\|_{\ltrd}\lsm
N_1(\eta)\|u(t)\|_{\ltrd}\lsm N_1(\eta).
\end{align}
To estimate the medium frequencies, we use Bernstein, mismatch
estimate Lemma \ref{L:mismatch_real} and Proposition \ref{local} to
obtain
\begin{align*}
&\qquad \|\phi_{>C(\eta)}\nabla P_{N_1(\eta)<\cdot\le N_2(\eta)}
u(t)\|_{\ltrd}\\
&\le \sum_{N_1(\eta)<N\le N_2(\eta)}\|\phi_{>C(\eta)}\nabla P_N
u(t)\|_{\ltrd}\\
&\le C(N_1(\eta), N_2(\eta))\max_{N_1(\eta)<N\le
N_2(\eta)}\|\phi_{>C(\eta)}\nabla \pn u(t)\|_{\ltrd}\\
&\le C(N_1(\eta), N_2(\eta))\max_{N_1(\eta)<N\le
N_2(\eta)}\biggl(\|\phi_{>C(\eta)}\nabla P_N\phi_{\le
\frac{C(\eta)}2}\pntd u(t)\|_{\ltrd} \\
&\qquad\qquad+\|\phi_{>C(\eta)}\nabla
P_N\phi_{>\frac{C(\eta)}2}\pntd u(t)\|_{\ltrd}\biggr)\\
&\le C(N_1(\eta), N_2(\eta))\biggl(\max_{N_1(\eta)<N\le N_2(\eta)}
N^{-1}C(\eta)^{-2}+N\|\phi_{\gtrsim\sqrt{C(\eta)}}u_0\|_{\ltrd}+NC(\eta)^{-\delta(d)}\biggr)\\
&\le C(N_1(\eta), N_2(\eta))\biggl(C(\eta)^{-2}+\|\phi_{\gtrsim\sqrt
{C(\eta)}}
u_0\|_{\ltrd}+C(\eta)^{-\delta(d)}\biggr)\\
&\le
C(N_1(\eta),N_2(\eta))\biggl(\|\phi_{\gtrsim\sqrt{C(\eta)}}u_0\|_{\ltrd}+C(\eta)^{-\delta}\biggr).
\end{align*}

For the high frequencies, we first use
\begin{align*}
\|\phi_{>C(\eta)}\nabla P_{>N_2(\eta)}u(t)\|_{\ltrd} \le
\|\nabla(\phi_{>C(\eta)}P_{>N_2(\eta)}u(t))\|_{\ltrd}+\frac
1{C(\eta)}\|u(t)\|_{\ltrd}
\end{align*}
to reduce matters to estimating $\|\nabla
(\phi_{>C(\eta)}P_{>N_2(\eta)}u(t))\|_{\ltrd}$, for which we use
dyadic decomposition, Bernstein, and mismatch estimate Lemma
\ref{L:mismatch_fre} and Proposition \ref{fre_decay} to estimate
\begin{align*}
&\qquad\|\nabla(\phi_{>C(\eta)}P_{>N_2(\eta)}u(t))\|_{\ltrd}^2\\
& \le \sum_{N>\frac{N_2(\eta)}4}\|\nabla
P_N(\phi_{>C(\eta)}P_{>N_2(\eta)}u(t))\|_{\ltrd}^2+\|\nabla P_{\le
\frac{N_2(\eta)}4}(\phi_{>C(\eta)}P_{>N_2(\eta)}u(t))\|_{\ltrd}^2\\
&\lsm \sum_{N>\frac{N_2(\eta)}4}\|\nabla
P_N(\phi_{>C(\eta)}P_{\frac N4<\cdot \le 4N} u(t))\|_{\ltrd}^2\\
&\qquad+\sum_{N>\frac{N_2(\eta)}4}\|\nabla
P_N(\phi_{>C(\eta)}(P_{\le \frac N4}+P_{>4N})u(t))\|_{\ltrd}^2\\
&\qquad
+N_2(\eta)\|P_{\le\frac{N_2(\eta)}2}(\phi_{>C(\eta)}P_{>N_2(\eta)}
u(t))\|_{\ltrd}^2\\
&\lsm \sum_{N\gtrsim{N_2(\eta)}}(N^2\|\pn
u_0\|_{\ltrd}^2+N^{-2\eps})
+\sum_{N>\frac{N_2(\eta)}4}N(C(\eta)N)^{-10}+C(\eta)^{-10}N_2(\eta)^{-9}\\
&\le \|P_{\gtrsim{N_2(\eta)}}\nabla
u_0\|_{\ltrd}^2+N_2(\eta)^{-2\eps}+C(\eta)^{-10}N_2(\eta)^{-9}.
\end{align*}
Therefore, the high frequencies give
\begin{align*}
\|\phi_{>C(\eta)}\nabla
P_{>N_2(\eta)}u(t)\|_{\ltrd}\lsm\|P_{\gtrsim{N_2(\eta)}}\nabla
u_0\|_{\ltrd}+N_2(\eta)^{-\eps}+C(\eta)^{-1}.
\end{align*}
Adding the estimates of three pieces together we obtain
\begin{align*}
\|\phi_{>C(\eta)}\nabla u(t)\|_{\ltrd}&\lsm
N_1(\eta)+C(N_1(\eta),N_2(\eta))(C(\eta)^{-\delta(d)}+\|\phi_{\gtrsim\sqrt{C(\eta)}}u_0\|_{\ltrd})\\
&+\|P_{\gtrsim{N_2(\eta)}}\nabla
u_0\|_{\ltrd}+N_2(\eta)^{-\eps}+C(\eta)^{-1}.
\end{align*}
Now first taking $N_1(\eta)$ sufficiently small, $N_2(\eta)$
sufficiently large depending on $\eta,u_0$, then choosing $C(\eta)$
sufficiently large depending on $\eta$, $N_1(\eta)$, $N_2(\eta)$,
$u_0$, we obtain
\begin{align*}
\|\phi_{>C(\eta)}\nabla u(t)\|_{\ltrd}\le \eta,
\end{align*}
as desired.

\section{Proof of Proposition \ref{fre_decay}, Proposition \ref{local} in 2,3
dimensions}\label{fre_23}

In this Section, we prove Proposition \ref{fre_decay}, \ref{local}
in two and three dimensions. We remind the readers that the only
information we need is that the solution verifies the admissible
symmetry and satisfies the Duhamel formula \eqref{duhamel}, the weak
compactness of the kinetic energy \eqref{weak_com}. In the proof
that follows, we will use these properties many times. We first give

\subsection{Proof of Proposition \ref{fre_decay} in 2,3 dimensions}

Let $u$ satisfy the conditions in Proposition \ref{fre_decay}, we
need to show $\forall N\ge 1$,
\begin{equation*}
\|\phi_{>1}\pn u(t)\|_{\ltrd}\lsm \|\pn u_0\|_{\ltrd}+N^{-1-\frac
1{4}}, \forall t\ge 0.
\end{equation*}
We first remark the following proof will also apply if we perturb
the cutoff function by $\phi_{>c}$ for a fixed constant $c$ and the
frequency projection $\pn$ by the fattened one $\pntd$.

We begin by projecting $u(t)$ onto incoming and outgoing waves. For
the incoming wave, we use Duhamel formula backward in time, for the
outgoing waves, we use Duhamel formula forward in time. We thus
write
\begin{align}
\phi_{>1}\pn u(t)&=\phi_{>1}\icm u(t)+\phi_{>1}\otg u(t)\notag\\
&=\phi_{>1}\icm\propt u_0\label{linear}\\
&\quad-i\phi_{>1}\int_0^t\icm
e^{i\tau\Delta}F(u(t-\tau))d\tau\label{in}\\
&\quad+i\phi_{>1}\int_0^{\infty}\otg \propgto
F(u(t+\tau))d\tau.\label{out}
\end{align}
The last integral should be understood in the weak $\ltrd$ sense. We
first control the linear term by Strichartz estimate:
\begin{equation*}
\eqref{linear}\lsm \|\pn u_0\|_{\ltrd}.
\end{equation*}

Now we look at the last two terms. Since the contribution from
\eqref{in} and \eqref{out} will be estimated in the same way, we
only give the details of the estimate of \eqref{out}. To proceed, we
first split it into different time pieces
\begin{align*}
\eqref{out}&=i\phi_{>1}\int_0^{\frac 1{N}}\otg\propgto
F(u(t+\tau))d\tau\\
&\quad+i\phi_{>1}\int_{\frac 1{N}}^1\otg\propgto
F(u(t+\tau))d\tau\\
&\quad+i\sum_{0\le k\le
\infty}\phi_{>1}\int_{2^k}^{2^{k+1}}\otg\propgto F(u(t+\tau))d\tau,
\end{align*}
then introducing cutoff functions in front of $F(u)$ to further
write \eqref{out} into
\begin{align}
\eqref{out}&=i\phi_{>1}\int_0^{\frac 1N}\otg \propgto\phi_{>\frac
12}F(u(t+\tau))d\tau.\label{spm}\\
&\quad+i\phi_{>1}\int_0^{\frac 1N}\otg \propgto\phi_{\le\frac
12}F(u(t+\tau))d\tau.\label{spt}\\
&\quad+i\phi_{>1}\int_{\frac 1N}^1\otg \propgto\phi_{>\frac{N\tau}
2}F(u(t+\tau))d\tau.\label{mpm}\\
&\quad+i\phi_{>1}\int_{\frac 1N}^1\otg
\propgto\phi_{\le\frac{N\tau}2}F(u(t+\tau))d\tau.\label{mpt}\\
&\quad+i\sum_{0\le k\le\infty}\phi_{>1}\int_{2^k}^{2^{k+1}}\otg
\propgto\phi_{>N2^{k-1}}F(u(t+\tau))d\tau.\label{lpm}\\
&\quad+i\sum_{0\le k\le \infty}\phi_{>1}\int_{2^k}^{2^{k+1}}\otg
\propgto\phi_{\le N2^{k-1}}F(u(t+\tau))d\tau.\label{lpt}
\end{align}
The remaining part of the proof is devoted to estimating these six
pieces.

\textit{Estimate of \eqref{spm}}:

For \eqref{spm}, we write $\phi_{>\frac 12}F(u)=\phi_{>\frac
12}F(u\phi_{>\frac 14})$, then use Strichartz estimate and mismatch
estimate Lemma \ref{L:mismatch_fre} to obtain,
\begin{align*}
\eqref{spm}&\lsm\|\pntd\phi_{>\frac 12}F(u\phi_{>\frac
14})\|_{L_{\tau}^1L_x^2([t,t+\frac 1N]\times\R^d)}\\
&\lsm \|\pntd\phi_{>\frac 12}P_{>\frac 18 N}F(u\phi_{>\frac
14})\|_{L_{\tau}^1L_x^2([t,t+\frac
1N]\times\R^d)}+\|\pntd\phi_{>\frac
12}P_{\le \frac 18N}F(u\phi_{>\frac 14})\|_{L_{\tau}^{1}L_x^2(\srd)}\\
&\lsm \frac 1N\|P_{>\frac 18 N}F(u\phi_{>\frac
14})\|_{L_t^{\infty}L_x^2(\srd)}+\frac 1N N^{-10}\|F(u\phi_{>\frac
14})\|_{L_t^{\infty}L_x^2(\srd)}\\
&\lsm N^{-\frac 54}\||\nabla|^{\frac 14}F(u\phi_{>\frac
14})\|_{L_t^{\infty}L_x^2(\srd)}+N^{-10}\|u\phi_{>\frac
14}\|_{L_t^{\infty}L_x^{\frac{2(d+4)}d}(\srd)}^{\frac{d+4}d}.
\end{align*}
Note in dimensions $d=2,3$, from Sobolev embedding, fractional chain
rule and \eqref{weak_com}, we have
\begin{align*}
\|u\phi_{>\frac 14}\|_{L_t^{\infty}L_x^{\frac{2(d+4)}d}(\srd)}&\lsm
\|\phi_{>\frac 14} u\|_{L_t^{\infty}H_x^1(\srd)}\lsm 1.\\
\||\nabla|^{\frac 14} F(u\phi_{>\frac
14})\|_{L_t^{\infty}L_x^2(\srd)}&\lsm \||\nabla|^{\frac
14}(\phi_{>\frac 14}u)\|_{L_t^{\infty}L_x^{4}(\srd)}\||\phi_{>\frac
14}u\|_{L_t^{\infty}L_x^{\frac
{16}{d}}(\srd)}^{\frac 4d}\\
&\lsm \|\phi_{>\frac 14} u\|_{L_t^{\infty}H_x^1(\srd)}^{1+\frac
4d}\lsm 1.
\end{align*}
We conclude the estimate of \eqref{spm} and obtain
\begin{equation}
\eqref{spm}\lsm N^{-\frac 54}. \label{est_spm}
\end{equation}

\textit{Estimate of \eqref{spt}}

For \eqref{spt}, since we do not have uniform control on $\|\nabla
u(t)\|_{\ltrd}$ inside the unit ball, we use the equation
\eqref{nls} to replace $F(u)$ by $(i\partial_t+\Delta )u$. We thus
need the following lemma:

\begin{lem}\label{replace}
For any $a,b\in\R$, and $C>0$, we have
\begin{align*}
&\qquad\int_a^b e^{-i\tau\Delta}(\phi_{\le
C}(i\partial_{\tau}+\Delta)u(t+\tau))d\tau\\
&=i e^{-ib\Delta}(\phi_{\le C}u(t+b))-ie^{-ia\Delta}(\phi_{\le
C}u(t+a)) \\
&\qquad
-\int_a^b \propgto(u\Delta \phi_{\le C}+2\nabla
u\cdot\nabla\phi_{\le C})(t+\tau)d\tau.
\end{align*}
\end{lem}
\begin{proof} The proof is a simple use of Fundamental Theory of Calculus.
\end{proof}

Using Lemma \ref{replace}, we bound \eqref{spt} as follows
\begin{align}
\|\eqref{spt}\|_{\ltrd}&=\|\phi_{>1}\int_0^{\frac
1N}\otg\propgto(\phi_{\le\frac
12}(i\partial_{\tau}+\Delta)u(t+\tau))d\tau\|_{\ltrd}\\
&\le \|\phi_{>1}\otg e^{-iN^{-1}\Delta}(\phi_{\le\frac 12}u(t+\frac
1N))\|_{\ltrd}\label{spt1}\\
&\quad+\|\phi_{>1}\otg(\phi_{\le\frac
12}u(t))\|_{\ltrd}\label{spt2}\\
&\quad+\|\phi_{>1}\int_0^{\frac 1N}\otg \propgto
(u(t+\tau)\Delta\phi_{\le\frac 12})d\tau\|_{\ltrd}\label{spt3}\\
&\quad+\|\phi_{>1}\int_0^{\frac 1N}\otg\propgto(\nabla\phi_{\le
\frac 12}\cdot\nabla u(t+\tau))d\tau\|_{\ltrd}\label{spt4}.
\end{align}
These four terms are going to be estimated in the same manner, so we
choose only to estimate \eqref{spt4} for the sake of simplicity.
From Proposition \ref{P:P properties}, the kernel obeys the estimate
\begin{align*}
|(\phi_{>1}\otg \propgto \chi_{\le \frac 12})(x,y)|&\lsm
N^d\langle N|x|-N|y|\rangle^{-m}\phi_{|x|>1}\chi_{|y|\le\frac 12}\\
&\lsm N^d\langle N(x-y)\rangle^{-m}\\
&\lsm N^d N^{-m/2}\langle
x-y\rangle^{-m/2}, \tau\in[0,\frac 1{N^2}],\\
|(\phi_{>1}\otg\propgto\chi_{\le\frac 12})(x,y)|&\lsm N^d\langle
N^2\tau+N|x|-N|y|\rangle^{-m}\phi_{|x|>1}\chi_{|y|\le\frac 12}\\
&\lsm N^d\langle N^2\tau+N|x|+N|y|\rangle^{-m}\phi_{|x|>1}\chi_{|y|\le\frac 12}\\
&\lsm N^d N^{-\frac m2}\langle x-y\rangle^{-\frac m2}, \
\tau\in[\frac 1{N^2},\frac 1N]
\end{align*}
for any $m>0$. We use this, Young's inequality and \eqref{weak_com}
to estimate \eqref{spt4} as
\begin{align*}
\eqref{spt4}&=\|\phi_{>1}\int_0^{\frac 1N}\otg\propgto\chi_{\le
\frac
12}\nabla\phi_{\le \frac 12}\nabla u(t+\tau)d\tau\|_{\ltrd}\\
&\lsm \sup_{\tau\in [0,\frac 1N]}\frac
1N\|\int(\phi_{>1}\otg\propgto\chi_{\le\frac
12})(x,y)\nabla\phi_{\le\frac 12}\cdot\nabla
u(t+\tau,y)dy\|_{\ltrd}\\
&\lsm N^{-10}\|\nabla u\phi_{>\frac 14}\|_{\ltrd}\\
&\lsm N^{-10}.
\end{align*}
The estimates of other terms give the same contribution, so finally
we have
\begin{equation*}
\eqref{spt}\lsm N^{-9}.
\end{equation*}

\vspace{0.2cm}

 \textit{Estimate of \eqref{mpm}}

Now we estimate \eqref{mpm}. We first write
$\phi_{>N\tau}F(u)=\phi_{>N\tau}F(\phi_{>N\tau/2}u)$, then further
decompose it by introducing a frequency projection:
\begin{align}
\eqref{mpm}&=i\phi_{>1}\int_{\frac
1N}^1\otg\propgto\phi_{>N\tau}P_{>\frac N8}F(\phi_{>N\tau/2}
u)(t+\tau)d\tau\label{mpm1}\\
&+i\phi_{>1}\int_{\frac 1N}^1\otg\propgto\phi_{>N\tau}P_{\le \frac
N8}F(\phi_{>N\tau/2})(t+\tau)d\tau.\label{mpm2}
\end{align}
\eqref{mpm1} can be estimated by weighted Strichartz estimate Lemma
\ref{L:wes}, Bernstein and \eqref{weak_com}. In 2-d, we have
\begin{align}
\|\eqref{mpm1}\|_{\ltrd}&\lsm \||x|^{-\frac
12}\phi_{>N\tau}P_{>\frac
N8}F(u\phi_{>N\tau/2})(t+\tau)\|_{L_{\tau}^{\frac 43}L_x^1([\frac
1N, 1]\times \R^2)}\\
&\lsm N^{-\frac 12}\|\tau^{-\frac 12}\|_{L_{\tau}^{\frac 43}([\frac
1N, 1])}\|P_{>\frac N8}
F(u\phi_{>N\tau/2})\|_{L_{\tau}^{\infty}L_x^1(\srt)}\\
&\lsm N^{-\frac 12}N^{-1}\|\nabla
(u\phi_{>N\tau/2})\|_{L_{\tau}^{\infty}L_x^2(\srt)}\|u\phi_{>N\tau/2}\|
_{L_{\tau}^{\infty}L_x^4(
\srt)}^2\\
&\lsm N^{-\frac 32}.
\end{align}

In three dimensions, we have
\begin{align}
\|\eqref{mpm1}\|_{\ltrd}&\lsm \||x|^{-\frac
12}\phi_{>N\tau}P_{>\frac
N8}F(u\phi_{>N\tau/2})(t+\tau)\|_{L_{\tau}^{\frac 87}L_x^{\frac
43}([\frac
1N, 1]\times \R^2)}\\
&\lsm N^{-\frac 12}\|\tau^{-\frac 12}\|_{L_{\tau}^{\frac 87}([\frac
1N, 1])}\|P_{>\frac N8}
F(u\phi_{>N\tau/2})\|_{L_{\tau}^{\infty}L_x^{\frac 43}(\srs)}\\
&\lsm N^{-\frac 12}N^{-1}\|\nabla
(u\phi_{>N\tau/2})\|_{L_{\tau}^{\infty}L_x^2(\srs)}\|u\phi_{>N\tau/2}\|
_{L_{\tau}^{\infty}L_x^{\frac {16}3}(
\srs)}^{\frac 43}\\
&\lsm N^{-\frac 32}.
\end{align}
Therefore, in two and three dimensions, we all have
\begin{equation*}
\eqref{mpm1}\lsm N^{-\frac 32}.
\end{equation*}

To estimate \eqref{mpm2}, we simply use the Strichartz and mismatch
estimate Lemma \ref{L:mismatch_fre} and \eqref{weak_com} to obtain
\begin{align*}
\eqref{mpm2}&\lsm \|\pntd\phi_{>N\tau}P_{\le
N/8}F(u\phi_{>{N\tau}/2})\|_{L_{\tau}^1L_x^2([t+\frac 1N,t+1]\times
\R^d)}\\
&\lsm
N^{-10}\|F(\phi_{>N\tau/2}u)\|_{L_{\tau}^{\infty}L_x^2(\srd)}\\
&\lsm N^{-10}\|\phi_{>N\tau/2}
u\|_{L_{\tau}^{\infty}L_x^{\frac{2(d+4)}d}(\srd)}^{\frac{d+4}d}\\
&\lsm N^{-10}.
\end{align*}
This completes the estimate of \eqref{mpm} so we get
\begin{align}
\|\eqref{mpm}\|_{\ltrd}\lsm N^{-\frac 32}.\label{mpm_est}
\end{align}

\vspace{0.2cm}

\textit{Estimate of \eqref{mpt}}

Now we treat the term \eqref{mpt}. The idea will be the same with
that of \eqref{spt}: we replace $F(u)$ by $(i\partial_t+\Delta )u$.
However, the situation here is a bit trickier since the cutoff
function also depends on $\tau$. By counting this factor, we improve
Lemma \ref{replace} to the following
\begin{lem}\label{replace_new}
For any $a, b\ge 0$, we have
\begin{align*}
&\quad\int_a^b \propgto(\phi_{\le
N\tau}(i\partial_\tau+\Delta)u(t+\tau))d\tau\\
&=i e^{-ib\Delta}(\phi_{\le Nb}u(t+b))-i e^{-ia\Delta}(\phi_{\le
Na}u(t+a))\\
&\quad -\int_a^b \propgto (u\Delta \phi_{\le N\tau}-2\nabla u\cdot
\nabla
\phi_{\le N\tau}))(t+\tau)d\tau\\
&\quad+i\int_a^b\propgto (\frac y{N\tau^2}\cdot(\nabla\phi_{\le
1})(\frac y{N\tau})u(t+\tau))d\tau.
\end{align*}
\end{lem}
Having this Lemma, we are able to write \eqref{mpt} as follows
\begin{align}
\eqref{mpt}&=i\phi_{>1}\int_{\frac 1N}^1 \otg \propgto[\phi_{\le
N\tau/2}(i\partial_\tau+\Delta)u(t+\tau)]d\tau\notag\\
&=-\phi_{>1}\otg e^{-i\Delta}(\phi_{\le \frac N2}u(t+1))\notag\\
&\quad+\phi_{>1}\otg e^{-i\frac 1N\Delta}(\phi_{\le \frac12}u(t+\frac 1N))\notag\\
&\quad-i\phi_{>1}\int_{\frac 1N}^1\otg \propgto(u\Delta\phi_{\le
N\tau/2}+2\nabla u\cdot \nabla \phi_{\le
N\tau/2}))(t+\tau)d\tau\notag\\
&\quad-\phi_{>1}\int_{\frac 1N}^1 \otg\propgto (\frac
y{N\tau^2}\cdot(\nabla \phi_{\le\frac 12})(\frac
y{N\tau})u(t+\tau))d\tau.\label{mpt4}
\end{align}
In the estimates of these terms, we will use the decay estimate of
the kernel: for any $\frac 1N\le\tau\le 1$, we have
\begin{align*}
|\phi_{>1}\otg\propgto\chi_{\le N\tau/2}(x,y)|&\lsm N^d \langle
N^2 \tau+N|x|-N|y|\rangle^{-m}\\
&\lsm N^d N^{-\frac m2}\langle N(x-y)\rangle^{-\frac m2},\ \forall
m\ge 0.
\end{align*}
We only give the estimate of the \eqref{mpt4} since the other terms
will be the same in principle. Using this decay estimate and Young's
inequality we have
\begin{align*}
&\qquad\|\eqref{mpt4}\|_{\ltrd} \\
&=\|\int_{\frac
1N}^1\int_{\R^d}(\phi_{>1}\otg\propgto\chi_{\le N\tau/2}(x,y)(\frac
y{N\tau^2}\cdot(\nabla\phi_{\le\frac12})(\frac y{N\tau})
u(t+\tau,y)dyd\tau\|_{\ltrd}\\
&\lsm N^{-11}\int_{\frac 1N}^1\|\frac x{N\tau^2}\phi_{|x|>\frac 14
N\tau}u(t+\tau)\|_{L_x^2(\R^d)} d\tau\\
&\lsm N^{-11} N\|u\|_{L_t^{\infty}L_x^2(\R\times\R^d)}\\
&\lsm N^{-10}.
\end{align*}
To conclude, we have
\begin{equation}
\|\eqref{mpt}\|_{\ltrd}\lsm N^{-10}.\label{est_mpt}
\end{equation}

To deal with \eqref{lpm}, we will need the following lemma which
controls the weighted Strichartz norm of the solution with spatial
cutoff on a unit time interval. We have
\begin{lem}\label{lem_wsbdd}
Let $d=2,3$. Let $I$ be a unit time interval, ie, $|I|=1$. Let $u$
be an $H_x^1$ solution of \eqref{nls} satisfying the weak
compactness condition \eqref{weak_com}. Then we have
\begin{equation*}
\||x|^{\frac{2(d-1)}q}\phi_{>1}u\|_{L_t^qL_x^{\frac{2q}{q-4}}(I\times\R^d)}\lsm
1.
\end{equation*}
\end{lem}
\begin{proof} The case $q=\infty$ is trivial. We only need to prove
the case $q=4$ since the general case follows by interpolation.

From radial Sobolev embedding Lemma \ref{L:radial_embed} we get
\begin{equation*}
\||x|^{\frac{d-1}2}\pn(\phi_{>1}u)\|_{L_x^{\infty}(\R^d)}\lsm
N^{\frac 12}\|\pn(\phi_{>1}u)\|_{\ltrd}\lsm \min(N^{\frac 12},
N^{-\frac 12}).
\end{equation*}
The desired estimate follows from the summation over dyadic pieces.
\end{proof}

Now we treat the $k$-th piece in \eqref{lpm} which we denote by
$\eqref{lpm}_k$. We put a frequency cutoff and further write it into
\begin{align}
\eqref{lpm}_k:&=i\phi_{>1}\int_{2^k}^{2^{k+1}}\otg\propgto\phi_{>N 2^k/2}
F(\phi_{>N2^k/4}u)(t+\tau)d\tau\\
&=i\phi_{>1}\int_{2^k}^{2^{k+1}}\otg\propgto\phi_{>N 2^k/2}P_{>\frac
N8}F(\phi_{>N2^k/4}u)
(t+\tau)d\tau\label{km}\\
&\qquad+i\phi_{>1}\int_{2^k}^{2^{k+1}}\otg\propgto\phi_{>N
2^k/2}P_{\le\frac N8}F(\phi_{>N2^k/4}u) (t+\tau)d\tau\label{kt}
\end{align}
We first look at \eqref{kt} which contains a frequency mismatch.
Hence from the $L^2_x$ boundedness of $\phi_{>1}\otg$, Strichartz
and mismatch estimate Lemma \ref{L:mismatch_fre} and
\eqref{weak_com}, we have
\begin{align*}
\|\eqref{kt}\|_{\ltrd}&\lsm \|\pntd \phi_{>N 2^k/2}P_{\le\frac
N8}F(\phi_{>N2^k/4}u)\|_{L_{\tau}^1L_x^2([t+2^k,
t+2^{k+1}]\times\R^d)}\\
&\lsm(N^2
2^k)^{-10}2^k\|\phi_{>N2^k/4}u\|_{L_{\tau}^{\infty}L_x^{\frac{2(d+4)}d}(\srd)}^{\frac{d+4}d}\\
&\lsm
N^{-20}2^{-9k}\|\phi_{>1}u\|_{L_t^{\infty}H_x^1(\srd)}^{\frac{d+4}d}\\
&\lsm (N2^k)^{-2}.
\end{align*}
For the another term \eqref{km}, we discuss the cases $d=2,3$
respectively. In 2 dimensions, using weighted Strichartz Lemma
\ref{L:wes}, Bernstein and Lemma \ref{lem_wsbdd} we have
\begin{align*}
\|\eqref{km}\|_{\ltrd}&\lsm \||x|^{-\frac 12}\phi_{>N2^k/2}P_{>\frac
N8}F(u\phi_{>N2^k/4})\|_{L_\tau^{\frac 43}L_x^1([t+2^k,
t+2^{k+1}]\times\R^2)}\\
&\lsm N^{-\frac 12}2^{-\frac k2}\|P_{>\frac
N8}F(u\phi_{>N2^k/4})\|_{L_\tau^{\frac 43}L_x^1([t+2^k,
t+2^{k+1}]\times\R^2)}\\
&\lsm N^{-\frac 32}2^{-\frac
k2}\|\nabla(u\phi_{>N2^k/4})\|_{L_\tau^\infty
L_x^2}\\
&\qquad\times\|u\phi_{>N2^k/4}\|_{L_\tau^4L_x^\infty([t+2^k,
t+2^{k+1}]\times\R^2)}\|u\phi_{>N2^k/4}\|_{L_\tau^2L_x^2([t+2^k,
t+2^{k+1}]\times\R^2)}\\
&\lsm N^{-2}2^{-\frac k2}\||x|^{\frac
12}u\phi_{>N2^k/4}\|_{L_\tau^4L_x^\infty([t+2^k,
t+2^{k+1}]\times\R^2)}\\
&\lsm N^{-2}2^{-\frac k4}.
\end{align*}
In 3 dimensions the corresponding estimates are as follows:
\begin{align*}
\|\eqref{km}\|_{\ltrd}&\lsm \||x|^{-\frac 12}\phi_{>N2^k/2}P_{>\frac
N8}F(u\phi_{>N2^k/4})\|_{L_\tau^{\frac 87}L_x^{\frac 43}([t+2^k,
t+2^{k+1}]\times\R^3)}\\
&\lsm N^{-\frac 12}2^{-\frac k2}N^{-1}\|\nabla
F(u\phi_{>N2^k/4})\|_{L_\tau^{\frac 87}L_x^{\frac 43}([t+2^k,
t+2^{k+1}]\times\R^3)}\\
&\lsm N^{-\frac 32}2^{-\frac k2}2^{\frac
78k-\frac{11}{72}k}\|\nabla(u\phi_{>N2^k/4})\|_{L_\tau^\infty
L_x^2}\\
&\qquad\times\|u\phi_{>N2^k/4}\|_{L_\tau^{\frac{72}{11}}L_x^{\frac{36}7}([t+2^k,
t+2^{k+1}]\times\R^3)}\|u\phi_{>N2^k/4}\|^{\frac 13}_{L_\tau^\infty
L_x^6([t+2^k,
t+2^{k+1}]\times\R^3)}\\
&\lsm N^{-\frac 32}2^{-\frac k2}2^{\frac
78k-\frac{11}{72}k}(N2^k)^{-\frac{11}{18}}\||y|^{\frac{11}{18}
}u\phi_{>N2^k/4}\|_{L_\tau^{\frac{72}{11}}L_x^{\frac{36}7}([t+2^k,
t+2^{k+1}]\times\R^3)}\\
&\lsm N^{-\frac 32}2^{-\frac {17}{72}k},
\end{align*}
which is acceptable. Collecting the estimates for \eqref{km},
\eqref{kt}, we obtain
\begin{align*}
\|\eqref{lpm}_k\|_{\ltrd}\lsm N^{-\frac 32}2^{-\frac 1{50}k}.
\end{align*}
Summing in $k$ gives that
\begin{align*}
\|\eqref{lpm}\|_{\ltrd}\lsm N^{-\frac 32}.
\end{align*}

\vspace{0.2cm}

\textit{Estimate of the last piece \eqref{lpt}}.

We have only one piece \eqref{lpt} left for which we will replace
$F(u)$ by $(i\partial_t+\Delta)u$ and use the kernel estimate.
Denote $\eqref{lpt}_k$ the $k$-th piece, we use Lemma \ref{replace}
to write
\begin{align*}
\eqref{lpt}_k&=-\phi_{>1}\otg e^{-i2^{k+1}\Delta}(\phi_{\le
N2^k/2}u(t+2^k))\\
&\quad+\phi_{>1}\otg e^{-2^k\Delta}(\phi_{\le N2^k/2}u(t+2^{k+1}))\\
&\quad-i\phi_{>1}\int_{2^k}^{2^{k+1}}\otg \propgto(u\Delta\phi_{\le
N2^k/2}+2\nabla u\cdot\nabla\phi_{\le N2^k/2})(t+\tau)d\tau
\end{align*}
Note for any $\tau\in[2^k, 2^{k+1}]$, from Lemma \ref{P:P
properties}, the kernel obeys
\begin{align*}
&\quad|\phi_{>1}\otg \propgto\phi_{\le N2^k/2}(x,y)|\\
&\lsm N^d\langle
N^2\tau+N|x|-N|y|\rangle^{-20}\phi_{|y|\le N2^k/2}\\
&\lsm N^d\langle N^2\tau+N|x|+N|y|\rangle^{-20}\\
&\lsm N^d(N^22^k)^{-10}\langle N(x-y)\rangle^{-10}.
\end{align*}
We immediately obtain
\begin{align*}
\|\eqref{lpt}_k\|_{\ltrd}\lsm N^{-10}2^{-k},
\end{align*}
after a simple application of Young's inequality.

Summing in $k$ gives us
\begin{equation*}
\|\eqref{lpt}\|_{\ltrd}\lsm N^{-10}.
\end{equation*}
We finish the estimates of all the six pieces
\eqref{spm}-\eqref{lpt}. Thus we succeed to estimate
$\|\eqref{out}\|_{\ltrd}\lsm N^{-1-\frac 14}$. Collecting the
estimates of \eqref{linear} to \eqref{out}, we establish Proposition
\ref{fre_decay} in 2,3 dimensions.

\subsection{Proof of Proposition \ref{local} in 2,3
dimensions}\label{local_23}

This subsection, we aim to prove that
\begin{align}
\|\phi_{>R}\pn u(t)\|_{\ltrd}\lsm \|\phi_{>R/2}
u_0\|_{\ltrd}+R^{-\frac 1{16}},\label{outside_est}
\end{align}
holds for all $t\ge 0$, $N\in[N_1,N_2]$ and $R>R_0(N_1, N_2,u)\gg
1$. We remark that since we have the freedom to choose sufficiently
large $R$, we will no longer accurately keep track of the power of
$N$, and often denote the power of $N$ by a constant $C$ which can
vary from one line to another.

To begin with, we use the in-out decomposition and forward, backward
Duhamel formula to write
\begin{align}
\phi_{>R}\pn u(t)&=\phi_{>R}P_N^-\propt u_0\label{r_linear}\\
&\quad -i\phi_{>R}\int_0^t P_N^- e^{i\tau\Delta}
F(u(t-\tau))d\tau\label{r_in}\\
&\quad +i\phi_{>R}\int_0^{\infty}\otg \propgto
F(u(t+\tau))d\tau.\label{r_out}
\end{align}
Again, the last integral should be understood in weak $\ltrd$ sense.
From Proposition \ref{P:P properties}, the kernel obeys:
\begin{align*}
|\phi_{>R}P_N^-\propt\phi_{\le R/2}(x,y)|\lsm N^d\langle
N(x-y)\rangle^{-20}\lsm N^C R^{-10}\langle N(x-y)\rangle^{-10}.
\end{align*}
Using this and the $L_x^2$ boundedness of $\phi_{>R}P_N^-$, we have
\begin{align*}
\eqref{r_linear}&=\|\phi_{>R}P_N^-\propt u_0\|_{\ltrd}\\
&\lsm
\|\phi_{>R}P_N^-\propt\phi_{>R/2}u_0\|_{\ltrd}+\|\phi_{>R}P_N^-\propt\phi_{\le
R/2}u_0\|_{\ltrd}\\
&\lsm \|\phi_{>R/2}u_0\|_{\ltrd}+N^C R^{-10}\\
&\lsm \|\phi_{>R/2} u_0\|_{\ltrd}+R^{-5}.
\end{align*}
Since \eqref{r_in}, \eqref{r_out} will be estimated in the same
manner, we choose to estimate one of them. We first split
\eqref{r_out} into time pieces:
\begin{align*}
\eqref{r_out}&=i\phi_{>R}\int_0^1\otg\propgto F(u(t+\tau))d\tau\\
&\quad+i\phi_{>R}\sum_{k=0}^{\infty}\int_{2^k}^{2^{k+1}}\otg\propgto
F(u(t+\tau))d\tau.
\end{align*}
Then introduce cutoff to further write it into
\begin{align}
\eqref{r_out}&=i\phi_{>R}\int_0^1\otg\propgto \phi_{\le
R/2}F(u(t+\tau))d\tau\label{rout_st}\\
&\quad+i\phi_{>R}\int_0^1 \otg\propgto
\phi_{>R/2}F(u(t+\tau))d\tau\label{rout_sm}\\
&\quad+i\sum_{k=0}^\infty \phi_{>R}\int_{2^k}^{2^{k+1}}\otg\propgto
\phi_{>\gamma R^{\frac 18}N^{\frac 78}2^{\frac
78k}}F(u(t+\tau))d\tau\label{rout_lm}\\
&\quad+i\sum_{k=0}^{\infty}\phi_{>R}\int_{2^k}^{2^{k+1}}\otg\propgto
\phi_{\le \gamma R^{\frac 18}N^{\frac 78}2^{\frac
78k}}F(u(t+\tau))d\tau.\label{rout_lt}
\end{align}
Here, $\gamma$ is a small constant, we can choose for example
$\gamma=\frac 1{100}$.

In the above four pieces, the main contribution comes from
\eqref{rout_sm}, \eqref{rout_lm} where $F(u)$ lies in large radius.
And we are going to estimate them by using weighted Strichartz. For
the tail part \eqref{rout_st}, \eqref{rout_lt}, we replace $F(u)$ by
$i\partial_t u+\Delta u$ and use the decay estimate of the kernel.
In what follows, we will only present the details of the last two
terms. The former two are much simpler so we skip them.

We first treat the $k$-th piece in \eqref{rout_lm}, which we denote
by $\eqref{rout_lm}_k$. For the sake of convenience, we write
$A:=\gamma R^{\frac 18}N^{\frac 78}2^{\frac 78k}$. In two
dimensions, using weighted Strichartz Lemma \eqref{L:wes}, we have
\begin{align*}
\|\eqref{rout_lm}_k\|_{\ltrd}&=\|\phi_{>R}\int_{2^k}^{2^{k+1}}\otg\propgto\phi_{>A}
F(\phi_{>A/2}u)(t+\tau)d\tau\|_{\ltrd}\\
&\lsm \||y|^{-\frac 12}F(\phi_{>A/2}u)\|_{L_\tau^{\frac
43}L_x^1([t+2^k,t+2^{k+1}]\times\R^2)}\\
&\lsm A^{-\frac 12}\|F(\phi_{>A/2}u)\|_{L_\tau^{\frac
43}L_x^1([t+2^k,t+2^{k+1}]\times \R^2)}\\
&\lsm A^{-1}2^{\frac k2}\||y|^{\frac
12}(\phi_{>A/2}u)\|_{L_\tau^4L_x^{\infty}([t+2^k,
t+2^{k+1}]\times\R^2)}\|\phi_{>A/2}u\|^2_{L_\tau^{\infty}L_x^2([t+2^k,t+2^{k+1}]\times\R^2)}.
\end{align*}
Using Lemma \ref{L:wes}, we continue the estimate
\begin{align*}
\|\eqref{rout_lm}_k\|_{\ltrd}
&\lsm A^{-1}2^{\frac 34k}\\
&\lsm R^{-\frac 18}2^{-\frac k8}N^{-\frac 78}\lsm R^{-\frac
1{16}}2^{-\frac k8}.
\end{align*}
In three dimensions, the corresponding computation is as follows:
applying weighted Strichartz estimate Lemma \ref{L:wes} with $q=8$,
Lemma \ref{lem_wsbdd} and \eqref{weak_com} we obtain
\begin{align*}
&\quad \|\eqref{rout_lm}_k\|_{\ltrd} \\
&\lsm \||y|^{-\frac
12}F(\phi_{>A/2}u)\|_{L_\tau^{\frac 87}L_x^{\frac
43}([t+2^k,t+2^{k+1}]\times \R^3)}\\
&\lsm A^{-\frac
12}\|\phi_{>A/2}u\|_{L_\tau^{4}L_x^\infty([t+2^k,t+2^{k+1}]\times\R^3)}\|\phi_{>A/2}u\|
_{L_\tau^{\frac
{32}{15}}L_x^{\frac{16}9}([t+2^k,t+2^{k+1}]\times\R^3)}^{\frac
43}.\\
&\lsm A^{-\frac 32}2^{\frac 58k}\||y|\phi_{>A/2} u\|_{L_\tau^4
L_x^\infty([t+2^k,t+2^{k+1}]\times\R^3)}\|u\phi_{>A/2}\|_{L_\tau^\infty
L_x^{\frac{16}9}(\srs)}^{\frac 43}\\
&\lsm A^{-\frac 32}2^{\frac 78k}\lsm R^{-\frac 3{16}}2^{-\frac
7{16}k}N^{-\frac {21}{16}}\lsm R^{-\frac 1{16}}2^{-\frac 7{16}k}.
\end{align*}
Summing in $k$ we obtain
\begin{equation}
\|\eqref{rout_lm}\|_{\ltrd}\lsm R^{-\frac 1{16}}.\label{est_rout_lm}
\end{equation}
This closes the estimates for \eqref{rout_lm}.

Now let us treat the $k$-th piece in \eqref{rout_lt}, we denote it
by $\eqref{rout_lt}_k$. Substituting $F(u)$ by $(i\partial_t+\Delta)
u$ and using Lemma \ref{replace} we have
\begin{align}
\eqref{rout_lt}_k&=-\phi_{>R}\otg e^{-i2^{k+1}\Delta}(\phi_{\le
A}u(t+2^{{k+1}}))\\
&\quad+\phi_{>R}\otg e^{-i2^k\Delta}(\phi_{\le A}u(t+2^k))\\
&\quad-i\phi_{>R}\int_{2^k}^{2^{k+1}}\otg\propgto(u\Delta \phi_{\le
A}+2\nabla u\cdot\nabla \phi_{\le A})(t+\tau)d\tau
\end{align}
Note for $|y|\le A$, $|x|>R$, $\tau\in[2^k,2^{k+1}]$, Young's
inequality gives that
\begin{align*}
|y|\le \gamma R^{\frac 18}2^{\frac 78k}N^{\frac 78}\ll R+N2^k\lsm
|x|+N\tau.
\end{align*}
we can control the kernel as: for any $\tau\in[2^k, 2^{k+1}]$,
\begin{align*}
|\phi_{>R}\otg\propgto \phi_{\le A}(x,y)|&\lsm N^d\langle
N^2\tau+N|x|-N|y|\rangle^{-20}\\
&\lsm N^C 2^{-10k}R^{-10}\langle N(x-y)\rangle^{-10}.
\end{align*}
Using this, Young's inequality and \eqref{weak_com} we immediately
get
\begin{align*}
\|\eqref{rout_lt}_k\|_{\ltrd}\lsm
2^{-5k}R^{-9}(\|u\|_{L_t^{\infty}L_x^2}+\|u\phi_{>1}\|_{L_t^\infty
H_x^1})\lsm 2^{-5k}R^{-5}.
\end{align*}
Summing in $k$ gives that
\begin{align}
\|\eqref{rout_lt}\|_{\ltrd}\lsm R^{-5}\label{est_rout_lt}.
\end{align}
Now we have estimates for all the pieces ready to conclude that
\begin{align}
\|\eqref{r_out}\|_{\ltrd}\lsm R^{-\frac 1{16}}.
\end{align}
Thus \eqref{outside_est} is established and we proved Proposition
\ref{local} in two and three dimensions.

\section{2+2 dimensions with splitting-spherical symmetry}\label{twotwo}

\subsection{Introduction and tools adapted to four dimensions}

In this section, we prove the frequency decay estimate Proposition
\ref{fre_decay} and spatial decay estimate Proposition \ref{local} in four
dimensions with admissible symmetry. From the definition, this
requires the initial data $u_0$ to be spherically symmetric in each of the
two two dimensional subspaces. More precisely, after a possible permutation or relabelling of the
coordinates, the solution $u$ satisfies
\begin{align*}
u(x_1,x_2,x_3,x_4,t)=u(x_1^2+x_2^2,x_3^2+x_4^2,t).
\end{align*}

Compared with the spherically symmetrical case, the situation here
is more complicated. As revealed in the 2,3-dimensional case, for any dyadic
number $N\gtrsim 1$, the
additional decay $N^{-1-\eps}$ is produced by those spherically
symmetric waves which travels at speed $N$, and which is supported
away from the origin.  In our $2+2$ splitting-spherically symmetric case,
due to the anisotropicity, one has to look at waves which
travels at the speed $N$ in one of the two subspaces in order to take
advantage of the partially spherical symmetry. However, a problem
one immediately encounters is that in the subpace where the waves are travelling,
these waves are not supported away from the origin. To treat this problem, we will give in
this section a uniform way (applying the in-out decomposition technique adapted to
the subspace) to deal with the close-to-origin piece and
away-from-origin piece. The price we pay is Lemma
\ref{Psmall_properties} which involves the kernel estimate close to
the origin. In the next section, we present a different way to deal
with this close-to-origin piece.

Besides the anisotropicity which makes the problem complicated in
disguise, the actual amount of computation is slightly lighter than the low dimensional
cases since in dimensions $d\ge 4$, one can always find a space so that the
nonlinearity in this space can be controlled by the mere mass of
$u$.

Before moving on, we introduce some notations that will be used
throughout this section.

\vspace{0.2cm}

We denote $\xf$ to be the two dimensional vector: $(\xft)$, and
$\xs$ as $(\xst)$. The differential operator $\nabla^{12}$,
$\nabla^{34}$, $\Delta_{12}$, $\Delta_{34}$ denote the
corresponding operators restricted to $(\xft)$ or $(\xst)$ plane.

By the same token, we define $P_N^{12}$, $P_N^{34}$ as two
dimensional Littlewood-Paley projection operators. $P_N^{12\pm}$,
$P_N^{34\pm}$ are the corresponding out-going and in-coming wave
projection operator. By augmenting with the identity operator on the other two variables, they can act
on a four-dimensional function.

We will also use the convention that when there is no superscript, the
notation should be understood as the normal one defined on $\R^4$. For
example, $P_N$ is the usual four-dimensional littlewood-Paley projection operator;
$\phi_{>R}$ should be understood as $\phi_{|x|>R}$ on $\R^4$ as well.

\vspace{0.2cm}

 We need the following weighted Strichartz estimate for splitting-spherically symmetric functions. We have

\begin{lem}(Weighted Strichartz estimate in splitting-spherical
symmetry case)\label{lem_wei_4d}.

Let $I$ be a time interval, $t_0\in I$. Let $u_0\in \ltrf$, $f\in
L_{t,x}^{\frac{2(d+2)}d}(I\times\R^4)$ be splitting-spherically
symmetric. Then the function $u(t,x)$ defined by
\begin{align*}
u(t):=e^{it\Delta }u_0-i\int_{t_0}^t e^{i(t-s)\Delta} f(s)ds
\end{align*}
is also splitting-spherically symmetrical and obeys the estimate.
For any $2< q\le \infty$,
\begin{align*}
\||\xf|^{\frac
2{3q}}u\|_{L_t^qL_x^{\frac{6q}{3q-4}}(I\times\R^4)}\lsm
\|u_0\|_{\ltrf}+\|f\|_{L_{t,x}^{\frac{2(d+2)}d}(I\times \R^4)}.
\end{align*}
The same conclusion holds true if we replace $\xf$ by $\xs$.
\end{lem}

\begin{proof} The proof is an adaptation to that in \cite{ktv:2d, kvz:blowup}.
By the standard Strichartz estimate Lemma
\ref{L:strichartz} and Christ-Kiselev lemma, it suffices to prove
\begin{align}
\||\xf|^{\frac 2{3q}}\propt u_0\|_{L_t^q L_x^{\frac
{6q}{3q-4}}(\rrf)}\lsm \|u_0\|_{\ltrf}.\label{w1}
\end{align}
By $TT^*$ argument, \eqref{w1} is reduced to proving the following
\begin{align*}
\|\int |\xf|^{\frac
2{3q}}e^{i(s-\tau)\Delta}&|\yf|^{\frac2{3q}}f(\tau)d\tau\|_{L_t^qL_x^{\frac{6q}{3q-4}}(\rrf)}\\
&\lsm \|f\|_{L_t^{q'}L_x^{\frac{6q}{3q+4}}(\rrf)}.
\end{align*}
However, this is the consequence of Hardy-Littlewood-Sobolev
inequality and the following dispersive estimate
\begin{equation}
\||\xf|^{\frac 2{3q}}\propt (|\yf|^{\frac
2{3q}}f)\|_{L_x^{\frac{6q}{3q-4}}(\R^4)}\lsm |t|^{-\frac
2q}\|f\|_{L_x^{\frac{6q}{3q+4}}(\R^4)}.\label{disperse}
\end{equation}
We now prove \eqref{disperse} for all $\frac 43\le q\le\infty$. When
$q=\infty$, this is just the trivial estimate
\begin{equation}
\|\propt f\|_{\ltrf}\lsm \|f\|_{\ltrf}.\label{trivial}
\end{equation}
When $q=\frac 43$, this is the pointwise estimate
\begin{equation}
\||\xf|^{\frac 12}\propt |\yf|^{\frac 12}
f\|_{L_x^{\infty}(\R^4)}\lsm |t|^{-\frac
32}\|f\|_{L_x^1(\R^4)}\label{infty_one}
\end{equation}
for all $f$ being spherically symmetric in $(\xft)$ variable. By
passing to radial coordinate, we write
\begin{align*}
(|\xf|^{\frac 12}\propt |\yf|^{\frac12})(x,y)&=|\xf|^{\frac
12}|\yf|^{\frac 12}\frac 1{(4\pi
it)^2}e^{i[(x_3-y_3)^2+(x_4-y_4)^2]/4t}\\&e^{i(|\xf|^2+|\yf|^2)/4t}
\int_0^{2\pi}e^{-i|\xf||\yf|\cos\theta/2t}d\frac{\theta}{2\pi}.
\end{align*}
A standard stationary phase yields that
\begin{equation*}
|(|\xf|^{\frac 12}\propt|\yf|^{\frac 12})(x,y)|\lsm |t|^{-\frac 32},
\end{equation*}
from which \eqref{infty_one} follows.

Finally we obtain \eqref{disperse} by interpolating between
\eqref{trivial} and \eqref{infty_one}. Indeed, \eqref{trivial} and
\eqref{infty_one} imply that the operator
\begin{equation*}
|\xf|^{\frac 12}\propt: L_x^1(|\xf|^{-\frac 12}dx)\rightarrow
L_x^{\infty}(dx),
\end{equation*}
with bound $|t|^{-\frac 32}$, and
\begin{equation*}
|\xf|^{\frac 12}\propt: L_x^2(dx)\to L_x^2(|\xf|^{-1} dx),
\end{equation*}
with bound 1. Using interpolation between $L^p$ spaces with
changing of measures (see\cite{BL}, page 120), we obtain
\eqref{disperse}. Lemma \ref{lem_wei_4d} is proved.
\end{proof}

We need the following Lemma to control the nonlinearity when $F(u)$
is located at large radii.

\begin{lem}\label{lem_56est}
Let $N$ be dyadic number. Let $L>0$ be a constant such that
$LN\gtrsim 1$. Then we have
\begin{equation}\label{56est}
\|P_{>N} F(\phi_{|\xf|>LN}u)\|_{L_x^{\frac 65+}(\R^4)}\lsm N^{-\frac 59} (NL)^{-\frac 29}=N^{-\frac
79}L^{-\frac 29}.
\end{equation}
The same estimate will still hold if we replace $\xf$ by $\xs$ and
$P_{>N}$ by the fattened operator $P_{\gtrsim N}$.
\end{lem}
\begin{proof}
This is the simple use of Bernstein, H\"older, radial Sobolev
embedding, and Lemma \ref{uni_bdd}. We give the estimates as
follows:
\begin{align*}
&\qquad\|P_{>N}F(\phi_{|\xf|>LN} u)\|_{L_x^{\frac 65+}(\R^4)}\\
&\lsm
N^{-\frac 59}\||\nabla|^{\frac 59}F(\phi_{|\xf|>LN} u)\|_{L_x^{\frac
65+}(\R^4)}\\
&\lsm N^{-\frac
59}\||\nabla|^{\frac 59}(\phi_{|\xf|>LN}u)\|_{\ltrf}\|\phi_{|\xf|>LN}u\|_{L_x^{3+}(\R^4)}\\
&\lsm N^{-\frac 59} (LN)^{-\frac 29}\|\phi_{\gtrsim 1}
u\|_{H_x^1(R^4)}\biggl\| \||\xf|^{\frac 29}\phi_{|\xf|>LN}
u\|_{L_{\xft}^{3+}(\R^2)}\biggl\|_{L_{\xst}^{3+}(\R^2)}\\
&\lsm N^{-\frac 79}L^{-\frac 29}\|\||\nabla^{12}|^{\frac
19+}(\phi_{|\xf|>LN}u)\|_{L_{\xft}^2(\R^2)}\|_{L_{\xst}^{3+}(\R^2)}\\
&\lsm N^{-\frac 79}L^{-\frac 29}\|\||\nabla^{12}|^{\frac
19+}(\phi_{>LN}u)\|_{L_{\xst}^{3+}(\R^2)}\|_{L_{\xft}^2(\R^2)}\\
&\lsm N^{-\frac 79}L^{-\frac 29}\||\nabla^{34}|^{\frac
13+}|\nabla^{12}|^{\frac 19+}(\phi_{>LN}u)\|_{\ltrf}\\
&\lsm N^{-\frac 79}L^{-\frac 29}\|\phi_{>LN} u\|_{H_x^1(\R^4)}\\
&\lsm N^{-\frac 79}L^{-\frac 29}.
\end{align*}
The lemma is proved.
\end{proof}

Now we are ready to give the

\subsection{Proof of Proposition \ref{fre_decay}}
In this subsection, we prove Proposition \ref{fre_decay} in four
dimensions with splitting-spherical symmetry. Note that for $|\xi|
\sim N$, $\xi \in \R^4$, since $|\xi|^2 = |\xi^{12}|^2 +
|\xi^{34}|^2$, we have either $|\xi^{12}| \sim N$, or $| \xi^{34}|
\sim N$ (or both hold), therefore
\begin{align*}
P_N \sim P_N \tilde P_N^{12} \tilde P_{\lsm N}^{34} + P_N \tilde P_N^{34} \tilde P_{\lsm N}^{12}.
\end{align*}
For notational simplicity, we shall drop the tildes and write
instead
\begin{align*}
\pn \sim \pn({\pnf}\pns+\pn^{34}P_{\le N}^{12}).
\end{align*}
Using triangle we obtain
\begin{align*}
\|\phi_{>10}\pn u(t)\|_{\ltrf}\lsm \|\phi_{>10}\pn{\pnf}\pns
u(t)\|_{\ltrf}+\|\phi_{>10}\pn{P_N^{34}}P^{12}_{\le N}\|_{\ltrf}
\end{align*}
By symmetry, we only need to consider the first one.
\begin{align}
\|\phi_{>10}\pn \pnf\pns u(t)\|_{\ltrf}\label{target}.
\end{align}
 Commuting the spatial and frequency cutoff in \eqref{target}, we
 have from Lemma \ref{L:mismatch_fre} that
\begin{align*}
&\quad \|\phi_{>10}\pn\pnf\pns u(t)\|_{\ltrf} \\
&\lsm
\|\phi_{>10}\pn\phi_{>5}\pnf\pns
u(t)\|_{\ltrf}\\
&\qquad+\|\phi_{>10}\pn\phi_{\le 5}\pnf\pns u(t)\|_{\ltrf}\\
&\lsm \|\phi_{>10}\pn\phi_{>5}\pnf\pns
u(t)\|_{\ltrf}+(5N)^{-10}\|u(t)\|_{\ltrf}\\&\lsm N^{-9}+
\|\phi_{>5}\pnf\pns u(t)\|_{\ltrf}.
\end{align*}
Therefore, Proposition \ref{fre_decay} will follow if we can prove
\begin{align}
\|\phi_{>5}\pnf \pns u(t)\|_{\ltrf}\lsm \|\pntd
u_0\|_{\ltrf}+N^{-1-\frac 1{10}}.\label{p12}
\end{align}

Like before, the additional decay factor of $N$ will either come
from the mismatch between the spatial cutoff and the linear
propagator or the spherical symmetry in $(\xft)$ (resp. $(x_3,x_4)$)
variables. To this end, we introduce a spatial cutoff in $\xf$
variables and further decompose the LHS of \eqref{p12} into:
\begin{align}
&\quad\|\phi_{>5}\pnf \pns u(t)\|_{\ltrf} \notag\\
&\le\|\phi_{>5}\phi_{|\xf|\le \frac 1{N^3}}\pnf\pns
u(t)\|_{\ltrf}\label{smb}\\
&\quad+\|\phi_{>5}\kmid\pnf\pns u(t)\|_{\ltrf}\label{mdb}\\
&\quad+\|\phi_{>5}\phi_{|\xf|>\frac 1N}\pnf\pns
u(t)\|_{\ltrf}\label{lgb}
\end{align}
As we said, the reason for doing so is that the Hankel function has
logarithmic singularity at the origin. So it is necessary to isolate
the singular regime where $|\xf|\le \frac 1{N^3}$ and deal with the
near-origin regime and away-from-origin regime separately.

\eqref{smb} can be estimated rather directly by the H\"older and Bernstein inequalities:
\begin{align*}
\eqref{smb}
&\lsm \|\phi_{|\xf|\le\frac 1{N^3}}\pnf\pns u(t)\|_{\ltrf}\\
&\lsm \|\frac 1{N^3}\|\pnf
u(t)\|_{L_{\xft}^{\infty}(\R^2)}\|_{L_{\xst}^2(\R^2)}\\
&\lsm \frac 1{N^2}\|\pnf u(t)\|_{\ltrf} \lsm \frac 1{N^2}.
\end{align*}
Thus the contribution due to \eqref{smb} is acceptable.
Next we estimate \eqref{mdb}. Using the in-out decomposition with
respect to $\xf$ variable and Duhamel formula, we estimate
\begin{align}
\eqref{mdb}&\lsm \|\phi_{>5}\kmid\pnp\pns u(t)\|_{\ltrf}\notag\\
&\qquad+
\|\phi_{>5}\kmid \pnn \pns u(t)\|_{\ltrf}\notag\\
&\lsm \|\phi_{>5}\kmid\pnn\pns \propt u_0\|_{\ltrf}\label{mdb_lin}\\
&\quad+\|\phi_{>5}\kmid\pnn\pns\int_0^t e^{i\tau\Delta}
F(u(t-\tau))d\tau\|_{\ltrf}\label{mdb_in}\\
&\quad+\|\phi_{>5}\kmid\pnp\pns\int_0^{\infty} \propgto
F(u(t+\tau))d\tau\|_{\ltrf}\label{mdb_out}.
\end{align}
The last integral should be understood to hold in the weak $L_x^2$
sense. For \eqref{mdb_lin}, we insert a fattened $\pn$ and use 
 the $L^2$ boundedness of $\pnn$ to get
\begin{align*}
\eqref{mdb_lin}& \lsm \|\kmid\pnn \pns\propt \pntd u_0\|_{\ltrf}\\
&\lsm \|\pntd u_0\|_{\ltrd},
\end{align*}
which is acceptable.

\eqref{mdb_in}, \eqref{mdb_out} will be treated in the same manner
so we only provide the details of the estimate of \eqref{mdb_out}.
Splitting different time pieces and inserting spatial cutoff, we
estimate
\begin{align}
\eqref{mdb_out}&\le
\|\phi_{>5}\kmid\pnp\pns\int_0^{\frac 1{50N}}\propgto\phi_{\le 1}F(u)(t+\tau)\tau\|_{\ltrf}\label{bdb}\\
&\quad+\|\phi_{>5}\kmid\pnp\pns\int_0^{\frac
1{50N}}\propgto\phi_{>1}
F(u)(t+\tau) d\tau\|_{\ltrf}\label{ssmb}\\
&\quad+\|\phi_{>5}\kmid\pnp\pns\int_{1}^\infty\propgto\phi_{>N\tau/2}
F(u)(t+\tau) d\tau\|_{\ltrf}\label{lm}\\
&\quad+\|\phi_{>5}\kmid\pnp\pns\int_{\frac
1{50N}}^1\propgto\phi_{>N\tau/2}F(u(t+\tau))d\tau\|_{\ltrf}\label{lms}\\
&\quad+\|\phi_{>5}\kmid\pnp\pns\int_{\frac
1{50N}}^\infty\propgto\phi_{\le N\tau/2} F(u)(t+\tau)
d\tau\|_{\ltrf}\label{lt}.
\end{align}
Then our task is reduced to controlling the these four terms which
we will do now.

\textbf{The estimate of \eqref{bdb}}

For this term, the additional decay in $N$ will come from the kernel
estimate Lemma \ref{L:kernel}. To proceed, we first use the $L^2$
boundedness of $\pnp$, Bernstein estimate, mismatch estimate Lemma
\ref{L:mismatch_real} and Strichartz to obtain:
\begin{align*}
\eqref{bdb}&\lsm \|\phi_{|\xs|>4}\pnp\pns\pntd\int_0^{\frac
1{50N}}\propgto\phi_{\le 1} F(u(t+\tau))d\tau\|_{\ltrf}\\
&\lsm \|\phi_{|\xs|>4}\pns\pntd\int_0^{\frac
1{50N}}\propgto\phi_{\le 1}F(u(t+\tau))d\tau\|_{\ltrf}\\
&\lsm \|\phi_{|\xs|>4}\pns\phi_{|\xs|\le 3}\pntd\int_0^{\frac
1{50N}}\propgto\phi_{\le 1}F(u(t+\tau))d\tau\|_{\ltrf}\\
&\quad+\|\phi_{|\xs|>3}\pntd\int_0^{\frac 1{50N}}\propgto \phi_{\le
1}F(u(t+\tau))d\tau\|_{\ltrf}\\
&\lsm N^{-10}\|\pntd\int_0^{\frac 1{50N}}\propgto\phi_{\le
1}F(u(t+\tau))d\tau\|_{\ltrf}\\
&\qquad+\|\int_0^{\frac 1{50N}}\phi_{>2}\pntd\propgto\phi_{\le
1}F(u(t+\tau))d\tau\|_{\ltrf}\\
&\lsm N^{-11}\|\pntd\phi_{\le
1}F(u)\|_{L_{\tau}^{\infty}L_x^2} \\
&\quad +\frac 1N\sup_{0\le\tau\le\frac
1{50N}}\|\phi_{>2}\pntd\propgto\phi_{\le 1}F(u(t+\tau))\|_{\ltrf}
\end{align*}
Now using Bernstein and the kernel estimate: for any
$\tau\in[0,\frac 1{50N}]$,
\begin{align*}
\phi_{>2}\pntd\propgto\phi_{\le 1}(x,y)&\lsm N^4\langle
N|x-y|\rangle^{-20}\phi_{|x|>2}\phi_{|y|\le 1}\\
&\lsm N^{-10}\langle|x-y|\rangle^{10},
\end{align*}
we continue the estimate of \eqref{bdb} as
\begin{align*}
\eqref{bdb}\lsm N^{-9}\|F(u)\|_{L_\tau^\infty L_x^1}\lsm N^{-5}.
\end{align*}

\textbf{The estimate of \eqref{ssmb}}

For this term, we simply use the Strichartz estimate and
weak localization of kinetic energy \eqref{weak_com},
\begin{align*}
\eqref{ssmb}&\lsm \|\pnp\pns\pntd\int_0^{\frac
1{50N}}\propgto\phi_{>1}F(u(t+\tau))d\tau\|_{\ltrf}\\
&\lsm \|\pntd\phi_{>1}F(u(t+\tau))\|_{L_\tau^2L_x^{\frac
43}([0,\frac 1N]\times\R^4)}\\
&\lsm N^{-\frac 32}\|\nabla(\phi_{>1}F(\phi_{>\frac
12}u))\|_{L_\tau^\infty L_x^{\frac 43}(\srf)}\\
&\lsm N^{-\frac 32}(\|\phi_{>\frac 12}u\|_{L_\tau^\infty
H_x^1(\srf)}\|\phi_{>\frac12} u\|_{L_\tau^\infty L_x^4(\srf)})\\
&\lsm N^{-\frac 32}.
\end{align*}

\textbf{The estimate of \eqref{lm}}

For this term, the decay comes from the partial spherical symmetry of the
solution. To this end, we first use the $L^2$ boundedness of $\pnp$
and mismatch estimate Lemma \ref{L:mismatch_fre} to simplify the
computation, then use the weighted Strichartz estimate Lemma
\ref{lem_wei_4d}, Lemma \ref{lem_56est} and the weak localization of kinetic energy
\eqref{weak_com} to obtain
\begin{align*}
\eqref{lm}&\lsm \|\pntd\int_{1}^\infty \propgto
\phi_{>N\tau/2}F(\phi_{>N\tau/4}u(t+\tau))d\tau\|_{\ltrf}\\
&\lsm \| \pntd\int_{1}^\infty \propgto\phi_{>N\tau/2}(P_{<\frac 18
N}+P_{>8N})F(u\phi_{>N\tau/4})(t+\tau)d\tau\|_{\ltrf}\\
&\quad+\|\pntd\int_{1}^\infty \propgto\phi_{>N\tau/2}P_{\frac
N8<\cdot\le
8N}F(u\phi_{>N\tau/4})(t+\tau)d\tau\|_{\ltrf}\\
&\lsm \int_{1}^\infty \|\pntd\phi_{>N\tau/2}(P_{<\frac
N8}+P_{>8N})F(u\phi_{>N\tau/4})(t+\tau)\|_{\ltrf}d\tau\\
&\quad+ \||x|^{-\frac 13+}\phi_{>N\tau/2}P_{\frac N8<\cdot\le
8N}F(u\phi_{>N\tau/4})\|_{L_\tau^{2-}L_x^{\frac
65+}([1,\infty)\times\R^4)}\\
&\lsm \int_{1}^{\infty}\|(N^2\tau)^{-10}F(u\phi_{>N\tau/4})(t+\tau)\|_{\ltrf}d\tau\\
&\quad+ N^{-\frac 13+}\biggl\|\tau^{-\frac 13+}\|P_{\frac
N8<\cdot\le 8N}F(u\phi_{>N\tau/4})\|_{L_x^{\frac
65+}(\R^4) }\biggr\|_{L_\tau^{2-}([1,\infty))}\\
&\lsm N^{-20}\int_{1}^\infty
\tau^{-10}d\tau\|F(u\phi_{>N\tau/4})\|_{L_\tau^\infty
L_x^2(\srf)} \\
& \quad+N^{-\frac 13+}\|\tau^{-\frac 13+}\tau^{-\frac
29}N^{-\frac 79}\|_{L_\tau^{2-}([1,\infty))}\\
&\lsm N^{-1-\frac 1{10}}.
\end{align*}

\textbf{The estimate of \eqref{lms}}

This term will be estimated in a similar way as \eqref{lm}. We first
use $L^2$ boundedness of $\pnp$, the mismatch estimate
\ref{L:mismatch_fre}, then use weighted Strichartz estimate
\ref{lem_wei_4d}, weak localization of kinetic energy energy \eqref{weak_com} and Bernstein to obtain
\begin{align*}
\eqref{lms}&\lsm \|\pntd\int_{\frac 1{50N}}^1
\propgto\phi_{>N\tau/2}F(\phi_{>N\tau/4}u)(t+\tau)d\tau\|_{\ltrf}\\
&\lsm \|\pntd\int_{\frac 1{50N}}^1\propgto\phi_{>N\tau/2}P_{\frac
N8<\cdot\le 8N}F(\phi_{>N\tau/4}u)(t+\tau)d\tau\|_{\ltrf}\\
&\quad+\|\pntd\int_{\frac 1{50N}}^1\propgto\phi_{>N\tau/2}(P_{\le
\frac N8}+P_{>8N})F(u)(t+\tau)d\tau\|_{\ltrf}\\
&\lsm \int_{\frac 1{50N}}^1\|\pntd\propgto\phi_{>N\tau/2}(P_{\le
\frac N8}+P_{>8N})F(u)(t+\tau)\|_{\ltrf}d\tau\\
&\quad+\||x|^{-\frac 13+}\phi_{>N\tau/2}P_{\frac N8<\cdot\le
8N}F(u\phi_{>N\tau/4})(t+\tau)\|_{L_\tau^{2-}L_x^{\frac 65+}([\frac
1{50N},1]\times \R^4)  }\\
&\lsm \int_{\frac 1{50N}}^1
(N^2\tau)^{-20}\|F(u\phi_{>N\tau/4})\|_{\ltrf} d\tau\\
&\quad+N^{-\frac 13+}\biggl\|\tau^{-\frac 13+}\|P_{\frac N8<\cdot\le
8N}F(u\phi_{>N\tau/4})\|_{L_x^{\frac
65+}}\biggr\|_{L_{\tau}^{2-}([\frac 1{50N},1])}\\
&\lsm N^{-10}\|u\phi_{>\frac
1{50}}\|_{L_{\tau}^{\infty}L_x^4(\srf)}^2 \\
&\qquad +N^{-\frac
13+}N^{-1}\biggl\|\tau^{-\frac 13+}\|\nabla
F(u\phi_{>N\tau})\|_{L_x^{\frac 65+}}\biggr\|_{L_\tau^{2-}([\frac
1{50N},1])}\\
& \lsm N^{-\frac 43+}\|\tau^{-\frac 13+}\|_{L_{\tau}^2([\frac
1{50N},1])}\|\nabla(u\phi_{>N\tau/4})\|_{L_\tau^\infty
L_x^2(\srf)} \cdot \\
& \quad \cdot \|u\phi_{>N\tau/4}\|_{L_\tau^\infty L_x^{3+}(\srf)}+ N^{-10}\\
&\lsm N^{-1-\frac 14}.
\end{align*}

\textbf{The estimate of \eqref{lt}}

We are left with this very last term for which we first control it
as
\begin{align*}
&\quad \eqref{lt} \\
&\lsm \|\kmid\pnp\int_{\frac 1{50N}}^\infty
\propgto\phi_{|\yf|\le N\tau/2}\pns\phi_{\le
N\tau/2}F(u(t+\tau))d\tau\|_{\ltrf}\\
&\lsm \int_{\frac 1{50N}}^\infty \|\kmid\pnp
e^{-i\tau\Delta_{12}}\phi_{|\yf|\le N\tau/2}\pns\phi_{\le
N\tau/2}F(u)(t+\tau)\|_{\ltrf}d\tau.
\end{align*}
Now we use the kernel estimate Lemma \ref{Psmall_properties} for
small $x$ regime:
\begin{align*}
&|\kmid\pnp e^{-i\tau\Delta_{12}}\phi_{|\yf| \le
N\tau/2}(\xf,\yf)|\\
&\quad\lsm N^2\log N\langle N^2\tau+N|y|\rangle^{-20}\kmid\phi_{|\yf|\le N\tau/2}\\
&\lsm N^2\log N\langle N^{2}\tau+N|y|+N|x|\rangle^{-20}\\
&\lsm N^2\log N |N^2\tau|^{-10}\langle N|x-y|\rangle^{-10}
\end{align*}
and Young's inequality to obtain
\begin{align*}
&\quad \eqref{lt} \\
&\lsm \int_{\frac 1{50N}}^\infty N^2\log N
(N^2\tau)^{-10}\biggl\|\| \langle
N|\cdot|\rangle^{-10}*_{12}(\pns\phi_{N\tau/2}F(u(t+\tau))\|_{L_{\xf}^2}\biggr\|_{L_{\xs}^2}d\tau\\
&\lsm N^2\log N N^{-20}\int_{\frac 1{50N}}^\infty \tau^{-10}
\biggl\| \|\pns\phi_{\le
N\tau/2}F(u(t+\tau))\|_{L_{\xf}^1}\|_{L_{\xs}^2} d\tau\\
&\lsm N^{-15}\int_{\frac 1{50N}}^\infty \tau^{-10} \| \|\pns\phi_{\le
N\tau/2}F(u(t+\tau))\|_{L_{\xs}^2}\|_{L_{\xf}^1} d\tau\\
&\lsm N^{-13}\int_{\frac 1{50N}}^\infty \tau^{-10}\|\phi_{\le
N\tau/2}F(u(t+\tau))\|_{L^1_x(\R^4)}d\tau\\
&\lsm N^{-2}\|F(u)\|_{L_\tau^\infty L_x^1}\\
&\lsm N^{-2}.
\end{align*}

Collecting the estimates for \eqref{bdb} through \eqref{lt}, we
conclude
\begin{align*}
\eqref{mdb_out}\lsm N^{-1-\frac 1{10}}.
\end{align*}
Therefore \eqref{mdb} gives the desired contribution $N^{-1-\frac
1{10}}$.

Finally we remark that after minor changes, the contribution from
\eqref{lgb} can be estimated in pretty much the same manner as for
\eqref{mdb}. More rigorously, in estimating \eqref{mdb}, we will
chop it into five pieces like from \eqref{bdb} to \eqref{lt}. Since
in the estimate of \eqref{ssmb}, \eqref{lm}, \eqref{lms}, we simply
throw away the cutoff in $\xf$ variable, these three estimates will
still be valid in this case. Recalling that in estimating
\eqref{bdb}, the decay essentially stems from the mismatch between
the spatial cutoff $\phi_{>5}$ and the propagator $\propgto
\phi_{\le 1}$ which is also available in this case, we are able to
handle the analogue to \eqref{bdb}. Finally, the analogue to the
very last term can also be treated by using the normal kernel
estimate Proposition \ref{P:P properties}. To conclude, we prove
\eqref{p12} and the proof of Proposition \ref{fre_decay} is
completed.

\subsection{Proof of Proposition \ref{local}}

In this subsection, we establish Proposition \ref{local} in 4
dimensions with admissible symmetry. More precisely we will show,
for any dyadic number $N_1, N_2>0$, there exist $R_0=R_0(u,N_1,N_2)$
such that
\begin{align*}
\|\phi_{>R}\pn u(t)\|_{\ltrf}\lsm \|\phi_{>\sqrt
R}u_0\|_{\ltrf}+R^{-\frac 1{100}}, \quad \forall\, R>R_0.
\end{align*}
\begin{proof}
Since $\pn \sim \pn\pnf\pns+\pn P_N^{34}P_{\le N}^{12}$, it suffices
for us to consider $$\|\phr\pn\pnf\pns u(t)\|_{\ltrf}.$$ Commuting
the spatial cut-off $\phr$ with the projector $\pn$, we have
\begin{align*}
&\qquad\|\phr\pn\pnf\pns u(t)\|_{\ltrf}\\
&\lsm \|\phr\pn\phrt\pnf\pns u(t)\|_{\ltrf}+\|\phr\pn\phi_{\le \frac
R2}\pnf\pns
u(t)\|_{\ltrf}\\
&\lsm \|P_N \phrt\pnf\pns u(t)\|_{\ltrf}+(RN)^{-10}\\
&\lsm \| \phrt\pnf\pns u(t)\|_{\ltrf}+R^{-5}.
\end{align*}
Let $c$ be a small number close to but smaller than 1. By
introducing spatial cutoffs, it reduces to controlling
\begin{align}
\|\phi_{|\xf|\le\frac 1R}\phrt\pnf\pns
u(t)\|_{\ltrf}\label{1}\\
\|\midd\phrt\pnf\pns u(t)\|_{\ltrf}\label{2}\\
\|\phi_{|\xf|>R^c}\phrt\pnf\pns u(t)\|_{\ltrf}\label{3}.
\end{align}
\eqref{1} is controlled simply by H\"older and Bernstein inequalities:
\begin{align*}
\eqref{1}&\lsm \|\phi_{|\xf|\le\frac 1R}\pnf u(t)\|_{\ltrf}\\
&\lsm \frac 1R\biggl\| \|\pnf
u(t)\|_{L_{\xft}^{\infty}(\R^2)}\biggr\|_{L_{\xst}^2(\R^2)}\\
&\lsm \frac NR\|u(t)\|_{\ltrf}\lsm R^{-\frac 12}.
\end{align*}
To control \eqref{2}, we use in-out decomposition and Duhamel
formula to get
\begin{align}
\eqref{2}&\le \|\midd\phrt P_N^{12-} \pns e^{it\Delta}
u_0\|_{\ltrf}\label{2.1}\\
&+\|\midd\phrt P_N^{12-} \pns\int_0^t
e^{i\tau\Delta}F(u)(t-\tau)d\tau\|_{\ltrf} \notag \\
&+\|\midd\phrt\pnp\pns\int_0^{\infty}\propgto
F(u)(t+\tau)d\tau\|_{\ltrf}.\label{2.2}
\end{align}
We first estimate \eqref{2.1} by writing
\begin{align}
\eqref{2.1}&\lsm \| \midd\phrt P_N^{12-} \pns \propt \frc
u_0\|_{\ltrf}\notag\\
&\qquad+\| \midd\phrt P_N^{12-} \pns\propt\frcs u_0\|_{\ltrf}\notag\\
&\lsm \|\frc u_0\|_{\ltrf}\notag\\
&+\| \midd\phrt P_N^{12-} \pns\propt\frcs u_0\|_{\ltrf}\label{2.1.1}.
\end{align}
This will give desired estimate once we establish
\begin{equation*}
\eqref{2.1.1}\lsm R^{-\frac 1{50}}.
\end{equation*}
We discuss two cases.

Case 1. $0\le t\le \frac R{100N}$. We write
\begin{align*}
\eqref{2.1.1}& \lsm \left\| \phi_{|x^{34}| >\frac R3} \tilde P_N^{12} P_{\le N}^{34} \tilde P_N e^{it\Delta}
\phi_{\le \frac 12 R^c} u_0 \right\|_{L^2(\R^4)} \\
& \lsm \left\| \phi_{|x^{34}| >\frac R3} P_{\le N}^{34} \tilde P_N e^{it\Delta} \phi_{\le \frac 12 R^c} u_0
\right\|_{L^2(\R^4)} \\
&\lsm \left\| \phi_{>\frac R4}  \tilde P_N e^{it\Delta}
 \phi_{\le \frac 12 R^c} u_0 \right\|_{\ltrf} \\
& \qquad + (NR)^{-10} \cdot \|u_0\|_{\ltrf}  \\
& \lsm R^{-5},
\end{align*}
where we have used the kernel estimate : $\forall\, t\le \frac R{100N}$
\begin{align} \label{one_star}
 \left| (\phi_{>\frac R4} \tilde P_N e^{it\Delta} \phi_{\le \frac 12 R^c} )(x,y) \right|
\lsm R^{-10} \langle N (x-y) \rangle^{-10}.
\end{align}

Case 2. $t>\frac R{100N}$. We first use triangle to split
\begin{align*}
\eqref{2.1.1} &\le \|\midd\pnn e^{it\Delta_{12}} \frcs u_0\|_{\ltrf}\\
&\le \|\phi_{\frac 1R<|\xf|\le\frac 1N}\pnn e^{it\Delta_{12}}\frcs u_0\|_{\ltrf}\\
&\qquad +\|\phi_{\frac 1N<|\xf|\le R^c}\pnn e^{it\Delta_{12}}\frcs
u_0\|_{\ltrf}.
\end{align*}
Since in both of the two terms, the kernels obey
\begin{align}
&\qquad|( \phi_{\frac 1R <|x^{12}|\le \frac 1N} P_N^{12-} e^{it\Delta_{12}} \phi_{|y^{12}|<\frac 12 R^c} )(\xf,\yf)|\notag\\
&\lsm (\log R)^{10} N^C\langle N^2t-N|\yf|\rangle^{-100}\phi_{\frac 1R<|\xf|\le\frac 1N} \phi_{|y^{12}|<\frac 12 R^c} \notag\\
&\lsm (\log R)^{10} N^C |N^2 t|^{-50}\langle N(\xf-\yf)\rangle^{-50}\notag\\
&\lsm R^{-10}|N^2t|^{-10}\langle N(\xf-\yf)\rangle^{-10}\lsm
R^{-10}\langle N(\xf-\yf)\rangle^{-10}\label{two_star}
\end{align}
and
\begin{align}
&|\phi_{\frac 1N<|\xf|\le R^c}\pnn e^{it\Delta_{12}}\phi_{|\yf|\le \frac 12 R^c}(\xf,\yf)|\notag\\
&\lsm N^C\langle N^2t+N|\xf|-N|\yf|\rangle^{-20}\phi_{\frac
1N<|\xf|\le R^c}\phi_{|\yf|\le
\frac 12 R^c}\notag\\
&\lsm N^C\langle N^2t+N|\xf|+N|\yf|\rangle^{-20}\notag\\
&\lsm N^C|N^2t|^{-10}\langle N(\xf-\yf)\rangle^{-10}\lsm
R^{-10}\langle N(\xf-\yf)\rangle^{-10}.\label{three_star}
\end{align}
Then
\begin{equation*}
\eqref{2.1.1}\lsm R^{-5}.
\end{equation*}
follows from these kernel estimate and Young's inequality. This
completes the estimate of \eqref{2.1}.

Now we estimate \eqref{2.2}. By splitting into time pieces and putting
spatial cutoff in front of $F(u)$, our task is reduced to bounding
the following four terms
\begin{align}
\|\midd\phrt\pnp\pns\int_0^{\frac R{100N}}\propgto\frc
F(u(t+\tau)) d\tau\|_{\ltrf}\label{2.2.1}\\
\|\midd\phrt\pnp\pns\int_0^{\frac R{100N}}\propgto\frcs
F(u(t+\tau))d\tau\|_{\ltrf}\label{2.2.2}\\
\|\midd\phrt\pnp\pns\int_{\frac R{100N}}^{\infty} \propgto\bnto
F(u(t+\tau))d\tau\|_{\ltrf}\label{2.2.3}\\
\|\midd\phrt\pnp\pns\int_{\frac R{100N}}^{\infty}\propgto\snto
F(u(t+\tau))d\tau\|_{\ltrf}\label{2.2.4}
\end{align}
We first estimate the tail part \eqref{2.2.2}, \eqref{2.2.4} where
the decay comes from the decay estimate \eqref{one_star},
\eqref{two_star}, \eqref{three_star}.

For \eqref{2.2.2}, we use \eqref{one_star}, Bernstein, the $L^2$
boundedness of $\tilde P_N^{12+}$ to obtain
\begin{align*}
\eqref{2.2.2}& \lsm \left \| \phi_{|x^{34}|>\frac R3} P_N^{12+} P^{34}_{\le N} \tilde P_N
\int_0^{\frac R{100N}} e^{-i\tau \Delta} \phi_{\le \frac 12 R^c} F(u)(t+\tau) d\tau \right\|_{\ltrf} \\
& \lsm \left\| \phi_{>\frac R4} \tilde P_N \int_0^{\frac R{100N}}
e^{-i\tau \Delta} \phi_{\le \frac 12 R^c} F(u)(t+\tau) d\tau \right\|_{\ltrf} \\
& \qquad + (RN)^{-10} \cdot \frac RN \cdot \| \tilde P_N (\phi_{\le \frac 12 R^c} F(u) ) \|_{L^\infty_\tau L_x^2(\srf)} \\
& \lsm R^{-5}.
\end{align*}

For \eqref{2.2.4}, we use \eqref{two_star}, \eqref{three_star},
Bernstein to get
\begin{align*}
\eqref{2.2.4} 
&\lsm \int_{\frac R{100N}}^{\infty}\|\midd\pnp\pns\propgto\snto
F(u(t+\tau))\|_{\ltrf}d\tau\\
&\lsm \int_{\frac R{100N}}^{\infty}\|\midd\pnp
e^{-i\tau\Delta_{12}}\phi_{|\xf|\le \frac{N\tau}{200}}\pns\snto
F(u(t+\tau))\|_{\ltrf} d\tau\\
&\lsm \int_{\frac R{100N}}^\infty \tau^{-10}N^C R^{-10}\biggl\| \|\pns\snto
F(u(t+\tau))\|_{L_{\xft}^1(\R^2)}\biggl\|_{L_{\xst}^2(\R^2)} d\tau\\
&\lsm N^C R^{-10}\int_{\frac
R{100N}}^{\infty}\tau^{-10}\|\|\pns\snto
F(u)\|_{L_{\xst}^2(\R^2)}\|_{L_{\xft}^1(\R^2)} d\tau\\
&\lsm N^C R^{-10}\int _{\frac R{100N}}^{\infty}
\tau^{-10}d\tau\|F(u)\|_{L_{\tau}^{\infty}L_x^1(\rrf)}\\
&\lsm R^{-5}.
\end{align*}

Now we treat the main contribution \eqref{2.2.1}, \eqref{2.2.3} for
which we will use weighted Strichartz and radial Sobolev embedding.
We start with \eqref{2.2.1}. We write
\begin{align}
\eqref{2.2.1}&\lsm \|\pnp\pns \tilde P_N \int_0^{\frac R{100N}}\propgto\frc
F(u\phi_{>R^c/4})(t+\tau) d\tau\|_{\ltrf}\notag\\
&\lsm \|\pntd\int_0^{\frac R{100 N}}\propgto \frc
F(u\phi_{>R^c/4})(t+\tau) d\tau\|_{\ltrf}\notag\\
&\lsm \|\pntd\int_0^{\frac R{100N}}\propgto \frc(P_{>16
N}+P_{\le\frac N{16}})F(u\phi_{>R^c/4})(t+\tau) d\tau\|_{\ltrf}\notag\\
& \quad +\|\pntd \int _0^{\frac R{100 N}}\propgto \frc P_{\frac
N{16}<\cdot\le 16 N} F(u\phi_{>R^c/4})(t+\tau)d\tau\|_{\ltrf}\notag\\
&\le \|\pntd \frc(P_{>16 N}+P_{\le \frac
N{16}})F(u\phi_{>R^c/4})\|_{L^1_{\tau}L_x^2([t,t+\frac R{100
N}]\times\R^4)}\notag\\
&+\|\pntd\int_0^{\frac R{100 N}} \propgto \frc P_{\frac
N{16}<\cdot\le 16 N}
F(u\phi_{>R^c/4})(t+\tau)d\tau\|_{\ltrf}.\label{2.2.1.2}
\end{align}
The first one is mismatched so gives us the bound
\begin{equation*}
(NR^c)^{-10}\frac RN\|\phi_{>R^c/4}
u\|_{L_{\tau}^{\infty}L_x^4(\srf)}^2\lsm R^{-5},
\end{equation*}
due to \eqref{weak_com}. The second term can be estimated as
follows. Since $\frc=\frc(\phi_{|\xf|>R^c/4}+\phi_{|\xs|>R^c/4})$,
we decompose \eqref{2.2.1.2} into four similar terms with the
following being one of the representatives
\begin{equation}
\|\pntd\int_0^{\frac R{100
N}}\propgto\frc\phi_{|\xf|>R^c/4}P_{N/16<\cdot\le 16
N}F(u\phi_{|\xs|>R^c/8})(t+\tau) d\tau\|_{\ltrf}\label{repsent}
\end{equation}
Using weighted Strichartz and Lemma \ref{lem_56est}, we control it
by
\begin{align*}
\eqref{repsent}&\le R^{-\frac c3+}\|P_{N/16<\cdot\le 16
N}F(\phi_{|\xs|>R^c/8}u)\|_{L_{\tau}^{2-}L_x^{\frac 65+}([t,t+\frac
R{100N}]\times \R^4)}\\
&\lsm R^{-\frac c3+}R^{\frac 12+}N^{-\frac 12-}\|P_{ N/16<\cdot \le
16 N} F(\phi_{|\xs|>R^c/8}u)\|_{L_{\tau}^{\infty}L_x^{\frac
65+}([t,t+\frac R{100N}]\times\R^4)}\\
&\lsm R^{-\frac c3+\frac 12+}N^{-\frac 12-}N^{-\frac 59}R^{-\frac
29c}.
\end{align*}
Since $c$ is sufficiently close to $1$ and $R$ sufficiently large,
we can have
\begin{equation*}
 \eqref{repsent}\lsm R^{-\frac 1{50}}
\end{equation*}
and therefore
\begin{align*}
 \eqref{2.2.1} \lsm R^{-\frac 1 {50}}
\end{align*}
which is acceptable.

Now we consider \eqref{2.2.3}. We first bound it as
\begin{align}
\eqref{2.2.3}&\lsm \|\int_{\frac R{100N}}^{\infty}\pntd
\propgto\bnto
F(u(t+\tau))d\tau\|_{\ltrf}\\
&\lsm \|\int_{\frac R{100N}}^\infty \pntd \propgto\bnto(P_{\le
N/16}+P_{>16N})F(u\phi_{>\frac{N\tau}{400}})d\tau\|_{\ltrf}\label{2.2.3.1}\\
&\qquad+\|\int_{\frac R{100N}}^{\infty}\pntd e^{-i\tau\Delta}\bnto
P_{ N/{16}<\cdot\le 16 N}
F(u\phi_{>\frac{N\tau}{400}})(t+\tau)d\tau\|_{\ltrf}.\label{2.2.3.2}
\end{align}
The first term contains a mismatch. Therefore by Minkowski and
mismatch estimate Lemma \ref{L:mismatch_fre}, we have
\begin{align*}
\eqref{2.2.3.1}&\lsm \int_{\frac R{100N}}^{\infty}\|\tilde P_{N}\bnto(P_{\le
N/16}+P_{>16N})F(u\phi_{>\frac{N\tau}{400}})(t+\tau)\|_{\ltrf}d\tau\\
&\lsm\int _{\frac
R{100N}}^{\infty}(N^2\tau)^{-10}\|F(u\phi_{>\frac{N\tau}{400}})(t+\tau)\|_{\ltrf}d\tau\\
&\lsm N^C
R^{-9}\|u\phi_{\gtrsim 1}\|_{L_{\tau}^{\infty}L_x^4([0,\infty)\times\R^4)}^2\\
&\lsm R^{-5}.
\end{align*}

Estimate of \eqref{2.2.3.2} will follow the similar way as for
\eqref{2.2.1.2}. Again we decompose it into four terms with an
example like the following:
\begin{equation}
\|\int_{\frac R{100 N}}^{\infty}\pntd
\propgto\bnto\phi_{|\xf|>\frac{N\tau}{400}}P_{N/16<\cdot\le 16
N}F(\phi_{|\xf|>\frac{N\tau}{800}}u)(t+\tau)d\tau\|_{\ltrf}.\label{repr2}
\end{equation}
Using Weighted Strichartz, Minkowski, Lemma \ref{lem_56est} we bound
it by
\begin{align*}
\eqref{repr2}&\lsm N^{-\frac 13+}\|\tau^{-\frac 13+}P_{N/16<\cdot\le
16 N}F(\phi_{|\xf|>\frac{N\tau}{800}}
u)(t+\tau)\|_{L_{\tau}^{2-}L_x^{\frac 65+}([\frac R{100
N},\infty)\times \R^4)}\\
&\lsm N^C \|\tau^{-\frac 13+}\tau^{-\frac
29}\|_{L_{\tau}^{2-}([\frac
R{100 N},\infty))}\\
&\lsm R^{-\frac 1{100}}.
\end{align*}

Collecting the estimates for \eqref{2.2.1} through \eqref{2.2.4}, we
get the bound for \eqref{2.2}
\begin{equation}
\eqref{2.2}\lsm R^{-\frac 1{100}}.
\end{equation}
Thus \eqref{2} gives the desired control:
\begin{equation*}
\eqref{2}\lsm \|\phi_{>\sqrt R} u_0\|_2+R^{-\frac 1{100}}.
\end{equation*}

 To finish the argument, we
still have to control \eqref{3}. However, modulo some modification,
the estimate of \eqref{3} will be essentially a repetition of that of \eqref{2}.
For this purpose, we only briefly outline the proof.

We first use in-out decomposition to reduce matters to control
\begin{align}
\|\phi_{|\xf|>R^c}\phrt\pnn\pns e^{it\Delta}
u_0\|_{\ltrf}\label{3.1}\\
\| \phi_{|\xf|>R^c}\phrt\pnn\pns\int_0^t
e^{i\tau\Delta}F(u(t-\tau))d\tau\|_{\ltrf}\label{3.2}\\
\|\phi_{|\xf|>R^c}\phrt\pnp\pns\int_0^{\infty} e^{-i\tau\Delta}
F(u(t+\tau))d\tau\|_{\ltrf}.\label{3.3}
\end{align}

For the linear part, the main contribution comes from $\phi_{>R^c/2}
u_0$ which gives us the bound
$$
\|\phi_{>R^c/2}u_0\|_{\ltrf}\le\|\phi_{>\sqrt R}u_0\|_{\ltrf}.
$$
The contribution due to another part $\phi_{\le R^c/2}$ is very
small, say smaller than $R^{-5}$ by using the decay estimate of the
kernel of the operator $\phi_{|\xf|>R^c}\pnn
e^{it\Delta_{12}}\phi_{|\xf|\le R^c/2}$.

The estimate of \eqref{3.2} and \eqref{3.3} will be the same, so we
look at \eqref{3.3}. We further decompose it into different time
pieces:
\begin{align}
\|P_N\phi_{|\xf|>R^c}\phrt\pnp\pns\int_0^{\frac RN}\propgto
F(u(t+\tau))d\tau\|_{\ltrf}\label{3.3.1}\\
\|P_N\phi_{|\xf|>R^c}\phrt\pnp\pns\int_{\frac RN}^{\infty}\propgto
F(u(t+\tau))d\tau\|_{\ltrf}\label{3.3.2}
\end{align}
We first consider \eqref{3.3.1} for which the main contribution
comes from $\frc F(u)$. We can estimate this part by weighted
Strichartz, radial Sobolev embedding and Lemma \ref{lem_56est}. The
contribution due to $\frcs F(u)$ is very small as will be deduced
from the kernel estimate. The same philosophy applies to
\eqref{3.3.2}, for this term the main contribution comes from
$\phi_{>N\tau/2} F(u) $. To conclude, we can handle all the pieces thus
get the bound
\begin{align*}
\eqref{3}\lsm \|\phi_{>\sqrt R}u_0\|_{\ltrf}+R^{-\frac 1{100}}.
\end{align*}

Collecting the estimate of \eqref{1}, \eqref{2} and \eqref{3}, we
finally prove the Proposition \ref{local} in $2+2$ dimensions is
completed.

\end{proof}

\section{Higher dimensional case with admissable symmetry}\label{highd}

In this section, we prove Proposition \ref{fre_decay} and
Proposition \ref{local} in high dimensions $d\ge 5$. The proof in
this case will be an adaptation of that in 2+2 case except for the
fact the numerology is more complicated. For this reason, we will
give a very brief outline of the proof, with emphasis on the
important changes.

To begin with, we explain the notations we will use in this section. We
denote $P^{d_1}_N$ as a Littlewood Paley projection on $d_1$-dimensional
space with the same explanation if we change the number
$d_1$ or replace $N$ by $\le N$, etc. Differential operators
$\nabla^{d_1}$, $\Delta_{d_1}$ should be understood acting on $d_1$
dimensional functions. We also use the convention that the operator
with no subscripts is the one defined on whole $\R^d$.

We first establish the following weighted Strichartz estimate for
solutions which are spherical symmetric only on subspaces of $\R^d$.

\begin{lem}\label{lem_wei_d}
Let the dimension $d\ge 5$. Let
$u_0(x)$, $f(t,x)$ be spherically symmetric on the subspace $\R^{d_1}$.
Then the function $u$ defined by
\begin{align*}
u(t)=e^{i(t-t_0)\Delta}u_0-i\int_{t_0}^t e^{i(t-s)\Delta}f(s)ds
\end{align*}
is also spherically symmetric on subspace $\R^{d_1}$, moreover
\begin{align*}
\||x^{d_1}|^{\frac{2(d_1-1)}{q(d+1-d_1)}}u\|_{L_t^qL_x^{\frac{2q(d+1-d_1)}{q(d+1-d_1)-4}}
(I\times\R^d)}\lsm
\|u_0\|_{\ltrd}+\|f\|_{L_{t,x}^{\frac{2(d+2)}{d+4} }(I\times\R^d)},
\end{align*}
holds $\forall\, q$ such that $q>2$ and $q\ge \frac 4{d+1-d_1}$.
\end{lem}
\begin{proof} With several changes, the proof is in principle the
same as Lemma \ref{lem_wei_4d}. Here we only briefly sketch the
proof.

From standard Strichartz estimate, Christ-Kiselev Lemma and $TT^*$
argument, it is reduced to showing
\begin{align*}
&\left\|\int|x^{d_1}|^{\frac{2(d_1-1)}{q(d+1-d_1)}}e^{i(t-\tau)\Delta}|y^{d_1}|^{\frac{2(d_1-1)}
{q(d+1-d_1)}}f(\tau)d\tau\right\|_{L_t^qL_x^{\frac{2q(d+1-d_1)}{q(d+1-d_1)-4}}(\R\times\R^d)}\\
&\lsm
\|f\|_{L_t^{q'}L_x^{\frac{2q(d+1-d_1)}{q(d+1-d_1)+4}}(\R\times\R^d)}.
\end{align*}
This will be a consequence of the following decay estimate and
Hardy-Littlewood-Sobolev inequality:
\begin{align*}
&\left\||x^{d_1}|^{\frac{2(d_1-1)}{q(d+1-d_1)}}e^{it\Delta}|y^{d_1}|^{\frac{2(d_1-1)}
{q(d+1-d_1)}}f\right\|_{L_x^{\frac{2q(d+1-d_1)}{q(d+1-d_1)-4}}(\R^d)}\\
&\lsm |t|^{-\frac
2q}\|f\|_{L_x^{\frac{2q(d+1-d_1)}{q(d+1-d_1)+4}}(\R^d)}.
\end{align*}
As in the proof of Lemma \ref{lem_wei_4d}, this decay estimate will
follow from the interpolation between the trivial case $q=\infty$
and the pointwise estimate
\begin{align*}
\||x^{d_1}|^{\frac{d_1-1}2}e^{it\Delta}|y^{d_1}|^{\frac{d_1-1}2}f\|_{L_x^\infty(\R^d)}
&\lsm |t|^{-\frac{d+1-d_1}2}\|f\|_{L_x^1(\R^d)},
\end{align*}
where $f$ is spherically symmetric in $(x_1,\cdots,x_{d_1})$
variable. By passing to the radial coordinate, we can write the
kernel as
\begin{align*}
&(|x^{d_1}|^{\frac{d_1-1}2}e^{it\Delta}|y^{d_1}|^{\frac{d_1-1}2})(x,y)\\
&=|x^{d_1}|^{\frac{d_1-1}2}|y^{d_1}|^{\frac{d_1-1}2}(4\pi
it)^{-\frac d2}
e^{\frac{i|x^{d-d_1}-y^{d-d_1}|^2}{4t}}e^{\frac{i(|x^{d_1}|^2+|y^{d_1}|^2)}{4t}}
\int_{S^{d_1-1}} e^{\frac{i|y^{d_1}|x^{d_1} \cdot \omega}{2t}}d\sigma(w).
\end{align*}
Stationary phase or applying the property of Bessel function then
yields the desired estimate.
\end{proof}

As we can see from the calculation of the four dimensional case, the main
difference in high dimensions will arise from the pieces where
$F(u)$ lives on large radii. For these pieces, we will use the above
lemma and radial Sobolev embedding to take advantage of the decay
property of a splitting-spherically symmetric function. It is here
that we need a restriction on the minimal dimension on which the
solution is spherically symmetric. Technically, the restriction stems
from the following

\begin{lem}\label{adm}
Let $d_1=[\frac d2]$ or $\frac d3<d_1\le\frac d2$ for sufficiently
large $d$. Let $u$ be splitting-spherically symmetric with splitting subspaces $\R^{d_1}$ and
$\R^{d-d_1}$ respectively. Then there exist parameters $(\alpha, \beta,p,q)$ such
that
\begin{align*}
\alpha>0,\beta>0, \, 2 \le p, q<\infty,\\
\frac{4\alpha}d+\frac{d_1-1}{d-d_1+1}>\frac 12,\\
\beta+\frac{4\alpha}d+\frac{d_1-1}{d-d_1+1}>1,\\
\frac 1p+\frac 4{dq}=\frac{d-d_1+3}{2(d-d_1+1)}-,\\
\beta+d(\frac 12-\frac 1p)\le 1,\\
d(\frac 12-\frac 1q)-1\le \alpha< d_1(\frac 12-\frac 1q).
\end{align*}
Moreover, for $NL\gtrsim 1$,
\begin{align}
\|\pn
F(\phi_{>NL}u)\|_{L_x^{\frac{2(d-d_1+1)}{d-d_1+3}+}(\R^d)}&\lsm
N^{-\beta}(NL)^{-\frac{4\alpha}d}.\label{eq_244_a}
\end{align}
\end{lem}
\begin{proof}
We first remark the restrictions on the parameters
$(\alpha,\beta,p,q)$ appear naturally when we try to control the LHS
of \eqref{eq_244_a}. For example, the second and the third one
correspond to the power in $L$ and $N$ being negative enough, which,
consequently yield the integrability in $t$ and $1+\eps$ power of
$N$ needed in later computations (cf. the proof of \eqref{eq72tmp}). The last three conditions correspond to the possibility of
using the Sobolev embedding, radial Sobolev embedding Lemma
\ref{L:radial_embed} and the estimate \eqref{weak_com} to establish \eqref{eq_244_a}. In the following, we first assume these
parameters can be taken and quickly prove \eqref{eq_244_a}, then
verify the elementary computations at the very end.

From Bernstein, fractional chain rule Lemma \ref{lem_chain}, we have
\begin{align*}
\text{LHS of\eqref{eq_244_a}}\lsm
N^{-\beta}\|\nabla|^{\beta}(\phi_{>NL}u)\|_{L_x^p}\|\phi_{>NL}u\|_{L_x^q}^{\frac
4d}.
\end{align*}
Now note that by assumption $d_1 \le \frac d2$ and $u$ is spherically symmetric when restricting to
the subspaces $\R^{d_1}$ and $\R^{d-d_1}$ respectively, the cutoff function $\phi_{>NL}$ must have
nontrivial projection to a $d_1$-dimensional subspace on which the restriction of $u$ is spherically symmetric.
By relabelling the coordinates if necessary, we can assume without loss of generality that
\begin{align*}
 \phi_{>NL} = \phi_{>NL} \phi_{|x^{d_1}|>\frac {NL}2}.
\end{align*}
Note that after picking out this $d_1$-dimensional subspace, $u$ is no longer necessarily spherically symmetric
on the remaining $(d-d_1)$-dimensional subspace. Now by Sobolev embedding,  radial Sobolev embedding
Lemma \ref{L:radial_embed} and weak localization of kinetic energy estimate \eqref{weak_com}, we continue to
estimate
\begin{align*}
&\quad \text{LHS of\eqref{eq_244_a}} \\
&\lsm N^{-\beta}\||\nabla|^{\beta+d(\frac 12-\frac
1p)}(\phi_{>NL}u)\|_{L_x^2}
(NL)^{-\frac{4\alpha}d}\||\xd|^{\alpha}\phi_{>NL}u\|_{L_x^q}^{\frac 4d}\\
&\lsm N^{-\beta}(NL)^{-\frac{4\alpha}d}\||\nabla|^{\beta+d(\frac
12-\frac 1p)}(\phi_{>NL}u)\|_{L_x^2} \|\||\nabla|^{-\alpha+d_1(\frac
12-\frac 1q)}(\phi_{>NL}u)\|_{L_{\xd}^2}\|_{L_{x^{d-d_1}}^q}^{\frac 4d}\\
&\lsm N^{-\beta}(NL)^{-\frac{4\alpha}d}\||\nabla|^{\beta+d(\frac
12-\frac 1p)}(\phi_{>NL}u)\|_{L_x^2} \|\||\nabla|^{-\alpha+d_1(\frac
12-\frac 1q)}(\phi_{>NL}u) \|_{L_{x^{d-d_1}}^q}\|_{L_{\xd}^2}^{\frac 4d}\\
&\lsm N^{-\beta}(NL)^{-\frac{4\alpha}d}\||\nabla|^{\beta+d(\frac
12-\frac 1p)}(\phi_{>NL}u)\|_{L_x^2}\||\nabla|^{-\alpha+d(\frac
12-\frac 1q)}(\phi_{>NL}u)\|_{L_x^2}^{\frac 4d}\\
&\lsm N^{-\beta}(NL)^{-\frac {4\alpha}d},
\end{align*}
where the last inequality follows from the fact that the indices of differentiation are
sandwiched between $0$ and $1$, which in turn can be controlled by interpolating the $L_x^2$
mass and weak localization estimate of kinetic energy.
Now we verify the existence of the set of parameters satisfying the aforementioned
conditions. In the case when $d$ is even and $d_1=\frac d2$, we
simply take
\begin{align*}
(\alpha,\beta,p,q)=(\frac{d-2}{2(d+2)}, \frac 12, 2,
\frac{2(d+2)}d+).
\end{align*}
In the case when $d$ is odd and $d_1=\frac{d-1}2$, we take
\begin{align*}
(\alpha,\beta,p,q)=(\frac{3(d-1)}{4(d+3)}-,\frac 12,
2,\frac{2(d+3)}d+).
\end{align*}
The validity of the chosen set follows from direct computation which
we omit.

Now we look at the asymptotic result for sufficiently large $d$. Let
$d_1=\eta d$. We rewrite the conditions equivalently as follows
\begin{align*}
\frac{4\alpha} d>\frac{1-3\eta}{2(1-\eta)}+\frac 1{(1-\eta)^2}\frac
1d+O(\frac 1{d^2}),\\
\beta+\frac{4\alpha} d>\frac 12+\frac{1-3\eta}{2(1-\eta)}+\frac
1{(1-\eta)^2}\frac
1d+O(\frac 1{d^2}),\\
\frac 1p+\frac 4{dq}=\frac 12+\frac 1{1-\eta}\frac 1d-\frac
1{(1-\eta)^2}\frac 1{d^2}+O(\frac 1{d^3}),\\
\beta+d(\frac 12-\frac 1p)\le 1,\\
d(\frac 12-\frac 1q)-1\le \alpha<\eta d(\frac 12-\frac 1q).
\end{align*}
When $\eta>\frac 13$, the first inequality holds automatically for
large $d$. So the genuine restriction comes from the last four. Now
we look at the last inequality, to make it valid, we require $\frac
12-\frac 1q$ is of order $\frac 1d$. Note from the third one,
\begin{align*}
\frac 1q=(\frac 12-\frac 1p)\frac d4+\frac 1{4(1-\eta)}-\frac
1{4d(1-\eta)^2}+O(\frac 1{d^2}),
\end{align*}
thus this forces
\begin{align*}
(\frac 12-\frac 1p)\frac d4+\frac 1{4(1-\eta)}=\frac 12.
\end{align*}
Hence, $p$, $q$ are simultaneously determined:
\begin{align*}
\frac 1q=\frac 12-\frac 1{4d(1-\eta)^2}+O(\frac 1{d^2}),\\
\frac 1p=\frac 12-(\frac 12-\frac 1{4(1-\eta)})\frac 4d.
\end{align*}
 This in turn produce a condition on $\beta$ from the fourth:
\begin{align*}
\beta \le \frac \eta{1-\eta}.
\end{align*}
Since $\eta > \frac 13$, by taking $\beta=\frac 12$, this and the
second hold true. Finally, by choosing
\begin{align*}
\alpha=\eta d(\frac 12-\frac 1q)-=\frac{\eta}{4(1-\eta)^2}+O(\frac
1d).
\end{align*}
The last one holds, hence the five conditions hold for the chosen
parameters.

\end{proof}

Now we have collected enough information to prove Proposition
\ref{fre_decay} and Proposition \ref{local}. We begin with the
frequency decay estimate, that is to show there exists
$\eps=\eps(d)$ such that,
 $\forall N\ge 1$
\begin{align} \label{eq72tmp}
\|\phi_{>10} P_N u(t)\|_{\ltrd}\le \|\pntd
u_0\|_{\ltrd}+N^{-1-\eps}.
\end{align}

\begin{proof}
Note that $d_1\le \frac d2$. If a frequency $|\xi|\sim N$, then
there must be a $d_1$ dimensional vector $|\xi^{d_i}|\sim N$, so
morally we have
\begin{align*}
P_N\sim\pn\tilde P_N^{d_1}.
\end{align*}
Without loss of generality, it is reduced to considering
\begin{align}\label{hi1}
\|\phi_{>10}\pn\pn^{d_1} u(t)\|_{\ltrd}
\end{align}
Due to the strong singularity of $P^{\pm}$ operators at the origin,
in high dimensions, we adopt a slightly different strategy.
Introducing the spatial cutoff in $x^{d_1}$ variable, we use
triangle inequality to bound
\begin{align}
\eqref{hi1}&\le \|\phi_{>10}\phi_{|x^{d_1}|\le \frac 1N}\pn\pd
u(t)\|_{\ltrd}\label{hi2}\\
&+\|\phi_{>10}\phi_{|\xd|>\frac 1N}\pn\pd u(t)\|_{\ltrd}\label{hi3}
\end{align}
We first estimate \eqref{hi2}. Using forward Duhamel \eqref{duhamel}
and split it into time pieces, we write
\begin{align}
\eqref{hi2}&\lsm \|\phi_{>10}\phi_{|\xd|\le \frac 1N}\pn
\pd\int_0^\infty \propgto F(u(t+\tau))d\tau\|_{\ltrd}\notag\\
&\lsm \|\phi_{>10}\phi_{|\xd|\le \frac
1N}\pn\pd\int_0^{N^{-\frac{2d}{d+1}}} \propgto F(u(t+\tau))d\tau\|_{\ltrd}\label{hi4}\\
&+\|\phi_{>10}\phi_{|\xd|\le \frac
1N}\pn\pd\int_{N^{-\frac{2d}{d+1}}}^{\infty}\propgto
F(u(t+\tau))d\tau\|_{\ltrd}\label{hi5}
\end{align}

We first estimate the short time piece. We write
$F(u)=\phi_{>1}F(u)+\phi_{\le 1} F(u)$ and estimate the contribution
from both by Strichartz, kernel estimate and Young's inequality as
\begin{align*}
\eqref{hi4}&\lsm \|\phi_{>10}\phi_{|\xd|\le \frac
1N}\pn\pd\int_0^{N^{-\frac{2d}{d+1}}} \propgto\phi_{>1}
F(u(t+\tau))d\tau\|_{\ltrd}\\
&+\|\phi_{>10}\phi_{|\xd|\le \frac
1N}\pn\pd\int_0^{N^{-\frac{2d}{d+1}}}
\propgto \phi_{\le 1}F(u(t+\tau))d\tau\|_{\ltrd}\\
&\lsm \|\phi_{>1}
F(u(t+\tau))\|_{L_{\tau}^{\frac{2d}{d+4}}L_x^{\frac{2d^2}{d^2+2d-8}}([0,N^{-\frac{2d}{d+1}}]
\times\R^d)}\\
&\quad +N^{-\frac{2d}{d+1}}\sup_{0\le \tau\le
{N^{-\frac{2d}{d+1}}}}\|\phi_{>10}\pn\pd\propgto\phi_{\le
1}F(u(t+\tau))\|_{\ltrd}\\
&\lsm N^{-1-\frac 3{d+1}}\|\phi_{>\frac
12}u\|_{L_t^{\infty}L_x^{\frac{2d}{d-2}}(\srd)}^{\frac
{d+4}{d}}+N^{-10}\|F(u)\|_{L_\tau^\infty
L_x^{\frac{2d}{d+4}}(\srd)}\\
&\lsm N^{-1-\frac
3{d+1}}+N^{-10}\|u\|_{L_t^{\infty}L_x^2(\R\times\R^d)}^{\frac
{d+4}d}\lsm N^{-1-\frac 3{d+1}}.
\end{align*}
In the last line, we have used the Kernel estimate:\footnote{The
kernel estimate Lemma \ref{L:kernel} will also apply when we have
one more projection operator $\pn^{d_1}$.} $\forall
\tau\in[0,N^{-\frac{2d}{d+1}}]$
\begin{align*}
|\phi_{>10}\pn\pn^{d_1}e^{-i\tau\Delta}\phi_{\le 1}(x,y)|\lsm
N^{-10}\langle x-y\rangle^{-10d}
\end{align*}
and Young's inequality.

Now we estimate the contribution from the long time piece. By
writing $F(u)=\phi_{>N\tau/2}F(u)+\phi_{\le N\tau/2}F(u)$, we
further split \eqref{hi5} into two pieces. The estimate when $F(u)$
is supported within a ball is given as follows. Commuting $\pn$
and $\phi_{\le N\tau/2}$ (thus adding a $N^{-2}$ from mismatch
estimate), we get
\begin{align*}
&\qquad\|\phi_{>10}\phi_{|\xd|\le \frac
1N}\pn\pd\int_{N^{-\frac{2d}{d+1}}}^\infty \propgto \phi_{\le
N\tau/2}F(u(t+\tau))d\tau\|_{\ltrd}\\
&\lsm \|\phi_{|\xd|\le \frac
1N}\pn\pd\int_{N^{-\frac{2d}{d+1}}}^\infty \propgto \phi_{\le
N\tau/2}P_{\frac N8<\cdot\le
8N} F(u(t+\tau))d\tau\|_{\ltrd}+N^{-2}\\
&\lsm N^{-2}+\int_{N^{-\frac{2d}{d+1}}}^\infty \|\phi_{|\xd|\le
\frac 1N}\pn\pd e^{-i\tau\Delta_{d_1}} \phi_{|y^{d_1}|\le
N\tau/2}P_{\frac N8<\cdot\le
8N}F(u(t+\tau))\|_{\ltrd}d\tau\\
&\lsm N^{-2}+\int_{N^{-\frac{2d}{d+1}}}^\infty \|\phi_{|\xd|\le
\frac 1N}\pd e^{-i\tau\Delta_{d_1}}\phi_{|y^{d_1}|\le
N\tau/2}P_{\frac N8<\cdot\le 8N}
F(u(t+\tau))\|_{\ltrd}d\tau\\
\end{align*}
Now using the kernel estimate (Lemma \ref{L:kernel})
\begin{align*}
&|(\phi_{|\xd|\le \frac 1N}\pn^{d_1}
e^{-i\tau\Delta_{d_1}}\phi_{|\yd|\le
N\tau/2})(\xd,\yd)|\\
&\lsm N^{d-2m}\tau^{-m}\langle N|\xd-\yd|\rangle^{-m},
\end{align*}
and Young's inequality we continue to control it by
\begin{align*}
& N^{-2}+ N^{d_1}\int_{N^{-\frac{2d}{d+1}}}^\infty |N^2\tau|^{-100d}
\left\|\|P_{\frac N8<\cdot \le
8N}F(u)\|_{L_{x^{d_1}}^{\frac{2d}{d+4}}}\right\|_{L_{x^{d-d_1}}^2}d\tau\\
&\lsm  N^{-2}+ N^{-2}\|F(u)\|_{L_\tau^\infty
L_x^{\frac{2d}{d+4}}(\srd)}\\
&\lsm N^{-2}.
\end{align*}
To estimate the contribution where $F(u)$ is supported outside the
ball, we use Lemma \ref{adm} to obtain
\begin{align*}
&\|\phi_{>10}\phi_{|\xd|\le \frac
1N}\pd\pn\int_{N^{-\frac{2d}{d+1}}}^{\infty} \propgto
\phi_{>N\tau/2}
F(u(t+\tau))d\tau\|_{\ltrd}\\
&\lsm \|\pd\int_{{N^{-\frac{2d}{d+1}}}}^\infty \propgto
\phi_{>N\tau/2}\pntd F(\phi_{>
N\tau/4}u)(t+\tau)d\tau\|_{\ltrd}+N^{-2}\\
&\lsm N^{-2}+\|(N\tau)^{-\frac{d_1-1}{d-d_1+1}+}\pntd
F(\phi_{>N\tau/4}u(t+\tau))\|_{L_\tau^{2-}L_x^{\frac{2(d-d_1+1)}{d-d_1+3}+}([{N^{-\frac{2d}{d+1}}},\infty)\times\R^d)}\\
&\lsm N^{-2}+N^{-\frac{d_1-1}{d-d_1+1}+}N^{-\beta-\frac{4\alpha}d}
\|\tau^{-\frac{d_1-1}{d-d_1+1}+}\tau^{-\frac{4\alpha}
d}\|_{L_{\tau}^{2-}([{N^{-\frac{2d}{d+1}}},\infty))}\\
&\lsm N^{-1-}.
\end{align*}
To conclude, we have
\begin{align*}
\eqref{hi5}\lsm N^{-1-}.
\end{align*}
Collecting the estimates for \eqref{hi4}, \eqref{hi5}, we obtain
\begin{align*}
\eqref{hi2}\lsm N^{-1-}.
\end{align*}
This finishes the estimate of the piece where $|x^{d_1}|\lsm \frac
1N$. To estimate \eqref{hi3} where $|x^{d_1}|$ is large, we will
have to use the in-out decomposition technique as we have done in
the $2+2$ case. As the details have been fully demonstrated in the
2,3 dimensions and $2+2$ case, we will not repeat the argument here.
We also leave the details of the spatial decay estimate Proposition
\ref{local} to interested readers.
\end{proof}

\end{document}